\def\E{\end{document}}
\documentclass[11pt]{article}
\usepackage{amssymb,amsmath}
\usepackage{hyperref}
\usepackage{mathrsfs}
\usepackage{amsfonts}
\usepackage{cases}\usepackage{latexsym}
\usepackage{times}
\usepackage{dsfont}
\usepackage{color}
\usepackage{graphics}\usepackage{amsthm}
\usepackage{ulem}

\newcommand{\m}[1]{\mathbb{#1}}

\begin{document}
\title{\bf Feedback approximate controllability  of  blowup points  for the  heat equation with anti-interference blowup profile
}

\date{}
\author{Ping Lin\footnote{  School of
Mathematics \& Statistics, Northeast Normal University, Changchun
130024,  P. R. China. E-mail address:
linp258@nenu.edu.cn.} \qquad Hatem Zaag\footnote{
Universit\'e Sorbonne Paris Nord,
LAGA, CNRS (UMR 7539), F-93430, Villetaneuse, France. Email:
hatem.zaag@univ-paris13.fr.
}
\\
\\
 }
 \maketitle

 \bibliographystyle{plain}

\newtheorem{theorem}{Theorem}
\newtheorem{definition}{Definition}[section]
\newtheorem{lemma}[definition]{Lemma}
\newtheorem{proposition}[definition]{Proposition}
\newtheorem{corollary}[definition]{Corollary}
\newtheorem{remark}[definition]{Remark}

\renewcommand{\theequation}{\thesection.\arabic{equation}}
\catcode`@=11 \@addtoreset{equation}{section} \catcode`@=12
{ {\bf Abstract.}
This paper is concerned with a  feedback approximate controllability problem of blowup points for the heat equation.  We show that the system   is  approximately controllable
 for  blowup points with feedback controls and the feedback operator is bounded at any time before  blowup. It is also proved that the blowup profile for feedback  controllability of
  blowup points   is stable with respect to initial data. That is,
  suppose that the initial data has a very small perturbation,
  the blowup profiles  also have tiny changes.
 More precisely, it just undergoes a tiny translation in space and
    time.
  This means that our feedback strategy is anti-interference.

\bigskip



\section{Introduction}

Control issues are almost ubiquitous. In a control problem, there is a controlled object and one (or more) control actions. According to the basic action principles of control, it can be divided into two types: open-loop control and closed-loop control (also known as feedback control). It is called open-loop control to design a series of instructions to control the object according to control requirements and external interference. The characteristic of this scheme is to control based solely on the expected results, regardless of the consequences of the control effect. Another type of method that determines the control effect based on changes in the controlled object is called closed-loop control, or feedback control. Whether it is to suppress the influence of external disturbances or reduce the influence of internal parameter changes, feedback control is far superior to open-loop control. It is worth pointing out that feedback is a special concept in control theory.

It is well known that blowup widely
 exists in nature and in practical applications. Some
nonlinear evolution differential  equations can be used to describe this phenomenon (see for instance \cite{Abdelhedi}, \cite{Biumj90}-\cite{Bricmont}, \cite{Stevens 1}-\cite{Stevens 3}, \cite{Filippas}, \cite{Filippas-1}, \cite{GKcpam85}-\cite{Guo-Hu},
\cite{Herrero-1}-\cite{Kang1}, \cite{Khenissy}, \cite{Zaag2}-\cite{Nguyen1}, \cite{TZtams19}, \cite{J.J.L.}, \cite{zhangzc1},
\cite{zhangzc2}).
 Roughly speaking, blowup is a conception  which means that a solution is
unbounded in finite time. In certain cases,
the blowup of a solution is desired.
For instance,
the dramatic increase in temperature  leads to the ignition of a chemical
reaction.
 However,  solutions to linear partial differential  equations  without control
generally globally exist. It is naturally interesting to find  feedback controls
 to these equations such that the corresponding solutions blow up in
 finite time at some given place.

\medskip

Let $\Omega$ be a bounded domain of $\mathbb{R}^n$ with smooth boundary $\partial \Omega$.
Consider the following   control system,
\begin{eqnarray}\label{xe1.1}
\left\{\begin{array}{ll} y_t-
\Delta y=\mathds{1}
_\omega u,&x\in \Omega,\ t>0,\\
y=0,& x\in \partial\Omega,\  t>0,\\
\displaystyle  y(x,0)=y_0(x),&
x\in \Omega.
\end{array}\right.
\end{eqnarray}
Here the control $u$ acts on a nonempty and open subset
$\omega\subset\Omega$, and  $\mathds{1}_\omega$ is the characteristic
function of the set $\omega$.

\medskip

In the absence of control, we know that the solution to system (\ref{xe1.1})
 globally exists. The following definition refers to blowup time and blowup point of a solution  to  (\ref{xe1.1}) with a feedback  control.
\begin{definition}[Notions of blowup time and point]\label{defbup}
Let $T>0$ and $a\in \Omega$. We say that $T$ is the blowup time and $a$ is  a blowup point of the  solution $y$ to system (\ref{xe1.1}) with  a feedback control
 $u$,  if $y(x,t)$ exists for all $(x,t)\in \Omega \times [0,T)$  with
 $|y(a_j,t_j)|\to +\infty$, as $j\rightarrow\infty$, for some sequences $a_j\to a$ and $t_j \to T$.
\end{definition}

In \cite{LinZaag}, Lin and Zaag proposed  the following feedback exact controllability problem for blowup points.

 \medskip

 \textbf{Problem ($P$)}. \textit{ Given   $a\in \Omega$, $T>0$ and initial data $y_0\in H_0^1(\Omega)$,   can  we find a  feedback control $u$
  such that $T$ is the blowup time of the solution $y$ to
  (\ref{xe1.1}), with $a$ being the unique blowup point of $y$?}

 \medskip

Lin and Zaag \cite{LinZaag} studied Problem ($P$) and obtained the
 following  result (see Theorem 2  in \cite{LinZaag}).

 \medskip

\begin{proposition} \label{prpo1.1}     For any $a\in \omega$ and any $T>0$, there exist $T_1\in (0, T/2)$ and $\widetilde{y}_0\in C_0^\infty(\omega)$ such that for any $y_0\in H_0^1(\Omega)$,  the  solution
$y$ to  (\ref{xe1.1}) with the following feedback control $$u(x,t):=\left\{\begin{array}{ll}-\mathds{1}_\omega^*P(t)({y}(t)-\widetilde{y}_0\big)(x)-\Delta \widetilde{y}_0(x),\ &(x,t)\in \Omega\times  [0,T-T_1), \\
|y|^{p-1}y(x,t), \ &(x,t)\in \Omega\times [T-T_1,T),
\end{array}\right.$$
exists on $[0,T)$, $T$ is the blowup time of $y$ and $y$
has a unique blowup point $a$, where $p>1$.
Here, $P\in C_S\big([0,T-T_1); \Sigma^+(H)\big)$ is the unique mild solution to the following Riccati system,
 \begin{eqnarray}\label{xe1.2}
\left\{\begin{array}{ll} P'(t)+\Delta^*P(t)+P(t)\Delta-P(t)\mathds{1}_\omega\mathds{1}^*_\omega P(t)=0, \ t\in[0,T-T_1)
,\\
\lim\limits_{(s,z)\rightarrow(T-T_1,z_0)}\langle P(s)z, z \rangle=+\infty,\ \mbox{for each} \ z_0\in H\
\mbox{and}\ z_0\neq 0.
\end{array}\right.
\end{eqnarray}

Moreover, for all $R>0$,
 \begin{eqnarray}\nonumber
\sup\limits_{\big\{|x-a|\leq R
\sqrt{(T-t)|\log(T-t)|}\big\}}\left|(T-t)^{\frac{1}{p-1}}y(x,t)-f\left(\frac{x-a}{\sqrt{(T-t)|\log(T-t)|}}\right)\right|\rightarrow 0\ \mbox{as} \ t\rightarrow T.
\end{eqnarray}
Here  $H=L^2(\Omega)$,  $\Sigma^+(H)$ denotes the Banach space of all symmetric and positive operators acting in $H$ and  $C_S([0,T-T_1);\Sigma^+(H))$ denotes the set of all mappings $S:[0,T-T_1)\rightarrow \Sigma^+(H)$ such that $S(\cdot
)z_0$ is continuous on $[0,T-T_1)$ for each $z_0\in H$. $\mathds{1}^*_\omega$ is the adjoint operator of $\mathds{1}_\omega$,
$\Delta^*$  is the adjoint operator of $\Delta$, and
\begin{eqnarray}\label{fnew}
f(\eta)=\left(p-1+\frac{(p-1)^2}{4p}|\eta|^2\right)^{-\frac{1}{p-1}}, \ \forall\ \eta\in \mathbb{R}.
\end{eqnarray}
\end{proposition}

Let $a\in \omega$, $T>0$ and $y_0\in H_0^1(\Omega)$. The proof of Proposition \ref{prpo1.1} (see the proof of Theorem 2, \cite{LinZaag})  relies on 2 arguments:

\medskip

\textit{A Nonlinear Construction Argument:} We constructed a
  blowup solution to some parabolic equation with a localized
  nonlinearity.
 More precisely, we found  a special initial data $\widetilde{y}_0$ supported in $\omega$ such
that the corresponding  solution to system (\ref{xe1.1}) (where initial time is replaced by $T-T_1$ and initial data is replaced by $\widetilde{y}_0$)  with the feedback control $u(x,t)=|y|^{p-1}y(x,t), \ (x,t)\in \Omega\times [T-T_1,T)$ ($p>1$), blows up in time $T$ with
$T_1\in(0,T/2)$ at unique blowup point $a$, with the prescribed blowup
profile;

\medskip

\textit{A Linear Control Argument:} We connected
  the given initial state $y_0$ to the initial data of the
  Construction Step. That is,
we proved that the  solution $y$ to (\ref{xe1.1}) with the feedback control $u(x,t)=-\mathds{1}_\omega^*P(t)({y}(t)-\widetilde{y}_0\big)(x)-\Delta \widetilde{y}_0(x),\ (x,t)\in \Omega\times  [0,T-T_1)$  satisfies $y(T-T_1)=\widetilde{y}_0$.
 Combining these two steps, we could prove Proposition \ref{prpo1.1}.

 \medskip

In order to state our problem studied in this paper, we need the following definition.

\begin{definition} [Notion of feedback approximate controllability of  blowup points] Given  $a\in \Omega$, $T>0$ and $y_0\in H_0^1(\Omega)$, we say that system (\ref{xe1.1}) is  approximately controllable  for  blowup points with feedback controls if
 for each $\varepsilon>0$, we can find a  feedback control $u$
such that the corresponding solution $y$ to
  (\ref{xe1.1}) blows up in time $T^*$ and
has a unique blowup point $a^*$ with $|T^*-T|<\varepsilon$ and $|a^*-a|<\varepsilon$.\end{definition}

In this paper, we would like to obtain a result on feedback approximate  controllability of  blowup points  for system (\ref{xe1.1}) with an anti-interference blowup profile. A natural question to ask is that the
 feedback exact controllability of blowup points for system (\ref{xe1.1}) has been well solved in \cite{LinZaag}, why the approximate one is to be studied?
The following two reasons give the answer to this question.

 \medskip

(i) As we see in Proposition \ref{prpo1.1}, the feedback operator $P(\cdot)$  which makes the solution $y$ of system (\ref{xe1.1}) to reach $\widetilde{y}_0$ in time $T-T_1$ is the solution to  Riccati system (\ref{xe1.2}). Since $P(\cdot)$ satisfies the second equality of (\ref{xe1.2}), it is indeed an unbounded operator. This is because the feedback null controllability
 result for linear heat equation obtained in \cite{Sirbu} was used to prove this property, and one can not make system (\ref{xe1.1}) be feedback null controllable by a bounded linear feedback operator due to the backward uniqueness of the solution to linear heat equation.

  \medskip

From point of view of application, it is better if the feedback operator could be bounded in any time before blowup. Since $P\in C_S\big([0,T-T_1); \Sigma^+(H)\big)$, we have that  for any $T'\in(0,T-T_1)$, $P\in C_S\big([0,T']; \Sigma^+(H)\big)$, i.e., $P(\cdot)$ is a bounded operator on $[0,T']$. To achieve this aim,
 we could first choose a time $T'$ close enough to $T-T_1$ and act the
 feedback control
 $u(\cdot,t)=-\mathds{1}_\omega^*P(t)({y}(t)-\widetilde{y}_0\big)(\cdot)-\Delta
 \widetilde{y}_0(\cdot)$ on $[0,T']$. To avoid errors, $T'$ should be
 allowed to be
 taken
 from a sufficiently small interval very close to $T-T_1$. On the other hand, as we mentioned above, the corresponding  solution to system (\ref{xe1.1}) (where initial time is replaced by $T-T_1$ and initial data is replaced by $\widetilde{y}_0$)  with the feedback control $u(x,t)=|y|^{p-1}y(x,t), \ (x,t)\in \Omega\times [T-T_1,T)$ ($p>1$), blows up in time $T$ at unique blowup point $a$.  Then, the following problem should be answered: For any $\varepsilon>0$, when the feedback control $u(\cdot,t)=-\mathds{1}_\omega^*P(t)({y}(t)-\widetilde{y}_0\big)(\cdot)-\Delta \widetilde{y}_0(\cdot)$ acts on time $[0,T']$  with $T'\in(0,T-T_1)$ in
a sufficiently small interval very close to $T-T_1$,  could $y(T')$  be sufficiently approximate to $\widetilde{y}_0$ such that when we further take the feedback control $|y|^{p-1}y$ as $t> T'$,  the solution $y$ to
  (\ref{xe1.1}) could blow up in time $T^*$ and
has  unique blowup point $a^*$, with $|T^*-T|<\varepsilon$ and $|a^*-a|<\varepsilon$? For this purpose, we naturally have  to consider  feedback approximate  controllability of  blowup points for system (\ref{xe1.1}).

  \medskip

(ii) It is important to show that whether the blowup profile for feedback  controllability of  blowup points   is stable with respect to initial data. That is, suppose that the initial data has a very small perturbation, will the blowup profiles  also have  tiny changes? The positive answer to this question would mean that our feedback strategy is anti-interference. This is very important in practical application. To this end, studying the case of approximate controllability is more reasonable.

   \medskip

 The following theorem is our main result, which could give a positive answer to the above questions proposed in (i) and (ii). Theorem \ref{Main Tho} could be  viewed as a result of approximate controllability with anti-interference blowup profile,
and the feedback operator stated in it is bounded at any time before  blowup.

\medskip

\begin{theorem} \label{Main Tho}     Given  $a\in \omega$ and $T>0$. Let $T_1\in (0, T/2)$ and $\widetilde{y}_0\in C_0^\infty(\omega)$ be given in Proposition \ref{prpo1.1}. Let $y_0\in H_0^1(\Omega)$. Then for any $\varepsilon>0$, there exists $\widehat{\varepsilon}_1\in (0, \min({\varepsilon}/{4}, {T_1}/{4}))$ such that we
can find ${\delta}\in(0,\min\{\varepsilon/4,T-T_1\})$, $\delta^*>0$ with
$({\delta}-\delta^*,{\delta}+\delta^*)\subset(0,\min\{\varepsilon/4,T-T_1\})$  and $\varepsilon^*>0$  such that for any $\widehat{\varepsilon}\in ({\delta}-\delta^*,{\delta}+\delta^*)$ and any initial data $y_0^*\in H_0^1(\Omega)$ satisfying $\|y_0^*-y_0\|_{H_0^1(\Omega)}<\varepsilon^*$, the  solution
$y$ to  (\ref{xe1.1}) with initial data $y_0^*$ and with the following feedback control
$$u(x,t):=\left\{\begin{array}{ll}-\mathds{1}_\omega^*P(t)({y}(t)-\widetilde{y}_0)(x)-\Delta \widetilde{y}_0(x),\ (x,t)\in \Omega\times  [0,T-T_1-\widehat{\varepsilon}),
\\
0,\ \ \ \ \ \ \ \ \ \ \ \ \ \ \ \ \ \ \ \ \ \ \ \ \ \ \ \ \ \ \ \ \ \ \ \ \ \ \ \ \ \ \ \ \ \ \ \ \ \ \ \  \   (x,t)\in \Omega\times  [T-T_1-\widehat{\varepsilon},T-T_1-\widehat{\varepsilon}+\widehat{\varepsilon}_1),\\
|y|^{p-1}y(x,t), \ \ \ \ \ \ \ \ \ \ \ \ \ \ \ \ \ \ \ \ \ \ \ \ \ \ \ \ \ \ \ \ \ \ (x,t)\in \Omega\times [T-T_1-\widehat{\varepsilon}+\widehat{\varepsilon}_1

,T^*)
\end{array}\right.$$
 blows up in time $T^*$ and
has  unique blowup point $a^*$, with $|T^*-T|<\varepsilon$ and $|a^*-a|<\varepsilon$,  where $p>1$.
Here, $P\in C_S\big([0,T-T_1); \Sigma^+(H)\big)$ is the unique mild solution to  Riccati system (\ref{xe1.2}).

\medskip

Moreover, for all $R>0$,
 \begin{eqnarray}\label{app1}
\sup\limits_{\big\{|x-a^*|\leq R
\sqrt{(T^*-t)|\log(T^*-t)|}\big\}}\left|(T^*-t)^{\frac{1}{p-1}}y(x,t)-f\left(\frac{x-a^*}{\sqrt{(T^*-t)
|\log(T^*-t)|}}\right)\right|\rightarrow 0
\end{eqnarray}
as $t\rightarrow T^*$, where
$f$ is defined in (\ref{fnew}).
\end{theorem}

\medskip

From the perspective of the purpose of control, the target of feedback controllability problem for blowup points
is ``infinity",  which is outside the state space
of the  solutions, while the targets of classical controllability problems (see for
instance \cite{Petisco}, \cite{zuazuax1}, \cite{F1}-\cite{Doubova},
\cite{zuazuax2}-\cite{Fernandez-Cara}, \cite{zuazuax3}, \cite{Kassab}, \cite{Labb}, \cite{Le}, \cite{zuazuax4}-\cite{Sirbu}, \cite{zhangx1}, \cite{zhangx2}, \cite{Zuazuab1}, \cite{Zuazuab2})
 are within the state space.
 On the other hand, feedback control may be considered as
 a closed-loop system which
 plays an effective role of control.

\medskip


In the past twenty years, several references studied the controllability of equations in the context of
blowup.
Doubova, Fern\'{a}ndez-Cara, Gonz\'{a}lez-Burgos and Zuazua
\cite{Doubova}, together with Fern\'{a}ndez-Cara and Zuazua \cite{Fernandez-Cara} considered the controllability of weakly blowing up semilinear parabolic
equations with open-loop controls. They showed that the considered systems are null and approximately controllable at any time. One can also see the recent works by B\'{a}rcena-Petisco \cite{Petisco}, Kassab \cite{Kassab} and
Le Balc'h \cite{Le}.
In those papers, blowup occurs if no control is applied. Meanwhile,
using appropriate controls, the solution can be steered to zero or as
near as possible to some given target in the state space at any given
time. In other words, the aim of those references was to prevent blowup
thanks to controls, which is the opposite to our purpose of feedback controllability
 problem for blowup points, where we aim at making the solution blow up.


\medskip

Recently, several references concerned the study of
feedback blowup controllability.  In addition to \cite{LinZaag},  Lin \cite{Lin} considered the blowup
controllability of the  heat equation with feedback controls. As a
matter of fact, Lin proved
that for any initial data in $H^{1}_{0}(\Omega)$  and for any time $T>0,$ there exist a number $p$ with $p\in (1,\infty)$ and a feedback control acting on an internal
subset of the space domain such that the $L^{p+1}$ norm of the
solution to (\ref{xe1.1}) blows up in $T.$ Later, Han, Liu and Lin
\cite{Lin2}  studied  the ordinary differential system $y'(t)=Ay(t)+Bu(t)$ in the case where $(A,B)$ is null controllable,
 $A$ and $B$ being time independent matrices. They first concluded that this ode system is exactly blowup controllable with feedback control. Furthermore, they showed that it is also
 approximately blowup controllable with the feedback operator bounded at any time before blowup.

\medskip

In some particular situations, one may need a
 sharper target for blowup. This is the case for instance in the
 mining process, where one may need blowup to occur at a specific
 space location in some given time.  The problem of blowup point controllability considered in this paper and \cite{LinZaag}
 could meet this need.

\medskip

The idea of the proof of our  main result in Theorem \ref{Main Tho} is as follows.

\medskip

Consider the following auxiliary system,
\begin{align}\label{profile1}
\left\{\begin{array}{ll} y_t-\Delta
y=\mathds{1}_\omega|y|^{p-1}y, \ \ \ \  x\in \Omega,\ t>0,\\
y=0,\ \ \ \ \ \ \ \ \ \ \ \ \ \ \ \ \ \ \
\ \ \ \ \ \ \ \ \ \ \  x\in \partial\Omega, \  t>0,\\
\displaystyle  y(x,0)=y_0(x),\ \ \ \ \ \ \ \ \ \ \ \ \ \  x\in
\Omega,
\end{array}\right.
\end{align}
where $p>1$ is arbitrary but fixed.

\medskip

First, we will remake the proof of our result on constructing a blowup solution for system \eqref{profile1}  with a prescribed profile,
i.e., Theorem 1 in our paper \cite{LinZaag}, which
is the following proposition. This plays an important role in establishing our main result.

\begin{proposition}\label{lemma1.3} For any $a\in \omega$, there exists
  $T_0>0$ such that for any $T\in (0,T_0)$, there exists some initial
  data $y_0\in C_0^\infty(\omega)$ such that the corresponding solution
$y$  to  (\ref{profile1}) exists on $[0,T)$, $T$ is the blowup time of $y$ and $y$
has a  unique blowup point $a$. Moreover, for all $R>0$,
 \begin{eqnarray}\label{1.3*}
\sup\limits_{\big\{|x-a|\leq R
\sqrt{(T-t)|\log(T-t)|}\big\}}\left|(T-t)^{\frac{1}{p-1}}y(x,t)-f\left(\frac{x-a}{\sqrt{(T-t)|\log(T-t)|}}\right)\right|\rightarrow 0
\end{eqnarray}
as $t\rightarrow T$, where
$f$ is defined in (\ref{fnew}).
\end{proposition}

Just as we did in the proof of Theorem 1 in \cite{LinZaag}, thanks to some cut-off function around
$a\in \omega$, we recover the $\mathbb{R}^n$ case, however,
with cut-off terms:
\begin{align}
 y_t-\Delta y=|y|^{p-1}y + \mbox{ cut-off terms, with } (x,t)\in \m R^n\times[0,T).\nonumber
  \end{align}
Then, we will use similar techniques to Mahmoudi, Nouaili and  Zaag \cite{Zaag2}
to prove that for any $a\in \omega$, there exists $T_0>0$ such that for any $T\in (0,T_0)$, one can find special initial data  such
that the corresponding  solution   blows up in time $T$ at unique blowup
point $a$, and with the prescribed blowup profile \eqref{1.3*}. The
control of the solution near that profile will be different, according
to whether we are in the ``blowup region'' or the ``regular region'':


\medskip

$\bullet$ In the blowup region, near the blowup point, we reduce the question to a finite-dimensional problem. Similarity variables will be used to control the solution near the profile.

\medskip

$\bullet$ In the regular region, far from the blowup point, we directly use standard parabolic estimates.

\medskip

Note that we proceed by contradiction to solve the finite-dimensional problem
and complete the proof of that step thanks to a topological argument.

\medskip

However, in subsection 2.1, we will reformulate the problem and give a modified definition of shrinking set and some preliminary lemmas, which are small adaptations of those used in the proof of
Theorem 1 in  \cite{LinZaag}. The order in which they appear in the proof will also change. The adaptations and the change of order would make
these lemmas  not only apply to construct a blowup solution corresponding to special initial data we chose in \cite{LinZaag}, but
also to initial data in a small neighborhood of it. This will play an important role in proving our main result of this paper. By these lemmas, we will follow the same techniques used in \cite{LinZaag} to prove
Proposition \ref{lemma1.3} (see subsection 2.2).

\medskip

Then, we proceed prove our main result in the following two steps.

\medskip

In the first step, we will prove a stability result of the blowup profile for  system \eqref{profile1}.

\medskip

The following theorem  gives a stability result of the blowup profile with respect to initial data for system (\ref{profile1}), which will be used to prove Theorem \ref{Main Tho}.

\begin{theorem}\label{Pros1.2} Let   $\widehat{y}$ be the constructed solution in Proposition \ref{lemma1.3}. $\widehat{y}_0$ is its initial data in $C_0^\infty(\omega)$, $\widehat{T}$ is its blowup time and $\widehat{a}\in \omega$
is its unique blowup point. Then for any $\varepsilon>0$, there exists  $\varepsilon_1>0$,   such that for  any initial data $y_0\in W_0^{1,q}(\Omega)\bigcap W^{2,q}(\Omega)$  with $q>n+2$ satisfying $\|y_0-\widehat{y}_0\|_{W_0^{1,q}(\Omega)\bigcap W^{2,q}(\Omega)}<\varepsilon_1$, the corresponding solution
$y$ to system (\ref{profile1}) blows up in finite time $T(y_0)$ and at a unique blowup point $a(y_0)$ with
$$|T(y_0)-\widehat{T}|<\varepsilon,\ |a(y_0)-\widehat{a}|<\varepsilon.$$

 Moreover, for all $R>0$,
 \begin{eqnarray}\label{1.3x}
\sup\limits_{\big\{|x-a(y_0)|\leq R
\sqrt{(T(y_0)-t)|\log(T(y_0)-t)|}\big\}}&\Big|(T(y_0)-t)^{\frac{1}{p-1}}{y}(x,t)\nonumber\\&-f\Big(\frac{x-a(y_0)}{\sqrt{(T(y_0)-t)|
\log(T(y_0)-t)|}}\Big)\Big|\rightarrow 0,
\end{eqnarray}
as $t\rightarrow T(y_0)$, where
$f$ is defined in (\ref{fnew}).
\end{theorem}

Theorem \ref{Pros1.2} shows the stability of the constructed solution in Proposition \ref{lemma1.3}, with respect to perturbations in initial data.
By the similar techniques used in \cite{LinZaag} and \cite{TZtams19}, it could be proved by using geometrical interpretation of the finite-dimensional parameters, reduction to a finite-dimensional problem and continuity (see details in Section 3).

\medskip


 In the second step, let  $a\in \omega$ and $T>0$.   Let $T_1\in (0, T/2)$ and $\widetilde{y}_0\in C_0^\infty(\omega)$ be given in Proposition \ref{prpo1.1}. Let $y_0\in H_0^1(\Omega)$. Recall that in Linear Control Argument, we connect the given initial state $y_0$ to  initial data ${\widetilde{y}}_0$
of the Construction Step in nonlinear argument at time $T-T_1$ by the feedback control $u(\cdot,t)=-\mathds{1}_\omega^*P(t)({y}(t)-\widetilde{y}_0\big)(\cdot)-\Delta \widetilde{y}_0(\cdot), \ t\in [0,T-T_1]$.
However, by the second equality of (\ref{xe1.2}), $P(\cdot)$ is indeed an unbounded operator.

\medskip

Since $P\in C_S([0,T-T_1); \Sigma^+(H))$, we could show
  that for initial data  $y_0^*\in H_0^1(\Omega)$ with $\|y_0^*-y_0\|_{H_0^1(\Omega)}$ small enough, the corresponding solution $y$ with the feedback control $u(\cdot,t)=-\mathds{1}_\omega^*P(t)({y}(t)-\widetilde{y}_0\big)(\cdot)-\Delta \widetilde{y}_0(\cdot), \ t\in [0,T']$ satisfies  $y(T')$ could be close enough to $\widetilde{y}_0$ in ${H_0^1(\Omega)}$, where $T'\in (0,T-T_1)$ could be any time in a sufficiently small interval  very close to  $T-T_1$.

\medskip

Furthermore, we could take 0 control to enhance the regularity of the solution $y$ in the interval of $(T',T'+\widehat{\varepsilon}_1]$,
where we take  $\widehat{\varepsilon}_1$  small enough to ensure that $y(T'+\widehat{\varepsilon}_1)$ could be close enough to $\widetilde{y}_0$ in $W_0^{1,q}(\Omega)\bigcap W^{2,q}(\Omega)$ with $q>n+2$.

\medskip

Finally, by Theorem \ref{Pros1.2},  we could complete the proof of Theorem \ref{Main Tho} (see details in Section 4).

 The rest of this paper is structured as follows. In Section
2, we will reconstruct a blowup solution with a prescribed profile. Section 3 will prove  a stability result of the blowup profile with respect to initial data for system (\ref{profile1}), i.e., Theorem \ref{Pros1.2}.
In section 4, we will give the proof of our main result.

\section{Reconstructing a blowup solution with a prescribed profile}

In this section, we will remake the proof of our result on constructing a blowup solution with a prescribed profile, i.e., Theorem 1 in our paper \cite{LinZaag}, which
is Proposition \ref{lemma1.3} in this paper. This plays an important role in establishing our main result. Some arguments will repeat those written in \cite{LinZaag}
 for the convenience of the readers and for the completeness of this paper.

\medskip

Throughout this paper, we will give the proof
when $n=1$, for simplicity, bearing in mind that the proof for $n\geq2$ is the same, with natural adaptations, as one can see from  the paper \cite{Nguyen2} done for the standard heat equation.

\subsection{Formulation of the problem and preliminary lemmas}

Consider some $a\in \omega$ and $T>0$. The aim of Proposition \ref{lemma1.3} is to show that for $T$
small enough, there is some $y_0\in C^\infty_0(\omega)$ such that
equation \eqref{profile1} fitted with initial data $y_0$ has a solution $y(x,t)$
defined for all $(x,t)\in \Omega \times [0,T)$ such that $y$ blows up
at time $T$ only at the point $a$, with the behavior described in
\eqref{1.3*}.

\medskip

Let $\chi_0\in C_0^\infty(\mathbb{R},[0,1])$ with
\begin{align}\label{1.7}\chi_0(\xi)=\left\{\begin{array}{ll} 1, \ |\xi|\leq 1,\\0, \ |\xi|\geq 2.
                                           \end{array}\right.\end{align}

In the regular region, we define $\overline{y}$ by
\begin{align}\label{ybar}\overline{y}(x,t)=& y(x,t)\overline{\chi}(x),\ x\in \Omega,\ t\geq 0,\end{align}
where for any $\xi\in \Omega$,
$\overline{\chi}(\xi)=1-\chi_0\big(\frac{4(\xi-a)}{\varepsilon_0}\big)$,
for some $\varepsilon_0>0$ will be fixed sufficiently small  later.
Then,  $\overline{y}$ satisfies the following equation:
\begin{align}\label{1.8}
\partial_t \overline{y}=\partial_{xx}\overline{y}+\mathds{1}_\omega|y|^{p-1}\overline{y}-2\overline{\chi}'\partial_x y-\overline{\chi}{'}{'}y.
\end{align}

In the blowup region, we make the following similarity variables transformation:
\begin{align}\label{W1}
W(z,s)=(T-t)^{1/(p-1)} y(x,t)=e^{-s/(p-1)}y(a+e^{-s/2}z,T-e^{-s})
\end{align}
with
\begin{align}\label{W1j}x-a=(T-t)^{1/2}z,\ T-t=e^{-s}. \end{align}
This transformation was first introduced by Hocking  and Stewartson \cite{Hocking} in the context of the Complex
      Ginzburg-Landau (CGL), which includes the semilinear heat
      equation as a special case. Later, it was used by Giga and Kohn in
      \cite{GKcpam85} to derive the well celebrated blowup results
      for the latter equation. Following \eqref{W1}, we see that
       $W(z,s)$ satisfies the following
equation in $(\Omega-a) e^{s/2}\times[s_0, \infty)$, with $s_0=-\log T$:
\begin{align}\label{1.11}
\partial_s W=\partial_z^2 W-\frac{1}{2}z\partial_z W-\frac{1}{p-1}W
+|W|^{p-1}W+(\mathds{1}_\omega-1)|W|^{p-1}W.
\end{align}

For all $z\in \mathbb{R}$, we define
\begin{align}\label{w}w(z,s)=\left\{\begin{array}{ll} W(z,s)\chi(z,s), \ z\in (\Omega-a) e^{s/2},\\0, \ \ \ \ \ \ \ \ \ \ \ \ \ \ \ \ \ \ \ \ \  \ \ z\in \mathbb{R}\backslash
(\Omega-a) e^{s/2},
\end{array}\right.\end{align}
with
\begin{align}\label{chi}\chi(z,s)=\chi_0\Big(\frac{ze^{-s/2}}{\varepsilon_0}\Big),\end{align}
where
\begin{equation}\label{inside}
(a-4\varepsilon_0,a+4\varepsilon_0)\subset \omega,
\end{equation}
  and $\varepsilon_0>0$ will be fixed small enough later.
%

\medskip

We see that whenever
%
initial data $y_0\in C_0^\infty(\omega)$  and $y$ exists on $[0,T)$,
it follows by the internal regularity of the heat
   equation that $w$ is in the space
   $C^{2,1}(\mathbb{R}\times[s_0,+\infty))$
   with $s_0=-\log T$.

   \medskip

It can be easily obtained that
\begin{align}\label{w1jia}\partial_s w=\partial_z^2 w-\frac{1}{2}z\partial_z w-\frac{1}{p-1}w
+|w|^{p-1}w+N(z,s),\ z\in \mathbb{R},\ s\geq-\log T,\end{align}
where
\begin{align}N(z,s)=\left\{\begin{array}{ll} W\partial_s\chi-2\partial_zW\partial_z\chi-W\partial_z^2\chi+\frac{1}{2}zW\partial_z\chi\\+|W|^{p-1}W(\chi-\chi^p),\   \ \ \ \ z\in (\Omega-a) e^{s/2},\ \ s\geq-\log T,\\0,\ \ \ \ \ \ \ \ \ \ \ \ \ \ \ \ \ \ \ \ \ \ \ \ \ \ \ \ \ \ \ \ \ \ \ \ \ \  z\in \mathbb{R}\setminus (\Omega-a)e^{s/2}, \ s\geq-\log T .
                           \end{array}\right.\end{align}


Let \begin{align}\label{w1}w=\varphi+q,\end{align}
where  \begin{align}\label{w2}\varphi=f(\frac{z}{\sqrt{s}})+\frac{\kappa}{2ps}\end{align}
with $\kappa=(p-1)^{-\frac{1}{p-1}}$ and $f$ is defined in (\ref{fnew}). Then,
$q$ is a solution to the following equation in $\mathbb{R}\times[s_0, \infty)$,
\begin{align}\label{2.20}
\partial_s q=(\mathcal{L}+V)q+B(z,s)+R(z,s)+N(z,s),
\end{align}
where
\begin{align}\label{LD}&
\mathcal{L}=\partial_z^2-\frac{1}{2}z\partial_z +1,\
V=p\varphi^{p-1}-\frac{p}{p-1},\\
\label{}&B(z,s)=|(\varphi+q)|^{p-1}(\varphi+q)-\varphi^p-p\varphi^{p-1}q,\\
\label{}&R(z,s)=\partial_z^2\varphi-\frac{1}{2}z\partial_z \varphi-\frac{1}{p-1}\varphi+\varphi^p-\partial_s \varphi,\label{R}\\
\label{}&N(z,s)=H+\partial_z G(z,s),\\
\label{H}&H(z,s)=W(\partial_s \chi+\partial_z^2 \chi+\frac{1}{2}z\partial_z\chi)+|W|^{p-1}W(\chi-\chi^p),\end{align}
and
\begin{align}\label{GGG}&G(z,s)=-2\partial_z\chi W,\ \partial_zG(z,s)=-2\partial_z^2\chi W-2\partial_z\chi\partial_z W.\end{align}

We give a decomposition of the solution $q$ to equation (\ref{2.20}) according to the spectrum of $\mathcal{L}$.
The operator $\mathcal{L}$ is self-adjoint on $\mathcal{D}(\mathcal{L})\subset L_\rho^2(\mathbb{R})$ with
\begin{align} \rho(z)=\frac{e^{-\frac{z^2}{4}}}{\sqrt{4\pi}},\end{align}
and $$L_\rho^2(\mathbb{R})=\Big\{\vartheta\in L_{loc}^2(\mathbb{R});\ \|\vartheta\|_{L_\rho^2}^2=\int |\vartheta|^2\rho(z)dz<+\infty\Big\}.$$
The spectrum of $\mathcal{L}$ is explicitly given by
$$\mbox{spec}(\mathcal{L})=\{1-\frac{m}{2};\ m\in \mathbb{N}\}.$$
 All the eigenvalues are simple. For $1-\frac{m}{2}$ corresponds the eigenfunction
\begin{align}
h_m(z)=\sum\limits_{n=0}^{[\frac{m}{2}]}\frac{m!}{n!(m-2n)!}(-1)^nz^{m-2n}.
\end{align}
The eigenvalues $h_m$ satisfy the following orthogonality condition
$$\int h_nh_m\rho dz=2^n n!\delta_{n,m}.$$
In the following, we will often use the notation
\begin{align}\label{km}k_m=h_m/\|h_m\|_{L_\rho^2}^2.\end{align}

\medskip

Now,
let us first introduce
\begin{align}\label{chi1}
\chi_1(z,s)=\chi_0\Big(\frac{|z|}{K_0\sqrt{s}}\Big),
\end{align}
where $\chi_0$ is defined in (\ref{1.7}), $K_0\geq 1$ will be chosen
large enough.
We then write
\begin{align}\label{qu}q=q_e+q_b,\end{align} where
\begin{align}\label{qu1}
q_b=q\chi_1,\ q_e=q(1-\chi_1).
\end{align}
From (\ref{1.7}) and (\ref{chi1}), we see that
\begin{align}\label{qu2}\mbox{supp } q_b(s)\subset B(0,2K_0\sqrt{s}),\
  \mbox{supp } q_e(s)\subset \mathbb{R}\setminus B(0,K_0\sqrt{s}).\end{align}
Second, we decompose $q_b$ as follows,
\begin{align}\label{qu3}
q_b(z,s)=\sum\limits_{m=0}^2 q_m(s)h_m(z)+q_{-}(z,s),
\end{align}
where $q_m$ is the projection of $q_b$ on $h_m$, $q_{-}(z,s)=P_{-}(q_b)$, and $P_{-}$ is the projection on $\{h_i;\ i\geq 3\}$, the negative subspace of the operator $\mathcal{L}$.

\medskip

 In order to reduce the infinite dimensional problem proposed in Proposition \ref{lemma1.3} into a dimensional one, we need the following definition as in \cite{LinZaag}, which is a small adaptation of the original one in \cite{LinZaag}
 (see Definition 2.1 in \cite{LinZaag}).

 \begin{definition} [Shrinking set] \label{s} For all $K_0\geq 1$, $\varepsilon_0>0$, $A>0$, $\mu\in(0,1)$, $\eta_0\in(0,1]$, $T>0$ and $t\in [0,T)$, we say
a function $y$ defined in $\Omega\times[0,t]$ belongs to the set
$S^*(K_0,\varepsilon_0,A,\mu,\eta_0,T,t)$ at time $t$, if

(i)  $q(s)\in V_{K_0,A}(s)$,  where $s=-\log(T-t)$, $z=e^{s/2}(x-a)$, $q(s)$ is defined in (\ref{W1}), (\ref{W1j}), (\ref{w}), (\ref{w1}) and (\ref{w2}),
and $V_{K_0,A}(s)$ is the set of all functions $r\in L^{\infty}(\mathbb{R})$ such that
\begin{align}\left\{\begin{array}{ll}\label{ru} |r_m(s)|\leq As^{-2} (m=0, 1),\ |r_2(s)|\leq A^2s^{-2}\log s,\\|r_{-}(z,s)|\leq As^{-2}(1+|z|^3),\ \ \ \ |r_e(z,s)|\leq A^2s^{-1/2},
\end{array}\right.\end{align}
where
\begin{align}\label{r0r1r2}\left\{\begin{array}{ll} r_e(z,s)=(1-\chi_1(z,s))r(z,s),\ r_{-}(z,s)=P_{-}(\chi_1r),\\
\mbox{for}\ m\in \mathbb{N},\ r_m(s)=\int k_m(z)\chi_1(z,s)r(z,s)\rho(z) dz.
\end{array}\right.\end{align}

(ii)  $\mbox{For all}\ \ x\in \Omega$ with $|x-a|\geq \mu\varepsilon_0$, $|y(x,t)|\leq \eta_0$.
\end{definition}

First, let us give an estimate which will be used in many places in this paper.

\medskip

Since $a\in\omega$, there exists $\delta_0\in (0,1)$ such that for any
$\varepsilon_0\in(0,\delta_0]$, $(a-4\varepsilon_0,a+4\varepsilon_0)\subset \omega$. Then, for any
$K_0\geq 1$ and any $\varepsilon_0\in(0,\delta_0]$, there exists $s_0^1(K_0,\varepsilon_0)>0$ such that when $s_0\geq  s_0^1$, $K_0\sqrt{s_0}<\varepsilon_0e^{s_0/2}$.

\medskip

Next, we can derive the following bounds when $y(t)\in
S^*(K_0,\varepsilon_0,A,\mu,\eta_0,T,t)$ (see Lemma 2.2 in \cite{LinZaag}). Throughout the paper, $C$ is a constant number, which may be
different according to the context.
\begin{lemma}\label{W^*}For all $K_0\geq 1$, $\varepsilon_0>0$ and $A\geq 1$, there exists $s_0^2(K_0,\varepsilon_0,A)>0$ such that if $s\geq  s_0\geq s_0^2$,
$\mu\in(0,1)$,
$\eta_0\in(0,1]$,
$t\in[0,T)$ and $y(t)\in S^*(K_0,\varepsilon_0,A,\mu,\eta_0,T,t)$,
where $t=T-e^{-s}$, then

(i) $\forall z\in \mathbb{R}$, $|q(z,s)|\leq CA^2\frac{\log s}{s^2}(1+|z|^3)$,

(ii) $\|q(s)\|_{L^\infty}\leq C A^2s^{-1/2}$,

(iii) $\|W(s)\|_{L^\infty}\leq (p-1)^{-1/(p-1)}+2$.
\end{lemma}

By Definition \ref{s}, there are two types of estimates: item (i) and item (ii). In the following, we will give two different
priori estimates, following two estimates. For item (i), the following lemma plays an important role in reducing the infinite dimensional problem into a finite one in the blowup region, which is a
small adaptation of Proposition 4.4 in \cite{LinZaag}.

\begin{lemma}\label{qa}There exists $K_{0}^1\geq 1$ such that for any $K_0\geq K_0^1$,   any $\varepsilon_0\in(0,\delta_0]$ and any
$\widetilde{A}>0$,   there exists $\overline{A}(\widetilde{A},K_0,\varepsilon_0)\geq 1
$ such that for all $A\geq\overline{A}(\widetilde{A},K_0,\varepsilon_0)$,
there exists $\overline{s}_0(\widetilde{A},A,K_0,\varepsilon_0)>0$ such that for any $s_0\geq \overline{s}_0(\widetilde{A},A,K_0,\varepsilon_0)$,  any
$\mu\in(0,1)$ and any $\eta_0\in(0, 1]$,
for any solution of (\ref{2.20}),
we have the following property: if
   \begin{align}\left\{\begin{array}{ll} \label{property}
              |q_m(s_0)|\leq As_0^{-2},\ m=0,1,\\
               |q_2(s_0)|\leq \widetilde{A}s_0^{-2}\log s_0,\\
             |q_{-}(z,s_0)|\leq \widetilde{A}(1+|z|^3)s_0^{-2},\\
            \|q_e(s_0)\|_{L^\infty}\leq \widetilde{A}s_0^{-1/2},\end{array}\right. \end{align}
      if for some $s_1\geq s_0$,
we have  $y(t)\in S^*(K_0,\varepsilon_0,A,\mu,\eta_0,T,t)$, whenever  $t=T-e^{-s}$ with $ s\in[s_0,s_1]$,
      then, for all $s\in[s_0,s_1]$,
            \begin{align}\left\{\begin{array}{ll}
               |q_2(s)|\leq {A^2}s^{-2}\log s-s^{-3},\\
             |q_{-}(z,s)|\displaystyle\leq {\frac{A}{2}}(1+|z|^3)s^{-2},\label{property1}\\
            \|q_e(s)\|_{L^\infty}\leq \displaystyle\frac{A^2}{2\sqrt{s}}.\end{array}\right.\end{align}
\end{lemma}

For item (ii), given a small $x$, define $t=t_0(x)$ by
\begin{align} \label{bu region1} |x-a|=K_0\sqrt{(T-t_0(x))|\log(T-t_0(x))}|\end{align}
to see that the solution is in fact ``flat'' at that time. Then,  we
see that the solution of (\ref{profile1}) remains ``flat'' for later
time. By ``flat'', we mean an almost space independent function. More precisely, we claim the following lemma, which is a small adaptation
of Lemma 4.6 in \cite{LinZaag}.

\begin{lemma}\label{estimate1}  Let $\widetilde{t}\in (0,T)$. There exist $0<\varsigma_0<1$ and $K_0^2\geq 1$  such that for any $K_0\geq K_0^2$, there exists $\widehat{\delta}\in
(0,\delta_0)$ such that for any $\varepsilon_0\in (0,\widehat{\delta}]$  and $A\geq 1$, there exists $s_0^3(K_0,\varepsilon_0,A)$
such that if $s_0\geq s_0^3$,  $\mu\in(0,1)$ and $\eta_0\in(0,1]$, then for  $x_0\in  \{z\in \mathbb{R};\ 0<|z-a|<\varepsilon_0\}$,
  $x=x_0+\xi\sqrt{T-t_0(x_0)}$ with $|\xi|\leq |\log(T-t_0(x_0))|^{1/4}$, and  for $t=t_0(x_0)+\tau(T-t_0(x_0))$, $t\in[0,\widetilde{t}]$, if $y(t)\in S^*(K_0,\varepsilon_0,A,\mu,\eta_0,T,t)$, then it
  holds that the solution $y$ with initial data
  $y_0\in  W_0^{1,q}(\Omega)\bigcap W^{2,q}(\Omega)$ ($q>n+2$) to system (\ref{profile1}) satisfies
$$\forall t_0(x_0)\leq  t\leq \widetilde{t},\ \Big|\frac{y(x,t)}{y^*(x_0)}-\frac{U_{K_0}(\tau)}{U_{K_0}(1)}\Big|\leq \frac{C}{|\log|x_0-a||^{\varsigma_0}},$$
where
\begin{align}y^*(x_0)=\Big[\frac{(p-1)^2|x_0-a|^2}{8p|\log|x_0-a||}\Big]^{-\frac{1}{p-1}},\end{align}
and
\begin{align}U_{K_0}(\tau)=\Big((p-1)(1-\tau)+\frac{(p-1)^2}{4p}K_0^2\Big)^{-\frac{1}{p-1}}.\end{align}
Moreover, $|y(x_0,t)|\leq C_1(K_0)|y^*(x_0)|$ for all $t\in[0,\widetilde{t}]$, where $C_1(K_0)$ is a constant depending on $K_0$.
\end{lemma}

In order to give a priori estimate following estimate (ii) in Definition \ref{s}, we need the following lemma, which is similar to
 Proposition 4.7 in \cite{LinZaag}.

\begin{lemma}\label{P44}
For any $\epsilon\in(0,1]$, any $\varepsilon_0\in(0,\widehat{\delta}]$ and any $\sigma_1\geq 0$, there exist
 a positive constant $C_2>1$ independent of $\epsilon,\varepsilon_0$ and $\sigma_1$, $\Theta(\epsilon,\sigma_1)=$$\frac{\epsilon}{4e^{(\sigma_1+1)^{p-1}}}
$ and $T_2(\epsilon,\varepsilon_0,\sigma_1)=\min\Big\{1,\Big(\frac{\epsilon}{4C_2e^{(\sigma_1+1)^{p-1}}(\sigma_1+1)
(\frac{1}{\varepsilon_0}+\frac{1}{\varepsilon_0^2})}\Big)^2\Big\}$
such that for any  $(\mu_1,\mu_2)\in \big\{(\frac{1}{4},\frac{1}{2}),\ (\frac{3}{4},\frac{7}{8})\big\}$, for any  $\overline{\eta}\in[0,\Theta]$  and any
$\overline{t}\in[0, T_2]$, if $y$ is the solution to system (\ref{profile1}) in the time interval $[0, \overline{t}]$ with initial data
  $y_0\in  W_0^{1,q}(\Omega)\bigcap W^{2,q}(\Omega)$ ($q>n+2$)
which satisfies
\begin{align}&(i)\ \mbox{for}\ x\in \Omega\ \mbox{with}\ |x-a|\in \big[\mu_1\varepsilon_0,\mu_2\varepsilon_0\big] \mbox{and}\ t\in [0, \overline{t}],\ |y(x,t)|\leq \sigma_1;\\
&(ii)\ \mbox{for}\ x\in \Omega\ \mbox{with}\ |x-a|\geq\mu_1\varepsilon_0,\  |y(x,0)|\leq \overline{\eta}.\end{align}
Then, for any $t\in[0,\overline{t}]$ and any $x\in \Omega$ with $|x-a|\geq \mu_2\varepsilon_0,$ it holds that
$$|y(x,t)|< \epsilon.$$
\end{lemma}
\begin{proof}
Let $\epsilon\in(0,1]$, $\varepsilon_0\in(0,\widehat{\delta}]$, $\sigma_1\geq 0$ and $(\mu_1,\mu_2)\in \big\{(\frac{1}{4},\frac{1}{2}),\ (\frac{3}{4},\frac{7}{8})\big\}$.
Define $\mathcal{Y}$  by
\begin{align}\mathcal{Y}(x,t)=& y(x,t)\overline{\chi}(x),\ x\in\Omega.\end{align}
Here $\widetilde{\chi}(x)=1-{\widetilde{\chi}}_0(\frac{x-a}{\varepsilon_0})$, where  $\widetilde{\chi}_0\in C_0^\infty (\mathbb{R},[0,1])$ with
\begin{align}\label{1.7x}\widetilde{\chi}_0(\xi)=\left\{\begin{array}{ll} 1, \ |\xi|\leq \mu_1,\\0, \ |\xi|\geq \mu_2.
                                           \end{array}\right.\end{align}
Then, by equation (\ref{profile1}), it holds that
\begin{align}\partial_t \mathcal{Y}=\partial_{xx}\mathcal{Y}+\mathds{1}_\omega|y|^{p-1}\mathcal{Y}-2\partial_x(\widetilde{\chi}' y)+\widetilde{\chi}{'}{'}y\nonumber.
\end{align}
Therefore, by (ii), we can write
\begin{align}
\|\mathcal{Y}(t)\|_{L^\infty(\Omega)}\leq \Big\|S(t)y(\cdot,0)\widetilde{\chi}+ \int_0^t S(t-t')\big[\mathds{1}_\omega|y|^{p-1}I_{|x-a|\geq\mu_1\varepsilon_0}\mathcal{Y}\nonumber\\
-2\partial_x(\widetilde{\chi}' I_{|x-a|\geq\mu_1\varepsilon_0}y)+\widetilde{\chi}{'}{'}I_{|x-a|\geq\mu_1\varepsilon_0}y\big]dt'\Big\|_{L^\infty(\Omega)},\label{y22}
\end{align}
where $S(t)$ is the heat kernel.

\medskip

Let $\Theta(\epsilon,\sigma_1)=\frac{\epsilon}{4e^{(\sigma_1+1)^{p-1}}}
$ and $\overline{\eta}\in(0,\Theta]$.  We proceed by contradiction, and assume that there exists
$\overline{t}\in (0, T_2]$ where $T_2(\epsilon,\varepsilon_0,\sigma_1)=\min\Big\{1,\Big(\frac{\epsilon}{4C_2e^{(\sigma_1+1)^{p-1}}(\sigma_1+1)
(\frac{1}{\varepsilon_0}+\frac{1}{\varepsilon_0^2})}\Big)^2\Big\}$, such that the conclusion does not hold for all $t\in[0,\overline{t}]$, where
$C_2$ will be defined later.

\medskip

Using (ii) and the continuity of $y$,
this means that there is $\widehat{t}\in (0,\bar t]$ such that the
conclusion holds for all $t\in [0, \widehat{t})$ and fails at $t=\widehat{t}$. This
implies that
\begin{equation}\label{five}
\|\mathcal{Y}(\widehat{t})\|_{L^\infty(\mu_2\epsilon_0 \le |x-a|,\ x\in \Omega)}=\epsilon.
\end{equation}
Therefore, since $\widetilde{\chi}'$ and $\widetilde{\chi}''$ are supported by $\{x\in\Omega;\ \mu_1\varepsilon_0\leq |x-a|\leq \mu_2\varepsilon_0\}$
and satisfy $|\widetilde{\chi}'|\leq C/\varepsilon_0$,
$|\widetilde{\chi}''|\leq C/\varepsilon_0^2$, it holds by (i), (ii) and
(\ref{y22}) that for all $ t\in[0,\widehat{t}]$, we have
\begin{align}
&\|\mathcal{Y}(t)\|_{L^\infty(\Omega)}\nonumber\\\leq \overline{\eta}+ &(\sigma_1+1)^{p-1} \int_0^t \|\mathcal{Y}(t')\|_{L^\infty(\Omega)}dt'+\frac{C(\sigma_1+1)}{\varepsilon_0}\int_0^t\frac{dt'}{\sqrt{t-t'}}+\frac{C(\sigma_1+1)}{\varepsilon_0^2}\int_0^t dt'\nonumber\\
\leq &\overline{\eta}+ (\sigma_1+1)^{p-1} \int_0^t
       \|\mathcal{Y}(t')\|_{L^\infty(\Omega)}dt'+\frac{C(\sigma_1+1)}{\varepsilon_0}\sqrt{\widehat{t}}+\frac{C(\sigma_1+1)}{\varepsilon_0^2}\widehat{t}.\nonumber
\end{align}
In the above estimate, when it comes to bounding $|y|I_{\{|x-a|\ge
  \mu_1\epsilon_0\}}$, we will bound it by $\sigma_1$ if
$|x-a|\le \mu_2\epsilon_0$, and by $\epsilon (\leq 1)$ if $|x-a|\ge
\mu_2\epsilon_0$. Then, by Gronwall's estimate, we have for all $ t\in[0,\widehat{t}]$:
\begin{align}
\|\mathcal{Y}(t)\|_{L^\infty(\Omega)}&\leq e^{(\sigma_1+1)^{p-1}}\Big(\overline{\eta}+\frac{C(\sigma_1+1)}{\varepsilon_0}\sqrt{\widehat{t}}+\frac{C(\sigma_1+1)}{\varepsilon_0^2}\widehat{t}\Big)\nonumber\\
&\leq e^{(\sigma_1+1)^{p-1}}\overline{\eta}+
  C_2e^{(\sigma_1+1)^{p-1}}(\sigma_1+1)\sqrt{\widehat{t}}\Big(\frac{1}{\varepsilon_0}+\frac{1}{\varepsilon_0^2}\Big).\nonumber
\end{align}
Hence, if $\Theta(\epsilon,\sigma_1)=\frac{\epsilon}{4e^{(\sigma_1+1)^{p-1}}}
$, $T_2(\epsilon,\varepsilon_0,\sigma_1)=\min\Big\{1,\Big(\frac{\epsilon}{4C_2e^{(\sigma_1+1)^{p-1}}(\sigma_1+1)
(\frac{1}{\varepsilon_0}+\frac{1}{\varepsilon_0^2})}\Big)^2\Big\}$,
$\overline{\eta}\in(0,\Theta]$ and $\widehat{t}\leq \overline{t}\leq T_2$,
we have $\|\overline{y}(t)\|_{L^\infty(\Omega)}\leq \epsilon/2$,
for all $t\in[0,\widehat{t}]$.

\medskip

This contradicts  \eqref{five} and completes the proof of
Lemma  \ref{P44}.
\end{proof}

The following transverse crossing lemma also plays an important role to obtain our main result (see Lemma 4.8 in \cite{LinZaag}).

\begin{lemma}\label{lemma 3.2} There exist $K_0^3\geq 1$ and $A_1\geq 1$ such that for any $K_0\geq K_0^3$, $A\geq A_1$ and
 $\varepsilon_0\in(0,\widehat{\delta}]$,  there exists $s_0^4(K_0,\varepsilon_0,A)$ such that for any $s_0\geq s_0^4$, $\mu\in(0,1)$ and $\eta_0\in(0,1]$, we have the following properties:
Assume there exists $\widetilde{s}_*\geq s_0$ such that $y(\widetilde{t}_*)\in S^*(K_0,\varepsilon_0,A,\mu,\eta_0,T,\widetilde{t}_*)$ with $\widetilde{t}_*=T-e^{-\widetilde{s}_*}$ and $(q_0,q_1)(\widetilde{s}_*)\in \partial [-\frac{A}{(\widetilde{s}_*)^2}, \frac{A}{(\widetilde{s}_*)^2}]^2$, then there exists $\delta_1>0$
such that for all $\overline{\delta}\in (0,\delta_1)$, $(q_0,q_1)(\widetilde{s}_*+\overline{\delta})\not\in [-\frac{A}{(\widetilde{s}_*+\overline{\delta})^2}, \frac{A}{(\widetilde{s}_*+\overline{\delta})^2}]^2$.
\end{lemma}

\subsection{Proof of Proposition \ref{lemma1.3}}

In this subsection, we will follow the almost same techniques with a small modification in the proof of Theorem 1 in \cite{LinZaag}
 to reprove Proposition \ref{lemma1.3} by the lemmas stated in subsection 2.1. The idea stated in this subsection will also be used to prove the
 stability of blowup profile for auxiliary system (\ref{profile1}) in the next section.

\medskip
The proof of Proposition \ref{lemma1.3} can be  divided into the following three steps.

\medskip

{\bf Step 1. Preparation of initial data}

\medskip

Given $a\in \omega$, we consider initial data for system (\ref{profile1}) defined for all $x\in \Omega$ by
\begin{align}\label{initial data}
y_0(x,d_0,d_1)=&T^{-\frac{1}{p-1}}\Big\{\varphi(z,s_0)\chi(8z,s_0)+\frac{A}{s_0^2}(d_0+d_1z)\chi_1(2z,s_0)\Big\}\nonumber\\
=&T^{-\frac{1}{p-1}}\Big\{\varphi\Big(\frac{x-a}{\sqrt{T}},-\log T\Big)\chi_0\Big(\frac{8(x-a)}{\varepsilon_0}\Big)\nonumber\\&+\frac{A}{s_0^2}\Big(d_0+\frac{d_1(x-a)}{\sqrt{T}}\Big)\chi_0\Big(\frac{2|x-a|}{K_0\sqrt{-T\log T}}\Big)\Big\},
\end{align}
where $T>0$ will be sufficiently small,
$s_0=-\log T$, $z=\frac{x-a}{\sqrt{T}}$, $\chi_0$ is defined in (\ref{1.7}), $\chi$ is defined in (\ref{chi}) and $\chi_1$ is defined in (\ref{chi1}).

\medskip

We have the following lemma, which is Lemma 3.1 in \cite{LinZaag}.

\begin{lemma} (Reduction for initial data) \label{initial data1} For
  any  $K_{0}\geq 1$, for any  $\varepsilon_0\in(0,\delta_0]$ and for
  any $A\geq 1$, there exists  $s_0^5(K_0,\varepsilon_0,A)>0$ such
  that for all
$s_0\geq s_0^5$, the
following holds:

\medskip

If initial data for equation (\ref{profile1}) are given by  (\ref{initial data}), then there exists a rectangle
\begin{align}D_T\subset [-2,2]^2,\end{align}
such that

(i) for all $(d_0,d_1)\in D_T$, we have $y_0\in C_0^\infty(\omega)$ and
$$y_0(\cdot,d_0,d_1)\in S^*(K_0,\varepsilon_0,A,1/4,0,T,0),$$
and \begin{align}
\left\{\begin{array}{ll}
|q_2(s_0)|\leq s_0^{-2}\log s_0,\\
$$|q_{-}(z,s_0)|\leq (1+|z|^3) s_0^{-2},\\
$$|q_e(\cdot,s_0)|\leq s_0^{-1/2}.\end{array}\right.\label{step 1}\end{align}

Moreover, the restriction of
$(d_0,d_1) \mapsto (q_0(s_0),q_1(s_0))$ to the boundary of $D_T$ has
degree 1.

(ii) for all $(d_0,d_1)\in D_T$ and for all $x\in \Omega$ with $|x-a|\geq \varepsilon_0/4$,
$$y_0(x,d_0,d_1)=0.$$
\end{lemma}

{\bf Step 2. Reduction to a finite-dimensional problem}

\medskip

Let $K_0\geq K_0^*:=\max\{K_0^1,K_0^2,K_0^3\}$, $\varepsilon_0\in(0,\widehat{\delta}]$, $A\geq A^*:=\max\{A_1,\overline{A}(1,K_0,\varepsilon_0)\}$,
$s_0=-\log T\geq s_0^*:=\max\{s_0^1,s_0^2,s_0^3,s_0^4,s_0^5,\overline{s}_0(1,A,K_0,\varepsilon_0)\}$ and $\eta_0\in(0,1]$.
Consider $(d_0,d_1)\in D_T$. By the classical theory of parabolic equations, we can define a
maximal solution $y$ to equation (\ref{profile1}) with initial data (\ref{initial data}), and a maximal time $t_*(d_0,d_1)\in [0,T]$ such that
\begin{align}
\forall t\in[0,t_*),\ y(t)\in S^*(K_0,\varepsilon_0,A,1/2,\eta_0,T,t).
\end{align}
Then, one of the following possibilities occurs:

\medskip

(1) either $t_*=T$,

(2) or $t_*<T$ and from continuity, $y(t_*)\in \partial S^*(K_0,\varepsilon_0,A,1/2,\eta_0,T,t_*)$,
in the sense that when $t=t_*$, one `$\leq$' symbol in the definition
of $S^*(K_0,\varepsilon_0,A,1/2,\eta_0,T,t_*)$ is replaced by the symbol ``$=$''.

\medskip

Our aim is to show that for $\mu=1/2$, for all $K_0$ and $A$ large enough, $\varepsilon_0$ and $T$ small enough and for any $\eta_0\in (0,1]$,
one can find $(d_0,d_1)\in D_T$ such that $t_*(d_0,d_1)=T$.

\medskip

We argue by contradiction
and assume that for all $(d_0,d_1)\in D_T$, $t_*(d_0,d_1)<T$.

\medskip

On one hand, by Lemma \ref{initial data1} and Lemma \ref{qa}, we can derive the following proposition.

\begin{proposition}\label{qb}For any $K_0\geq K_0^*$, for any $\varepsilon_0\in(0,\widehat{\delta}]$,  for any $A\geq A^*$,
 for any  $s_0\geq s_0^*$ and for any  $\eta_0\in(0,1]$,
we have the following property: if $(d_0,d_1)\in D_T$ and
            if for $s_1\geq s_0$, we have $\forall s\in[s_0,s_1]$,
            $y(t)\in S^*(K_0,\varepsilon_0,A,1/2,\eta_0,T,t)$ with $t=T-e^{-s}$, then, for all $s\in[s_0,s_1]$,
            \begin{align}\left\{\begin{array}{ll}
            \label{p2}|q_2(s)|\leq {A^2}s^{-2}\log s-s^{-3},\\
             |q_{-}(z,s)|\leq\displaystyle {\frac{A}{2}}(1+|z|^3)s^{-2},\\
            \|q_e(s)\|_{L^\infty}\leq \displaystyle\frac{A^2}{2\sqrt{s}}. \end{array}\right.\end{align}
\end{proposition}

On the other hand, write $\mathcal{R}$ for $\{x\in \Omega;\ |x-a|\geq \varepsilon_0/{2}\}$.
We have that
\begin{align}\label{P2}\mbox{if}\  x\in \mathcal{R},\ \mbox{then}\
  |y(x,t_*)|\leq \frac{\eta_0}{2}, \end{align}
provided that the parameters satisfy some conditions.

\medskip

Indeed, we consider $K_0\geq K_0^*$,  $\varepsilon_0\in(0,\widehat{\delta}]$, $A\geq A^*$, $\eta_0\in(0,1]$,
and
$$s_0\geq s_1^*:=\max{\Big\{s_0^*(K_0,\varepsilon_0,A),
 -\log\Big(T_2\Big(\frac{\eta_0}{2},\varepsilon_0,C_1(K_0) \Big[\frac{(p-1)^2\big(\frac{\varepsilon_0}{4}\big)^2}{8p|\log\frac{\varepsilon_0}{4}\Big|}\Big]^{-\frac{1}{p-1}}\Big)\Big\}}.$$
 Applying
Lemma \ref{estimate1}, we see that if $y(t)\in S^*(K_0,\varepsilon_0,A,1/2,\eta_0,T,t)$ for all
$t\in[0,t_*]$, then for $x\in \Omega$ with
 $$\frac{\varepsilon_0}{4}\leq|x-a|\leq \frac{\varepsilon_0}{2},\ \forall \ t\in[0,t_*],$$$$\ |y(x,t)|\leq C_1(K_0)|y^*({x})|\leq C_1(K_0) \Big[\frac{(p-1)^2\Big(\frac{\varepsilon_0}{4}\Big)^2}{8p|\log\frac{\varepsilon_0}{4}\Big|}\Big]^{-\frac{1}{p-1}}.$$
Using  Lemma \ref{initial data1}, we see that for all $x\in \Omega\ \mbox{with}\ |x-a|\geq \frac{\varepsilon_0}{4}$, $y(x,0)=0$.

\medskip

Therefore, Lemma \ref{P44} applies with $(\mu_1,\mu_2)=(\frac{1}{4},\frac{1}{2})$,
$\epsilon=\frac{\eta_0}{2}$, $\overline{\eta}=0$ and
$\sigma_1=C_1(K_0)\Big[\frac{(p-1)^2\big(\frac{\varepsilon_0}{4}\big)^2}{8p|\log\frac{\varepsilon_0}{4}\Big|}\Big]^{-\frac{1}{p-1}}$, and we see that
 \begin{align}\label{P5}\forall x\in \Omega\ \mbox{with}\ |x-a|\geq\frac{\varepsilon_0}{2}, \forall \ t\in[0,t_*],\ |y(x,t)|\leq \frac{\eta_0}{2}.\end{align}

Hence, (\ref{P2}) holds.\\

{\bf Step 3. Final argument for Proposition \ref{lemma1.3}}

\medskip

In this step, we shall conclude the proof of Proposition \ref{lemma1.3}}
 thanks to a topological argument.

\medskip

  We fix
  $K_0\geq K_0^*$, $\varepsilon_0\in(0,\widehat{\delta}]$, $A\geq A^*$, $\eta_0\in(0,1]$ and $S_0\geq {s}_1^*$.
  Take then $s_0\geq S_0$ and introduce $T=e^{-s_0}$.

  \medskip

Note first that for any $(d_0,d_1)\in D_T$ defined in Lemma
\ref{initial data1}, equation (\ref{profile1}) fitted with initial data
\eqref{initial data} has a unique solution $y(\cdot,d_0,d_1)\in
C([0,t^*),
L^\infty(\Omega))$ for some maximal $t^*=t^*(d_0,d_1)\in(0,+\infty]$,
thanks to a classical fixed point argument. With
some classical parabolic estimates, we easily prove that $\nabla y (\cdot,d_0,d_1) \in
C([0,t^*), L^\infty(\Omega))$. Using Lemma \ref{initial data1}, we see that
$y_0(\cdot,d_0,d_1)\in S^*(K_0,\varepsilon_0,A,1/4,0,T,0)$. Therefore, we may introduce the largest
time $t_*=t_*(d_0,d_1)\in [0,\min(T,t^*(d_0,d_1))]$ such that for all $t\in
[0,t_*)$, $y(t,d_0,d_1)\in S^*(K_0,
\epsilon_0,A,1/2,\eta_0, T,t)$. Then, 2 cases may occur:\\
(i) either $t_*(d_0,d_1) = T$,\\
(ii) or $t_*(d_0,d_1) <T$, and in that case, one of the
``$\le$'' signs appearing in Definition \ref{s} where $S^*(K_0,
\epsilon_0,A,1/2,\eta_0, T,t)$ is
introduced should be replaced by a ``$=$'' sign.

\medskip

We proceed in two steps:\\
1. we argue by contradiction and prove the existence of some
$(d_0,d_1)\in D_T$ such that $t_*(d_0,d_1) = T$.\\
2. we consider such a $(d_0,d_1)\in D_T$ and show that the
maximal existence time of $y(\cdot,d_0,d_1)$ is $T$, namely that $t^*(d_0,d_1)=T$. Then, we conclude the proof of Proposition \ref{lemma1.3}.

\bigskip

\textbf{1. A topological argument}

\bigskip

We proceed by contradiction, and assume that for any $(d_0,d_1)\in
D_T$, $t_*(d_0,d_1) <T$. By definition of $t_*(d_0,d_1)$, Proposition
\ref{qb} applies and estimate (\ref{P2}) holds, leading to the fact the
equality case mentioned in item 2 above at time $t=t_*$ concerns
only the first two components mentioned in Definition \ref{s}, namely
that
%
%
%
%
%
%
%
%
%
\[
 (q_0,q_1)(d_0,d_1,s_*(d_0,d_1))\in \partial [-\frac{A}{(s_*)^2},\frac{A}{(s_*)^2}]^2
 \mbox{ where } s_*=-\log(T-t_*).
 \]
In particular, we can define the rescaled flow:
$$\Phi: D_T\rightarrow \partial ([-1,1])^2,$$
$$(d_0,d_1)\rightarrow \frac{s_*(d_0,d_1)^2}{A}(q_0,q_1)(d_0,d_1,s_*(d_0,d_1)).$$
Then, we have the following proposition.
\begin{proposition}\label{PP}
  $ $\\
1) $\Phi$ is a continuous mapping from $D_T$ to $\partial ([-1,1]^2)$.\\
2) The restriction of $\Phi$ to the boundary $\partial D_T$  has
degree 1.
\end{proposition}

From this proposition, a contradiction follows thanks to the following statement which is a consequence of Brouwer's lemma,
and could be obtained directly by the proof of of Theorem 12.8  in Smoller \cite{Sspri94} (pp. 136), with a natural adaptation.

 \begin{proposition}\label{S94}  Consider $\Psi$ a continuous
  function from the unit ball of $\mathbb{R}^N$ to its boundary. Then, its restriction
  to the boundary has degree zero.\end{proposition}

  This means that there exists $(d_0,d_1)\in D_T$ such that
  $t_*(d_0,d_1) =T$. By definition of $t_*$, this means that $y(t)\in
   S^*(K_0,\varepsilon_0,A,1/2,\eta_0,T,t)$ for all $t\in [0,T)$.

  \bigskip

  \textbf{2. Conclusion of the proof of Proposition \ref{lemma1.3}}

\bigskip

  Let us summarize the result we have just proved:

  \medskip

If we fix
$K_0\geq K_0^*$, $\varepsilon_0\in(0,\widehat{\delta}]$, $A\geq A^*$, $\eta_0\in(0,1]$, $s_0\geq {s}_1^*$ and $T=e^{-s_0}$,
then we see that for some $(d_0,d_1)\in D_T$, equation (\ref{profile1})
with initial data (\ref{initial data}) has a solution $y$ such that
for any $t\in[0,T)$, $y(t)\in S^*(K_0,\varepsilon_0,A,1/2,\eta_0,T,t)$.
By  Lemma \ref{W^*}, we see that for any $s\geq-\log T$ and for any $z\in \mathbb{R}$,
$$|q(z,s)|\leq \frac{CA^2}{\sqrt{s}}.$$
By (\ref{fnew}), (\ref{w}), (\ref{w1}) and (\ref{w2}), it holds that  for
any $s\geq-\log T$ and for any $z$ with $|z|\leq
\varepsilon_0e^{s/2},$ we have
$$\Big|W(z,s)-f(\frac{z}{\sqrt{s}})\Big|\leq \frac{CA^2}{\sqrt{s}}+\frac{C}{s}.$$
Hence, if $T$ is small enough, we obtain from (\ref{W1}) and (\ref{W1j}) that for any
$t\in[0,T)$ and for any $x\in \Omega$ with $|x-a|\leq \varepsilon_0 $,
\begin{align}\label{e*}\Big|(T-t)^{1/(p-1)} y(x,t)-f\Big(\frac{x-a}{\sqrt{(T-t)|\log(T-t)|}}\Big)\Big|\leq \frac{C(A)}{\sqrt|\log(T-t)|}.\end{align}
This implies (\ref{1.3*}).

\medskip

This also implies that $|y(a,t)|\sim [(p-1)(T-t)]^{-1/(p-1)}$ as
$t\to T$. Hence, $y$ blows up in finite time $T$, and $a$ is a blowup
point of $y$. In particular, this means that $t^*(d_0,d_1)$, the maximal existence
time of $y$, is $T$.


\medskip

It remains to prove that all points $x_0\neq a$ are not blowup points.

\medskip

We know from item (ii) in Definition \ref{s} that for all $x\in \Omega$ with $|x-a|\geq \frac{\varepsilon_0}{2},\ |y(x,t)|\leq \eta_0$.
Thus,  any $x_0\in \Omega$ with $|x_0-a|\geq \frac{\varepsilon_0}{2}$ is not a blowup point. Now, if $0<|x_0-a|\leq \varepsilon_0/2$, the following result from Giga and Kohn \cite{Giga-1} allows us to conclude
Proposition \ref{lemma1.3}.

\begin{proposition}\label{GK}
For all $K>0$, there is $\eta_1>0$ such that if $\nu(\xi,\tau)$ solves
\begin{align}\label{a2.5}|\nu_t-\Delta \nu|\leq K(1+|\nu|^p),\end{align}
and satisfies
$$|\nu(\xi,\tau)|\leq \eta_1(T-t)^{-\frac{1}{p-1}},$$
for all $(\xi,\tau)\in B(a,r)\times [T-r^2,T)$ for some $a\in \mathbb{R}$ and $0<r\leq 1$, then $\nu$ does not blows up at $(a,T)$.
\end{proposition}

  Indeed, since $0<|x_0-a|\leq \varepsilon_0/2$, it follows from (\ref{e*}) that
\begin{align}\sup\limits_{|x-x_0|\leq\frac{|x_0-a|}{2}}|(T-t)^{1/(p-1)} y(x,t)|\leq \Big|f\Big(\frac{x_0-a}{2\sqrt{(T-t)|\log(T-t)|}}\Big)\Big|+ \frac{C(A)}{\sqrt{|\log(T-t)|}}\rightarrow 0,\nonumber\end{align}
as $t\rightarrow T$.
By (\ref{profile1}), it is easily obtained that  $y$ satisfies  assumption (\ref{a2.5}).
Therefore, applying Proposition \ref{GK}, we see that $x_0$ is not a blowup point.

\medskip

This completes the proof of Proposition \ref{lemma1.3}. $\Box$

\bigskip

\section{Stability of blowup profile for auxiliary system}

In this section, we shall prove Theorem \ref{Pros1.2}, which gives a stability result of the blowup profile with respect to initial data for
auxiliary system (\ref{profile1}).

\subsection{Formulation of the problem}

Let $\widehat{a}\in \omega$. From section 2, we know that there exists $K_0^*\geq 1$ such that for any $K_0\geq K_0^*$, there exists $\widehat{\delta}>0$ with
$(\widehat{a}-4\widehat{\delta},\widehat{a}+4\widehat{\delta})\subset\omega$ such that
for any $K_0\geq K_0^*$ and $\varepsilon_0\in(0,\widehat{\delta}]$, there exists $A^*(K_0,\varepsilon_0)\geq 1$ such that for any
$\widehat{A}\geq A^*(K_0,\varepsilon_0)$ and $\widehat{\eta}_0\in(0,1]$, there exists $\widehat{s}_1^*(K_0,\varepsilon_0,\widehat{A},\widehat{\eta}_0)$
such that for any $\widehat{T}>0$ with $\widehat{s}_0=-\log \widehat{T}\geq \widehat{s}_1^*(K_0,\varepsilon_0,\widehat{A},\widehat{\eta}_0)$,
we could construct an initial data $\widehat{y}_0({K}_0,\varepsilon_0,\widehat{A},\widehat{T})\in C_0^\infty(\omega)$ such that
for any $t\in[0,\widehat{T})$, $\widehat{y}(t)\in S^*({K}_0,\varepsilon_0,\widehat{A},1/2,\widehat{\eta}_0,\widehat{T},t)$. In particular,
 $\widehat{q}(s)\in V_{{K}_0,\widehat{A}}(s)$ for all $s\geq-\log \widehat{T}$, and  $|\widehat{y}(x,t)|\leq \widehat{\eta}_0$ for all $x\in \Omega$ with $|x-\widehat{a}|\geq\frac{\varepsilon_0}{2}$.
   Moreover, $\widehat{T}$ is its blowup time and $\widehat{a}$
is its unique blowup point.

\medskip

Fix
$K_0\geq K_0^*$, $\widehat{A}\geq A^*$, $\varepsilon_0\in(0,\widehat{\delta}]$,  and for any $\eta_0\in(0,1]$ but fixed, we fix
$\widehat{\eta}_0\in \Big(0,\frac{\eta_0}{8e^{(\sigma_1+1)^{p-1}}}\Big)$, $\sigma_1= C_1(K_0)\Big[\frac{(p-1)^2\big(\frac{\varepsilon_0}{4}\big)^2}{8p|\log\frac{\varepsilon_0}{4}|}\Big]^{-\frac{1}{p-1}}$ with
$C_1(K_0)\geq 1$ is defined in Lemma \ref{estimate1} and we fix $\widehat{s}_0\geq \widehat{s}_1^*(K_0,\varepsilon_0,\widehat{A},\widehat{\eta}_0)$.

\medskip

Consider an initial data $y_0\in W_0^{1,q}(\Omega)\bigcap W^{2,q}(\Omega)$ such that $\|\widetilde{\varepsilon}_0\|_{W_0^{1,q}(\Omega)\bigcap W^{2,q}(\Omega)}$ is small, where  $q>n+2$ and
$$\widetilde{\varepsilon}_0=y_0-\widehat{y}_0.$$
We denote by $y_{y_0}$ the solution of equation (\ref{profile1}) with initial data $y_0$ and $T(y_0)\leq +\infty$ its maximal time of existence.

\medskip

 For $(T,a)\subset\mathbb{R}^+\times\omega$, we introduce
\begin{align}\label{W11}
{W}_{T,a,y_0}(z,s)=(T-t)^{1/(p-1)} {y}_{y_0}(x,t)=e^{-s/(p-1)}{y}_{y_0}(a+e^{-s/2}z,T-e^{-s})
\end{align}
with
\begin{align}\label{W1jj}x-a=(T-t)^{1/2}z,\ T-t=e^{-s}. \end{align}

Define
\begin{align}\label{www}w_{T,a,y_0}(z,s)=\left\{\begin{array}{ll} W_{T,a,y_0}(z,s)\chi(z,s), \ &z\in (\Omega-a) e^{s/2},\\0, \ &z\in \mathbb{R}\backslash(\Omega-a) e^{s/2},
\end{array}\right.\end{align}
where $\chi$ is defined in (\ref{chi}).

\medskip

Let \begin{align}\label{w11}q_{T,a,y_0}(z,s)=w_{T,a,y_0}(z,s)-\varphi(z,s),\end{align}
where   $\varphi$ is defined in (\ref{w2}).
We then write
\begin{align}\label{quu}q_{T,a,y_0}=q_e+q_b,\end{align} where
\begin{align}\label{qu1}
q_b=q_{T,a,y_0}\chi_1,\ q_e=q_{T,a,y_0}(1-\chi_1).
\end{align}
Second, we decompose $q_b$ as follows,
\begin{align}\label{qu3}
q_b(z,s)=\sum\limits_{m=0}^2 q_m(s)h_m(z)+q_{-}(z,s),
\end{align}
where $q_m$ is the projection of $q_b$ on $h_m$, $q_{-}(z,s)=P_{-}(q_b)$, and $P_{-}$ is the projection on $\{h_i;\ i\geq 3\}$, the negative subspace of the operator $\mathcal{L}$.

\medskip

Our aim is to show that for  $A>0$, $s_0\geq \widehat{s}_0=-\log \widehat{T}$ large enough and for some $\mu\in(0,1)$, if $\|\widetilde{\varepsilon}_0\|_{W_0^{1,q}(\Omega)\bigcap W^{2,q}(\Omega)}$  is small enough, then
$T(y_0)$ is finite, and $y_{y_0}$ blows up at time $T(y_0)$ only at unique blowup point $a(y_0)$, with
\begin{align}\label{convergence}T(y_0)\rightarrow \widehat{T},\ a(y_0)\rightarrow \widehat{a}, \ \mbox{as}\ \widetilde{\varepsilon}_0=y_0-\widehat{y}_0\rightarrow 0\ \mbox{in} \ W_0^{1,q}(\Omega)\cap W^{2,q}(\Omega),\end{align} $y_{y_0}(t)\in S^*(K_0,\varepsilon_0,A,\mu,\eta_0,T(y_0),t)$, which impies $q_{T(y_0),a(y_0),y_0}(s)\in V_{K_0,A}(s)$
for any $s\geq \sigma_0$ large enough and $t=T(y_0)-e^{-s}$.

\medskip

In order to achieve the aim in (\ref{convergence}), we will study $q_{T,a,y_0}$, where $(T,a)$ is arbitrary in a small neighborhood of $(\widehat{T},\widehat{a})$, hoping to have some hint that some
 $(\overline{T}(y_0),\overline{a}(y_0))$ close to $(\widehat{T},\widehat{a})$ will correspond to the aimed $({T}(y_0),{a}(y_0))$.

\medskip

Note that  all the $q_{T,a,y_0}$ satisfy the same equation for all $(z,s)\in \mathbb{R}\times[-\log T,-\log(T-T(y_0))_+)$ (by convention, we note $-\log 0_+=\infty$),
\begin{align}\label{equation L}
\partial_s q_{T,a,y_0}=(\mathcal{L}+V)q_{T,a,y_0}+B_{T,a,y_0}(z,s)+R(z,s)+N_{T,a,y_0}(z,s),
\end{align}
where $\mathcal{L}$ and $V$ are defined in (\ref{LD}), $R$ is defined in (\ref{R}) and
\begin{align}\label{LDD}&B_{T,a,y_0}(z,s)=|(\varphi+q_{T,a,y_0})|^{p-1}(\varphi+q_{T,a,y_0})-\varphi^p-p\varphi^{p-1}q_{T,a,y_0},\\
\label{}&N_{T,a,y_0}(z,s)=H_{T,a,y_0}+\partial_z G_{T,a,y_0}(z,s),\\
\label{H}&H_{T,a,y_0}(z,s)=W_{T,a,y_0}(\partial_s \chi+\partial_z^2 \chi+\frac{1}{2}z\partial_z\chi)+|W_{T,a,y_0}|^{p-1}W_{T,a,y_0}(\chi-\chi^p),\\
\label{GGG}&G_{T,a,y_0}(z,s)=-2\partial_z\chi W_{T,a,y_0},\ \partial_zG_{T,a,y_0}(z,s)=-2\partial_z^2\chi W_{T,a,y_0}-2\partial_z\chi\partial_z W_{T,a,y_0}.
\end{align}



Introducing
$$\widetilde{\varepsilon}(x,t)=y_{y_0}(x,t)-\widehat{y}(x,t),\ \mbox{for all}\ x\in\Omega,\ \mbox{and}\ 0\leq t<\min(T(y_0),\widehat{T}).$$
Let \begin{align}\label{dddd}\tau=(T-\widehat{T})e^{s_0},\ \alpha=(a-\widehat{a})e^{s_0/2},\ \sigma_0=s_0-\log(1+\tau)\ \mbox{and}\ \widetilde{z}=z\sqrt{(1+\tau)}+\alpha.\end{align}

For any $s_0\in[-\log \widehat{T},-\log(\widehat{T}-T(y_0))_+)$,  $q_{T,a}$ is
defined (\textit{at least}) on $(-\log T,\sigma_0]$ for $T$ near
$\widehat{T}$.

\medskip

For any $s_0\in[-\log \widehat{T},-\log(\widehat{T}-T(y_0))_+)$, we have that if $z\in (\Omega-a)e^{\sigma_0/2}$, then
\begin{align}\label{qt}&q_{T,a,y_0}(z,\sigma_0)=\overline{\psi}(s_0,y_0,T,a,z)\nonumber\\:=&(1+\tau)^{\frac{1}{p-1}}
[\widehat{q}(\widetilde{z},s_0)+\varphi(\widetilde{z},s_0)+e^{-\frac{s_0}{p-1}}
\widetilde{\varepsilon}(e^{-\frac{s_0}{2}}\widetilde{z}+\widehat{a},\widehat{T}-e^{-s_0})\chi(z,\sigma_0)]-\varphi(z,\sigma_0)\nonumber\\
&+e^{-\frac{\sigma_0}{p-1}}\widehat{y}
(ze^{-\frac{\sigma_0}{2}}+a,T-e^{-\sigma_0}))[\chi(z,\sigma_0)-\chi(\widetilde{z},s_0)].
\end{align}

If $z\not\in (\Omega-a)e^{\sigma_0/2}$, then
\begin{align}\label{zzz}q_{T,a,y_0}(z,\sigma_0):=\overline{\psi}(s_0,y_0,T,a,z)=-\varphi(z,\sigma_0).\end{align}

\medskip

In the following, $\overline{\psi}(s_0,y_0,T,a,z)$ appears as initial data for equation (\ref{equation L}) at initial time $s=\sigma_0$. Consider
the mapping $(T,a)\rightarrow \overline{\psi}(s_0,y_0,T,a,z)$.  Similarly as in \cite{LinZaag} and \cite{TZtams19}, we start from  $\overline{\psi}(s_0,y_0,T,a,z)$ at initial time $s=\sigma_0$ including
the reduction to a finite-dimensional problem, the construction procedure would work.

\subsection{Behavior of ``initial data" $\overline{\psi}$ for $(T,a)$ near $(\widehat{T},\widehat{a})$}

In this subsection, we will prove  the following proposition, which gives the properties of initial data $\overline{\psi}(s_0,y_0,T,a,z)$.

\medskip

Recall that we have fixed
$K_0\geq K_0^*$, $\varepsilon_0\in(0,\widehat{\delta}]$, $\widehat{A}\geq A^*$. Let $\eta_0\in(0,1]$.
We fix $\widehat{\eta}_0\in \Big(0,\frac{\eta_0}{8e^{(\sigma_1+1)^{p-1}}}\Big)$, where $\sigma_1= C_1(K_0)\Big[\frac{(p-1)^2\big(\frac{\varepsilon_0}{4}\big)^2}{8p|\log|\frac{\varepsilon_0}{4}||}\Big]^{-\frac{1}{p-1}}$ with
$C_1(K_0)\geq 1$. Let $b=\frac{(p-1)^2}{4p}$.
\begin{proposition}\label{prop3.1}There exists $\overline{C}\geq 1$ such that for any $A\geq \overline{C}\widehat{A}$, there
exists $\overline{s}_0(A)\geq \widehat{s}_1^*(K_0,\varepsilon_0,\widehat{A},\widehat{\eta}_0)$ large enough, such that for any $s_0\geq\overline{s}_0(A)$, there exists $\widehat{\varepsilon}_0(s_0)>0$ small enough such that
for all $y_0\in
\{y\in {W_0^{1,q}(\Omega)\cap W^{2,q}(\Omega)};\ \|y-\widehat{y}_0\|_{W_0^{1,q}(\Omega)\cap W^{2,q}(\Omega)}\leq \widehat{\varepsilon}_0(s_0)\}$, where
$q>n+2$, the following holds:

\medskip

(i) There exists a set
\begin{align}\label{dbar}\overline{D}_{A,s_0,y_0}\subset \Big\{(T,a)\Big|\ |T-\widehat{T}|\leq \frac{2e^{-s_0}(p-1)A}{\kappa s_0^2},\ |a-\widehat{a}|\leq \frac{e^{-s_0/2}(p-1)^2A}{b\kappa s_0}\Big\},\end{align}
whose boundary is a Jordan curve such that the mapping
\begin{align}\label{}(T,a)\rightarrow \sigma_0^2(\overline{\psi}_0,\overline{\psi}_1)(s_0,y_0,T,a),
\ \mbox{where}\ \sigma_0=s_0-\log(1+\tau),\end{align}
is one-to-one from $\overline{D}_{A,s_0,y_0}$ onto $[-A,A]^2$. Moreover, it is of degree -1 on the boundary.

\medskip

(ii) For all $(T,a)\in \overline{D}_{A,s_0,y_0},\ \overline{\psi}(s_0,y_0,T,a)\in V_{{K}_0,A}(\sigma_0)$ with strict inequalities
except for the first two, namely $(\overline{\psi}_0,\overline{\psi}_1)(s_0,y_0,T,a)$, in the sense that
\begin{align}&|\overline{\psi}_m|\leq\frac{A}{\sigma_0^2},\ m=0,\ 1,\nonumber\\
&|\overline{\psi}_2|< C\widehat{A}^2\sigma_0^{-2}\log \sigma_0,\nonumber\\
&\left\|\frac{\overline{\psi}_{-}(z,\sigma_0)}{1+|z|^3}\right\|_{L^{\infty}(\mathbb{R})}<\frac{C\widehat{A}}{\sigma_0^2},\nonumber\\
&\Big\|{\overline{\psi}_{e}(z,\sigma_0)}\Big\|_{L^{\infty}(\mathbb{R})}<\frac{C\widehat{A}^2}{\sigma_0^{\frac{1}{2}}},\nonumber\end{align}
for some positive constant $C>0$.

(iii) For all $(T,a)\in \overline{D}_{A_0,s_0,y_0}$ and for all $x\in\Omega$
 with $|x-a|\geq 3\varepsilon_0/4$, $|{y}_{y_0}(x,T-e^{-\sigma_0})|\leq2\widehat{\eta}_0$.

\end{proposition}

Define
\begin{align}\label{gL}\overline{g}(z,\sigma_0)=(1+\tau)^{\frac{1}{p-1}}g(\widetilde{z},s_0),\end{align}
where $\widetilde{z}=z\sqrt{1+\tau}+\alpha$, $\sigma_0=s_0-\log(1+\tau)$, $g(s_0)\in L^{\infty}(\mathbb{R})$.
In order to prove Proposition \ref{prop3.1}, we need the following lemma, which is an adaptation of Lemma 6.4 in \cite{TZtams19}.
\begin{lemma}\label{lemma g}
(Expansions of $\overline{g}(z,\sigma_0))$ There exists $\widetilde{s}_{0}>0$ large enough such that for any $s_0\geq \widetilde{s}_{0}$,
and $(\tau,\alpha)\in \mathbb{R}^2$ satisfying
\begin{align}\label{taualpha}|\tau|\leq \frac{1}{16},\ |\alpha|\leq \frac{1}{4}\end{align}
such that if $g(s_0)\in L^{\infty}(\mathbb{R})$ and $\overline{g}$ is given by (\ref{gL}), then

\medskip

(i) $(\frac{\sigma_0}{s_0})^2\overline{g}_0(\sigma_0)=g_0(s_0)(1+\frac{\tau}{p-1})+O(\alpha g_1(s_0))$
+$O((|\tau|+\alpha^2)g_2(s_0))+O((e^{-s_0}+\tau^2+\frac{|\tau|}{s_0}+|\tau\alpha|+|\alpha|^3)\|g(s_0)\|_{L^{\infty}}$);

(ii) $(\frac{\sigma_0}{s_0})^2\overline{g}_1(\sigma_0)=g_1(s_0)(1+\frac{(p+1)\tau}{2(p-1)})+2\alpha g_2(s_0)$
+$O((|\tau|+\alpha^2)g_3(s_0))+O((e^{-s_0}+\tau^2+\frac{|\tau|}{s_0}+|\tau\alpha|+|\alpha|^3)\|g(s_0)\|_{L^{\infty}}$);

(iii) $\overline{g}_2(\sigma_0)=g_2(s_0)+O(\alpha g_3(s_0))$
+$O(\alpha^2g_4(s_0))+O((e^{-s_0}+
|\tau|+|\alpha|^3)\|g(s_0)\|_{L^{\infty}}$);

(iv) $\sigma_0^{-2}\partial_\tau[\sigma_0^2\overline{g}_0](\sigma_0)=\frac{g_0(s_0)}{p-1}+O(g_2(s_0))+O((s_0^{-1}+
|\tau|+|\alpha|)\|g(s_0)\|_{L^{\infty}}$);

(v) $\partial_\tau[(\frac{\sigma_0}{s_0})^2\overline{g}_1](\sigma_0)=O(g_1(s_0))+O(g_3(s_0))+O((s_0^{-1}+
|\tau|+|\alpha|)\|g(s_0)\|_{L^{\infty}}$);

(vi) $\partial_\alpha \overline{g}_0(\sigma_0)=O(g_1(s_0))+O(\alpha g_2(s_0))+O((e^{-s_0}+
|\tau|+|\alpha|^2)\|g(s_0)\|_{L^{\infty}}$);

(vii) $\partial_\alpha \overline{g}_1(\sigma_0)=2g_2(s_0)+O(\alpha g_3(s_0))+O((e^{-s_0}+
|\tau|+|\alpha|^2)\|g(s_0)\|_{L^{\infty}}$);

(viii) $\|\frac{\overline{g}_{-}(z,\sigma_0)}{1+|z|^3}\|_{L^{\infty}(\mathbb{R})}\leq C\|\frac{{g}_{-}(\widetilde{z},s_0)}{1+|\widetilde{z}|^3}\|_{L^{\infty}(\mathbb{R})}
+C|\alpha|(|g_1(s_0)|+|g_2(s_0)|)+O((s_0^{-2}+
|\tau|+|\alpha|^3)\|g(s_0)\|_{L^{\infty}}$);

(ix) $\|\overline{g}_{e}(z,\sigma_0)\|_{L^{\infty}(\mathbb{R})}\leq \|{g}_{e}(\widetilde{z},s_0)\|_{L^{\infty}(\mathbb{R})}
+C(|\tau|+|\alpha|)\|g(s_0)\|_{L^{\infty}(\mathbb{R})}.$
\end{lemma}

By Lemma \ref{lemma g}, we can obtain the following lemma.

\begin{lemma}\label{lem 2.2} (Expansions of modes for $(T,a)$ close to $(\widehat{T},\widehat{a})$) For $s_0>0$ large enough, there exists $C_0(s_0)>0$ such that
for any
\begin{align}\label{lem 6.2}\|\widetilde{\varepsilon}_0\|_{W_0^{1,q}(\Omega)\cap W^{2,q}(\Omega)}\leq \frac{1}{C_0}, \ |\tau|\leq \frac{1}{16},\ |\alpha|\leq \frac{1}{4},
\end{align} where $q>n+2$, we have the following inequalities:
\begin{align}(i)\
&\Big|\frac{\sigma_0^2}{s_0^2}\overline{\psi}_0(s_0,y_0,T,a)-\widehat{q}_0(s_0)
-\frac{\kappa \tau}{p-1}\Big|\nonumber\\ \leq& C\|\widetilde{\varepsilon}_0\|_{W_0^{1,q}(\Omega)\cap W^{2,q}(\Omega)}+C\frac{\widehat{A}}{s_0^2}|{\tau}|
+C\frac{\widehat{A}}{s_0^2}|\alpha|+C\frac{\widehat{A}^2}{s_0^{\frac{1}{2}}}e^{-s_0}
+C\frac{\widehat{A}^2}{s_0^{\frac{1}{2}}}(|{\tau}|+|\alpha|^3)
\nonumber\\&+C\frac{\widehat{A}^2}{s_0^{2}}\log s_0(\tau
+|{\alpha}|^2)+Ce^{-s_0}+C{\tau}^2+C\frac{|\tau|}{s_0}+C\frac{{\alpha}^2}{s_0}
+C|\alpha|e^{-s_0/2};
\nonumber\end{align} \begin{align}(ii)\ &
\Big|\frac{\sigma_0^2}{s_0^2}\overline{\psi}_1(s_0,y_0,T,a)-\widehat{q}_1(s_0)
+\frac{2b\kappa}{(p-1)^2}\frac{\alpha }{s_0}\Big|\nonumber\\ \leq& C\|\widetilde{\varepsilon}_0\|_{W_0^{1,q}(\Omega)\cap W^{2,q}(\Omega)}+C\frac{\widehat{A}^2}{s_0^{\frac{1}{2}}}e^{-s_0}
+C\frac{\widehat{A}}{s_0^2}|{\tau}|+C\frac{\widehat{A}^2}{s_0^2}\log s_0|\alpha|\nonumber\\&+C\frac{\widehat{A}^2}{s_0^{\frac{1}{2}}}\Big(|{\tau}|^2+|\tau\alpha|+\frac{|\tau|}{s_0}+|\alpha|^3\Big)+
\Big(\frac{C\widehat{A}}{s_0^2}+e^{-Cs_0}\Big)(|\tau|+|\alpha|^2)\nonumber\\&+C\Big(\frac{|\alpha|}{s_0^2}+e^{-s_0}+|{\tau}|^2+|\tau\alpha|
+\frac{|\tau|}{s_0}+|\alpha|^3\Big)
+C|{\alpha}|e^{-s_0/2};
\nonumber\\ (iii)\ &|\overline{\psi}_2(s_0,y_0,T,a)|\leq C\|\widetilde{\varepsilon}_0\|_{W_0^{1,q}(\Omega)\cap W^{2,q}(\Omega)}+
\frac{C\widehat{A}}{s_0^{2}}|\alpha|+\frac{C\widehat{A}}{s_0^{2}}|\alpha|^2+\nonumber\\&+\frac{C\widehat{A}^2}{{s_0^\frac{1}{2}}}(|\tau|+|\alpha|^3)+
\frac{C\widehat{A}^2}{{s_0^\frac{1}{2}}}e^{-s_0}
+C(|\tau|+|\alpha|^3+e^{-s_0})
+\frac{C}{s_0^2}
\nonumber\\&+Ce^{-Cs_0}|\alpha|+Ce^{-Cs_0}|\alpha|^2+C\frac{|\tau|}{s_0^2}+\widehat{A}^2s_0^{-2}\log s_0+C|\alpha|e^{-s_0/2};\nonumber\\
 (iv)\ &\Big|\sigma_0^{-2}\partial_\tau[\sigma_0^2\overline{\psi}_0](s_0,y_0,T,a)-\frac{\kappa}{p-1}\Big|\nonumber\\ \leq&C\|\widetilde{\varepsilon}_0\|_{W_0^{1,q}(\Omega)\cap W^{2,q}(\Omega)}+
\frac{C\widehat{A}}{s_0^{2}}+\frac{C\widehat{A}^2}{s_0^{\frac{1}{2}}}(|\tau|+|\alpha|)+C(|\tau|+|\alpha|)+\frac{C}{s_0}
+\frac{C\widehat{A}^2}{s_0^{3/2}}\nonumber\\&+Ce^{-Cs_0}+Ce^{-Cs_0}|\alpha|+C\frac{\widehat{A}^2}{s_0^{2}}\log s_0;\nonumber\\
 (\mbox{v})\ &\Big|\partial_\tau\Big[\frac{\sigma_0^2}{s_0^2}\overline{\psi}_1\Big](s_0,y_0,T,a)\Big|\nonumber\\ \leq &C\|\widetilde{\varepsilon}_0\|_{W_0^{1,q}(\Omega)\cap W^{2,q}(\Omega)}+
\frac{C\widehat{A}}{s_0^{2}}+\frac{C\widehat{A}^2}{s_0^{\frac{1}{2}}}(s_0^{-1}+|\tau|+|\alpha|)\nonumber\\&+C(\frac{1}{s_0}+|\tau|+|\alpha|)
+Ce^{-Cs_0}+Ce^{-Cs_0}|\alpha|;\nonumber\\
(\mbox{vi})\ &|\partial_\alpha \overline{\psi}_0(s_0,y_0,T,a)|\leq C\|\widetilde{\varepsilon}_0\|_{W_0^{1,q}(\Omega)\cap W^{2,q}(\Omega)}+
\frac{C\widehat{A}}{s_0^{2}}+\frac{C\widehat{A}^2}{s_0^2}\log s_0|\alpha|+\nonumber\\&+\frac{C\widehat{A}^2}{{s_0^\frac{1}{2}}}(|\tau|+|\alpha|^2)+
\frac{C\widehat{A}^2}{{s_0^\frac{1}{2}}}e^{-s_0}
+C\Big(|\tau|+|\alpha|^2+e^{-s_0}+\frac{|\alpha|}{s_0}\Big)
\nonumber\\&+Ce^{-Cs_0}|\alpha|+Ce^{-Cs_0};\nonumber\\
(\mbox{vii})\ &|\partial_\alpha \overline{\psi}_1(s_0,y_0,T,a)+\frac{2b\kappa}{(p-1)^2s_0}|\nonumber
\\\leq&C\|\widetilde{\varepsilon}_0\|_{W_0^{1,q}(\Omega)\cap W^{2,q}(\Omega)}+
\frac{C\widehat{A}}{s_0^{2}}|\alpha|+\frac{C\widehat{A}^2}{s_0^2}\log s_0+\nonumber\\&+\frac{C\widehat{A}^2}{s_0^{\frac{1}{2}}}(|\tau|+|\alpha|^2)+
\frac{C\widehat{A}^2}{s_0^{\frac{1}{2}}}e^{-s_0}+\frac{C}{s_0^2}
+C(|\tau|+|\alpha|^2)
\nonumber\\&+Ce^{-Cs_0}|\alpha|+Ce^{-Cs_0};\nonumber\end{align}
\begin{align}
(viii)\ &\left\|\frac{\overline{\psi}_{-}(s_0,y_0,T,a)}{1+|z|^3}\right\|_{L^\infty(\Omega)} \nonumber\\
\leq
& C\|\widetilde{\varepsilon}_0\|_{W_0^{1,q}(\Omega)\cap W^{2,q}(\Omega)}+\frac{C\widehat{A}}{s_0^{2}}+\frac{C\widehat{A}}{s_0^{2}}|\alpha|+
\frac{C\widehat{A}^2}{s_0^{\frac{1}{2}}}(|\tau|+|\alpha|^3)
+C\frac{\widehat{A}^2}{s_0^{2}}\log s_0|\alpha|+\frac{C\widehat{A}^2}{s_0^{5/2}}\nonumber\\&+Ce^{-Cs_0}|\alpha|
+\frac{C}{s_0}|\tau|+C\frac{|\alpha|}{s_0^2}+C\frac{|\alpha|^2}{s_0^2}+C|\tau|^2+Ce^{-s_0};\nonumber\\
(ix)\ &\Big\|\overline{\psi}_{e}(s_0,y_0,T,a)\Big\|_{L^{\infty}(\mathbb{R})} \nonumber\\
\leq& C\|\widetilde{\varepsilon}_0\|_{W_0^{1,q}(\Omega)\cap W^{2,q}(\Omega)}
+\frac{C\widehat{A}^2}{s_0^{\frac{1}{2}}}(1+|\tau|+|\alpha|)
+C\Big(|\tau|+\frac{|\alpha|}{s_0^{\frac{1}{2}}}\Big)+\frac{C}{s_0}+Ce^{-Cs_0}|\alpha|.\nonumber
\end{align}
\end{lemma}

\begin{proof} From (\ref{qt}) and (\ref{zzz}), we write
$$\overline{\psi}(s_0,y_0,T,a,z)=\overline{\psi}^1(z,\sigma_0)+\overline{\psi}^2(z,\sigma_0)+\overline{\psi}^3(z,\sigma_0)+\overline{\psi}^4(z,\sigma_0)
+\overline{\psi}^5(z,\sigma_0)$$
for $z\in (\Omega-a)e^{\sigma_0/2}$, where
\begin{align}&\overline{\psi}^1(z,\sigma_0)=(1+\tau)^{\frac{1}{p-1}}e^{-\frac{s_0}{p-1}}
\widetilde{\varepsilon}(e^{-\frac{s_0}{2}}\widetilde{z}+\widehat{a},\widehat{T}-e^{-s_0})\chi(z,\sigma_0);\label{psibar1} \\
&\overline{\psi}^2(z,\sigma_0)=(1+\tau)^{\frac{1}{p-1}}\widehat{q}(\widetilde{z},s_0);\label{psibar2} \\
&\overline{\psi}^3(z,\sigma_0)=(1+\tau)^{\frac{1}{p-1}}\varphi(\widetilde{z},s_0);
\label{psibar3} \\
&\overline{\psi}^4(z,\sigma_0)=-\varphi(z,\sigma_0);
\label{psibar4} \\
&\overline{\psi}^5(z,\sigma_0)=e^{-\frac{\sigma_0}{p-1}}\widehat{y}
(ze^{-\frac{\sigma_0}{2}}+a,T-e^{-\sigma_0}))[\chi(z,\sigma_0)-\chi(\widetilde{z},s_0)],\label{psibar5}
\end{align}
and for $z\not\in (\Omega-a)e^{\sigma_0/2}$,
\begin{align}\overline{\psi}^1=\overline{\psi}^2=\overline{\psi}^3=\overline{\psi}^5=0,\
\overline{\psi}^4(z,\sigma_0)=-\varphi(z,\sigma_0)\label{psibar7}. \end{align}

{\bf Step 1. Expansion of $\overline{\psi}^1$}

\medskip

When $t=\widehat{T}-e^{-s_0}$,  $\widehat{y}$ is meaningful in $[0,\widehat{T}-e^{-s_0}]$. Thus,
${y}_{y_0}$is meaningful in $[0,\widehat{T}-e^{-s_0}]$ if $y_0$ is close enough to $\widehat{y}_0$ in $W_0^{1,q}(\Omega)\cap W^{2,q}(\Omega)$ with $q>n+2$,
which implies $t<\min(T(y_0),\widehat{T})$
and it follows that
\begin{align}\label{tilde}\|\widetilde{\varepsilon}(\widehat{T}-e^{-s_0})\|_{L^\infty(\Omega)}\leq \widetilde{C}_0\|\widetilde{\varepsilon}_0\|_{W_0^{1,q}(\Omega)\cap W^{2,q}(\Omega)},\end{align}
 whenever $\|\widetilde{\varepsilon}_0\|_{W_0^{1,q}(\Omega)\cap W^{2,q}(\Omega)}\leq 1/\widetilde{C}_0$ for some
$\widetilde{C}_0=\widetilde{C}_0(s_0)$, where $q>n+2$.

\medskip

\medskip

Now, we prove that for $s_0$ large enough and $(\tau,\alpha)$ satisfies (\ref{lem 6.2}), we have
\begin{align}\label{1}&|\overline{\psi}_0^1(\sigma_0)|+|\overline{\psi}_1^1(\sigma_0)|+|\overline{\psi}_2^1(\sigma_0)|+
\Big|\partial_\tau\Big[\Big(\frac{\sigma_0}{s_0}\Big)^2\overline{\psi}_0^1\Big](\sigma_0)\Big|
+\Big|\partial_\tau\Big[\Big(\frac{\sigma_0}{s_0}\Big)^2\overline{\psi}_1^1\Big](\sigma_0)\Big|\nonumber\\
&+|\partial_\alpha \overline{\psi}_0^1(\sigma_0)|+|\partial_\alpha \overline{\psi}_1^1(\sigma_0)| +\left\|\frac{\overline{\psi}_{-}^1(z,\sigma_0)}{1+|z|^3}\right\|_{L^\infty(\mathbb{R})}+\|\overline{\psi}_e^1(\sigma_0)\|_{L^\infty(\mathbb{R})}\nonumber\\
\leq &C\|\widetilde{\varepsilon}(\widehat{T}-e^{-s_0})\|_{L^{\infty}(\Omega)}\leq C_0(s_0)\|\widetilde{\varepsilon}_0\|_{
{W_0^{1,q}(\Omega)\cap W^{2,q}(\Omega)}}.
\end{align}

Indeed, for $\overline{\psi}_0^1(\sigma_0)$, it holds that
\begin{align}|\overline{\psi}_0^1(\sigma_0)|=&\Big|\int \chi_1(z,\sigma_0)\overline{\psi}^1(z,\sigma_0)\rho(z)dz\Big|\nonumber\\=&\Big|\int_{(\Omega-a)e^{\sigma_0/2}}\rho(z)(1+\tau)^{\frac{1}{p-1}}e^{-\frac{s_0}{p-1}}
\widetilde{\varepsilon}(e^{-\frac{s_0}{2}}\widetilde{z}+\widehat{a},\widehat{T}-e^{-s_0})\chi_1(z,\sigma_0)\chi(z,\sigma_0)dz\Big|.\nonumber\end{align}
It is easy to see that $z\in (\Omega-a)e^{\sigma_0/2}$ is equivalent to  $\widetilde{z}\in (\Omega-\widehat{a})e^{s_0/2}$.
Hence, \begin{align}\label{psi10}|\overline{\psi}_0^1(\sigma_0)|\leq C\|\widetilde{\varepsilon}(\widehat{T}-e^{-s_0})\|_{L^\infty(\Omega)}.\end{align}

Similarly, we have
\begin{align}|\overline{\psi}_1^1(\sigma_0)|
\leq C\|\widetilde{\varepsilon}(\widehat{T}-e^{-s_0})\|_{L^{\infty}(\Omega)},
\end{align}
 \begin{align}|\overline{\psi}_2^1(\sigma_0)|
\leq C\|\widetilde{\varepsilon}(\widehat{T}-e^{-s_0})\|_{L^{\infty}(\Omega)},
\end{align}
and
\begin{align}\label{psiee}\|\overline{\psi}_e^1(\sigma_0)\|_{{L^\infty(\mathbb{R})}}=\|(1-\chi_1(z,\sigma_0))\overline{\psi}^1(\sigma_0)\|_{{L^\infty(\mathbb{R})}}
\leq  C\|\widetilde{\varepsilon}(\widehat{T}-e^{-s_0})\|_{L^{\infty}(\Omega)}.
\end{align}


By the above estimates,
\begin{align}\label{psinegative}&\left\|\frac{\overline{\psi}^1_{-}(z,\sigma_0)}{1+|z|^3}\right\|_{L^{\infty}(\mathbb{R})}=\left\|\frac{\overline{\psi}^1(z,\sigma_0)\chi_1(z,\sigma_0)-\sum\limits_{m=0}^2 \overline{\psi}_m^1(\sigma_0)h_m(z)}{1+|z|^3}\right\|_{L^{\infty}(\mathbb{R})}\nonumber\\ \leq &C\|\widetilde{\varepsilon}(\widehat{T}-e^{-s_0})\|_{L^{\infty}(\Omega)}.\end{align}

For $\frac{\partial}{\partial\tau}\Big[\Big(\frac{\sigma_0}{s_0}\Big)^2\overline{\psi}_0^1\Big](\sigma_0)$, we have

\begin{align}\Big(\frac{\sigma_0}{s_0}\Big)^2\overline{\psi}_0^1(\sigma_0)= & \Big(\frac{s_0-\log(1+\tau)}{s_0}\Big)^2\int_{\mathbb{R}}\chi_1(z,\sigma_0){\overline{\psi}}^1(z,\sigma_0)\rho(z)dz \nonumber\\
 =&\Big(\frac{s_0-\log(1+\tau)}{s_0}\Big)^2(1+\tau)^{-\frac{1}{2}}\int_{\mathbb{R}}\chi_1(z,s_0-\log(1+\tau)){\overline{\psi}}^1
 (z,\sigma_0)\rho(z)d\widetilde{z}.\nonumber
\end{align}

We see that
\begin{align}{\overline{\psi}}^1(z,\sigma_0)=\left\{\begin{array}{ll}\label{ru} (1+\tau)^{\frac{1}{p-1}}e^{-\frac{s_0}{p-1}}
\widetilde{\varepsilon}(e^{-\frac{s_0}{2}}\widetilde{z}+\widehat{a},\widehat{T}-e^{-s_0})\chi(z,\sigma_0),\ z\in (\Omega-a)e^{\sigma_0/2},\\
0,\ \mbox{otherwise}.\nonumber
\end{array}\right.\end{align}
However, we can take $s_0$ large enough such that $|a-\widehat{a}|=|\alpha e^{-\frac{s_0}{2}}|\leq  e^{-\frac{s_0}{2}}<\frac{\varepsilon_0}{4}$. Thus, since
$(\widehat{a}-4\widehat{\delta},\widehat{a}+4\widehat{\delta})\subset\omega$, $\varepsilon_0\in(0,\widehat{\delta}]$ and $\chi(z,\sigma_0)$ is 0 outside a small neighborhood of $a$ when $|a-\widehat{a}|<\frac{\varepsilon_0}{4}$, it is easy to check that ${\overline{\psi}}^1(z,\sigma_0)$ is smooth in $\mathbb{R}$.

\medskip

Then, we have that for $s_0$ large enough,
\begin{align}&\frac{\partial}{\partial\tau}\Big[\Big(\frac{\sigma_0}{s_0}\Big)^2\overline{\psi}_0^1\Big](\sigma_0)=
-2\frac{1}{s_0(1+\tau)}\frac{s_0-\log(1+\tau)}{s_0}
\int_{(\Omega-a)e^{\sigma_0/2}}\rho(z)(1+\tau)^{\frac{1}{p-1}}e^{-\frac{s_0}{p-1}}\nonumber\\&\cdot
\widetilde{\varepsilon}(e^{-\frac{s_0}{2}}(z\sqrt{1+\tau}+\alpha)+\widehat{a},\widehat{T}-e^{-s_0})\chi_1(z,s_0-\log(1+\tau))\chi(z,\sigma_0)dz
\nonumber\\&-\frac{1}{2}(1+\tau)^{-1}\Big(\frac{s_0-\log(1+\tau)}{s_0}\Big)^2
\int_{(\Omega-a)e^{\sigma_0/2}}\rho(z)\chi_1(z,s_0-\log(1+\tau))(1+\tau)^{\frac{1}{p-1}}e^{-\frac{s_0}{p-1}}\nonumber\\&\cdot
\widetilde{\varepsilon}(e^{-\frac{s_0}{2}}(z\sqrt{1+\tau}+\alpha)+\widehat{a},\widehat{T}-e^{-s_0})
\chi(z,\sigma_0)dz
\nonumber\\&+\frac{1}{\sqrt{4\pi}}\Big(\frac{s_0-\log(1+\tau)}{s_0}\Big)^2
\int_{(\Omega-a)e^{\sigma_0/2}}\rho(z)\frac{z^2}{4}(1+\tau)^{\frac{1}{p-1}-1}\chi_1(z,s_0-\log(1+\tau))e^{-\frac{s_0}{p-1}}\nonumber\\&\cdot
\widetilde{\varepsilon}(e^{-\frac{s_0}{2}}(z\sqrt{1+\tau}+\alpha)+\widehat{a},\widehat{T}-e^{-s_0})
\chi(z,\sigma_0)dz
\nonumber\\&-\frac{1}{2}\Big(\frac{s_0-\log(1+\tau)}{s_0}\Big)^2
\int_{(\Omega-a)e^{\sigma_0/2}}\rho(z)z(1+\tau)^{\frac{1}{p-1}-1}\chi_{1z}(z,s_0-\log(1+\tau))e^{-\frac{s_0}{p-1}}\nonumber\\&\cdot
\widetilde{\varepsilon}(e^{-\frac{s_0}{2}}(z\sqrt{1+\tau}+\alpha)+\widehat{a},\widehat{T}-e^{-s_0})
\chi(z,\sigma_0)dz
\nonumber\\&-\frac{1}{2}\Big(\frac{s_0-\log(1+\tau)}{s_0}\Big)^2
\int_{(\Omega-a)e^{\sigma_0/2}}\rho(z)(1+\tau)^{\frac{1}{p-1}-1}\chi_1(z,s_0-\log(1+\tau))e^{-\frac{s_0}{p-1}}\nonumber\\&\cdot
\widetilde{\varepsilon}(e^{-\frac{s_0}{2}}(z\sqrt{1+\tau}+\alpha)+\widehat{a},\widehat{T}-e^{-s_0})
\chi_z(z,\sigma_0)zdz
\nonumber\\&+\Big(\frac{s_0-\log(1+\tau)}{s_0}\Big)^2
\int_{(\Omega-a)e^{\sigma_0/2}}\rho(z){\frac{1}{p-1}}(1+\tau)^{\frac{1}{p-1}-1}\chi_1(z,s_0-\log(1+\tau))
e^{-\frac{s_0}{p-1}}\nonumber\\&\cdot
\widetilde{\varepsilon}(e^{-\frac{s_0}{2}}(z\sqrt{1+\tau}+\alpha)+\widehat{a},\widehat{T}-e^{-s_0})\chi(z,\sigma_0)dz
\nonumber\\&+\Big(\frac{s_0-\log(1+\tau)}{s_0}\Big)^2
\int_{(\Omega-a)e^{\sigma_0/2}}\rho(z)(1+\tau)^{\frac{1}{p-1}}e^{-\frac{s_0}{p-1}}\nonumber\\&\cdot
\widetilde{\varepsilon}(e^{-\frac{s_0}{2}}(z\sqrt{1+\tau}+\alpha)+\widehat{a},\widehat{T}-e^{-s_0})
\chi_1(z,\sigma_0)\frac{\chi_{\sigma_0}(z,\sigma_0)}{-(1+\tau)}dz\nonumber\\&+\Big(\frac{s_0-\log(1+\tau)}{s_0}\Big)^2
\int_{(\Omega-a)e^{\sigma_0/2}}\rho(z)(1+\tau)^{\frac{1}{p-1}}e^{-\frac{s_0}{p-1}}\nonumber\\&\cdot
\widetilde{\varepsilon}(e^{-\frac{s_0}{2}}(z\sqrt{1+\tau}+\alpha)+\widehat{a},\widehat{T}-e^{-s_0})
\chi(z,\sigma_0)\frac{\chi_{1\sigma_0}(z,\sigma_0)}{-(1+\tau)}dz,\end{align}
from which and the following equalities,
\begin{align}\label{tilde1}\chi_{1\sigma_0}(z,\sigma_0)=-\chi_0'\Big(\frac{|z|}{K_0\sqrt{\sigma_0}}\Big)\frac{|z|}{K_0}\frac{\sigma_0^{-3/2}}{2},\ \chi_{1z}(z,\sigma_0)=\chi_0'\Big(\frac{|z|}{K_0\sqrt{\sigma_0}}\Big)\frac{1}{K_0\sqrt{\sigma_0}},\end{align}
\begin{align}\label{tilde2}\chi_{\sigma_0}(z,\sigma_0)=-\chi_0'\Big(\frac{ze^{-\sigma_0/2}}{\varepsilon_0}\Big)\frac{z}{\varepsilon_0}\frac{e^{-\sigma_0/2}}{2},\
\chi_{z}(z,\sigma_0)=\chi_0'\Big(\frac{ze^{-\sigma_0/2}}{\varepsilon_0}\Big)\frac{e^{-\sigma_0/2}}{\varepsilon_0},\end{align}
we have that for $s_0$ large enough,
\begin{align}\label{partial tao 0} \Big|\frac{\partial}{\partial\tau}\Big[\Big(\frac{\sigma_0}{s_0}\Big)^2\overline{\psi}_0^1\Big](\sigma_0)\Big|\leq C\|\widetilde{\varepsilon}_0(\widehat{T}-e^{-s_0})\|_{L^{\infty}(\Omega)}.\end{align}

Similarly to the above inequality,  we have  that for $s_0$ large enough, \begin{align} \Big|\frac{\partial}{\partial\tau}\Big[\Big(\frac{\sigma_0}{s_0}\Big)^2\overline{\psi}_1^1\Big](\sigma_0)\Big|\leq C\|\widetilde{\varepsilon}_0(\widehat{T}-e^{-s_0})\|_{L^{\infty}(\Omega)},\end{align}
\begin{align}|\partial_\alpha \overline{\psi}^1_0(\sigma_0)|\leq C\|\widetilde{\varepsilon}_0(\widehat{T}-e^{-s_0})\|_{L^{\infty}(\Omega)},\end{align}
and
\begin{align}\label{partial alpha} |\partial_\alpha \overline{\psi}^1_1(\sigma_0)|\leq C\|\widetilde{\varepsilon}_0(\widehat{T}-e^{-s_0})\|_{L^{\infty}(\Omega)}.\end{align}

Then, by (\ref{psi10})-(\ref{psinegative}), (\ref{partial tao 0})-(\ref{partial alpha}) and (\ref{tilde}), we can get (\ref{1}).

\medskip

{\bf Step 2. Expansion of $\overline{\psi}^2$}

\medskip

Since
for any $t\in[0,\widehat{T})$, $\widehat{y}(t)\in S^*({K}_0,{\varepsilon}_0,\widehat{A},1/2,\widehat{\eta}_0,\widehat{T},t)$,
where $\widehat{q}(s)\in V_{{K}_0,\widehat{A}}(s)$, for all $s\geq-\log \widehat{T}$, and  $|\widehat{y}(x,t)|\leq \widehat{\eta}_0$ for all $x\in \Omega$ with $|x-\widehat{a}|\geq\frac{{\varepsilon}_0}{2}$.  In particular, for $s_0\geq-\log \widehat{T}$, we have
\begin{align}\left\{\begin{array}{ll}\label{ru1} |\widehat{q}_m(s_0)|\leq \widehat{A}s_0^{-2} (m=0, 1),\ |\widehat{q}_2(s_0)|\leq \widehat{A}^2s_0^{-2}\log s_0,\\|\widehat{q}_{-}(\widetilde{z},s_0)|\leq \widehat{A}s_0^{-2}(1+|\widetilde{z}|^3),\ \ \ \ |\widehat{q}_e(\widetilde{z},s_0)|\leq \widehat{A}^2s_0^{-1/2},
\end{array}\right.\end{align}
and
\begin{align}\label{ru2}\|\widehat{q}(s_0)\|_{L^\infty}\leq C\widehat{A}^2s_0^{-1/2}.\end{align}

(i) For $s_0$ large enough,
\begin{align}\label{psi2inequality}\Big|\Big(\frac{\sigma_0}{s_0}\Big)^2\overline{\psi}^2_0(\sigma_0)-\widehat{q}_0(s_0)\Big|\leq&\frac{C\widehat{A}}{s_0^{2}}|\tau|+\frac{C\widehat{A}}{s_0^{2}}|\alpha|+
\frac{C\widehat{A}^2e^{-s_0}}{s_0^{\frac{1}{2}}}+\frac{C\widehat{A}^2}{s_0^{\frac{1}{2}}}(|\tau|+|\alpha|^3)
\nonumber\\&+C(|\tau|+|\alpha|^2)\frac{\widehat{A}^2}{s_0^{2}}\log s_0.\end{align}

Indeed, by (\ref{psibar2}) and (\ref{psibar7}),
\begin{align}\overline{\psi}^2(z,\sigma_0)=\left\{\begin{array}{ll}\label{ru} (1+\tau)^{\frac{1}{p-1}}\widehat{q}(\widetilde{z},s_0),\ \widetilde{z}\in (\Omega-\widehat{a})e^{s_0/2},\\0,
\ \widetilde{z}\not\in (\Omega-\widehat{a})e^{s_0/2},
\end{array}\right.\end{align}
we have by Lemma \ref{lemma g} (i) that for $s_0$ large enough,
\begin{align}\label{psi20}&\Big|\Big(\frac{\sigma_0}{s_0}\Big)^2\overline{\psi}^2_0(\sigma_0)-\widehat{q}_0(s_0)\Big|=\Big|\widehat{q}_0(s_0)\frac{\tau}{p-1}+O(\alpha \widehat{q}_1(s_0))\nonumber\\
&+O((|\tau|+\alpha^2)\widehat{q}_2(s_0))+O\Big(\Big(e^{-s_0}+\tau^2+\frac{|\tau|}{s_0}+|\tau\alpha|+|\alpha|^3\Big)\|g^1(s_0)\|_{L^{\infty}}\Big)\Big|,\end{align}
where
\begin{align}{g}^1(\widetilde{z},s_0)=\left\{\begin{array}{ll}\label{g2} \widehat{q}(\widetilde{z},s_0),\
\widetilde{z}\in (\Omega-\widehat{a})e^{s_0/2}(\Longleftrightarrow z\in (\Omega-a)e^{\sigma_0/2}),\\0,
\ \widetilde{z}\not\in (\Omega-\widehat{a})e^{s_0/2}.
\end{array}\right.\end{align}
Hence, by (\ref{ru1}), (\ref{ru2}) and (\ref{psi20}), (\ref{psi2inequality}) hold.


\medskip


\medskip

(ii) For $s_0$ large enough,
\begin{align}\label{psi3inequality}&\Big|\Big(\frac{\sigma_0}{s_0}\Big)^2\overline{\psi}^2_1(\sigma_0)-\widehat{q}_1(s_0)\Big|\nonumber\\\leq&
\frac{C\widehat{A}^2e^{-s_0}}{s_0^{\frac{1}{2}}}+\frac{C\widehat{A}^2}{s_0^{\frac{1}{2}}}\Big(|\tau|^2+|\alpha|^3+|\tau\alpha|+\frac{|\tau|}{s_0}\Big)
+\frac{C\widehat{A}}{s_0^{2}}|\tau|+\frac{C\widehat{A}^2}{s_0^{2}}\log s_0|\alpha|\nonumber\\&+(\frac{C\widehat{A}}{s_0^{2}}
+e^{-Cs_0})(|\tau|+|\alpha|^2).\end{align}

Indeed, since $k_0$, $k_1$, $k_2$ are all orthogonal to $k_3$, we have
\begin{align}&|\widehat{q}_3(s_0)|\nonumber\\
=&\Big|\int_{\mathbb{R}} \widehat{q}_{-}(\widetilde{z},s_0)k_3(\widetilde{z})\chi_1(\widetilde{z},s_0)\rho(\widetilde{z})d\widetilde{z}
+\int_{\mathbb{R}}\sum\limits_{m=0}^2 q_m(s_0)h_m(\widetilde{z})k_3(\widetilde{z})\chi_1(\widetilde{z},s_0)\rho(\widetilde{z})d\widetilde{z}\Big|
\nonumber\\
\leq &C\frac{\widehat{A}}{s_0^2}+Ce^{-Cs_0}.\label{q33}\end{align}
Similarly, we have
\begin{align}\label{q44}|\widehat{q}_4(s_0)|\leq &C\frac{\widehat{A}}{s_0^2}+Ce^{-Cs_0}.\end{align}
Then, we have by Lemma \ref{lemma g} (ii) that for $s_0$ large enough,
\begin{align}\label{3}&\Big|\Big(\frac{\sigma_0}{s_0}\Big)^2\overline{\psi}^2_1(\sigma_0)-\widehat{q_1}(s_0)\Big|\nonumber\\=&\Big|\widehat{q}_1(s_0)
\frac{(p+1)\tau}{2(p-1)}+2\alpha \widehat{q}_2(s_0)
+O((|\tau|+\alpha^2)\widehat{q}_3(s_0))\nonumber\\ &+O\Big(\Big(e^{-s_0}+\tau^2+\frac{|\tau|}{s_0}+|\tau\alpha|+|\alpha|^3)\|g^2(s_0)\|_{L^{\infty}}\Big)\Big|.\end{align}
Hence, by (\ref{ru1}), (\ref{ru2}), (\ref{g2})  and (\ref{q33}), (\ref{psi3inequality}) holds.

\medskip

(iii) For $s_0$ large enough,
\begin{align}\label{psi4inequality}|\overline{\psi}_2^2(\sigma_0)|\leq &
\frac{C\widehat{A}^2e^{-s_0}}{s_0^{\frac{1}{2}}}+\frac{C\widehat{A}^2}{s_0^{\frac{1}{2}}}(|\tau|+|\alpha|^3)
+\frac{C\widehat{A}}{s_0^{2}}|\alpha|+\frac{C\widehat{A}}{s_0^{2}}|\alpha|^2\nonumber\\&+Ce^{-Cs_0}|\alpha|+Ce^{-Cs_0}|\alpha|^2+\frac{\widehat{A}^2}{s_0^{2}}\log s_0.\end{align}

Indeed, by Lemma \ref{lemma g} (iii) that for $s_0$ large enough,
\begin{align}\label{4}|\overline{\psi}_2^2(\sigma_0)|=&\Big|\widehat{q}_2(s_0)+O(\alpha \widehat{q}_3(s_0))
+O(\alpha^2\widehat{q}_4(s_0))+O((e^{-s_0}+
|\tau|+|\alpha|^3)\|g^2(s_0)\|_{L^{\infty}})\Big|.\end{align}
Hence, by (\ref{ru1}), (\ref{ru2}), (\ref{g2}), (\ref{q33}) and (\ref{q44}), (\ref{psi4inequality}) holds.

\medskip

(iv) For $s_0$ large enough,
\begin{align}\label{psi5inequality}|\sigma_0^{-2}\partial_\tau[\sigma_0^2\overline{\psi}_0^2](\sigma_0)|\leq
\frac{C\widehat{A}^2}{s_0^{\frac{1}{2}}}(|\tau|+|\alpha|)
+\frac{C\widehat{A}}{s_0^{2}}+\frac{C\widehat{A}^2}{s_0^{3/2}}+\frac{C\widehat{A}^2}{s_0^{2}}\log s_0.\end{align}

Indeed, by Lemma \ref{lemma g} (iv),
\begin{align}|\sigma_0^{-2}\partial_\tau[\sigma_0^2\overline{\psi}_0^2](\sigma_0)|=\Big|\frac{\widehat{q}_0(s_0)}{p-1}
+O(\widehat{q}_2(s_0))+O((s_0^{-1}+
|\tau|+|\alpha|)\|g^2(s_0)\|_{L^{\infty}}\Big|.\nonumber\end{align}
Hence, by (\ref{ru1}), (\ref{ru2}) and (\ref{g2}), (\ref{psi5inequality}) holds.

\medskip

(v) For $s_0$ large enough,
\begin{align}\label{psi6inequality}|\sigma_0^{-2}\partial_\tau[\sigma_0^2\overline{\psi}_1^2](\sigma_0)|
\leq
\frac{C\widehat{A}^2}{s_0^{\frac{1}{2}}}(s_0^{-1}+|\tau|+|\alpha|)
+\frac{C\widehat{A}}{s_0^{2}}+Ce^{-Cs_0}.\end{align}

Indeed, by Lemma \ref{lemma g} (v),
\begin{align}&|\sigma_0^{-2}\partial_\tau[\sigma_0^2\overline{\psi}_1^2](\sigma_0)|
=\Big|\sigma_0^{-2}s_0^2\partial_\tau\Big[(\frac{\sigma_0}{s_0})^2\overline{\psi}_1^2\Big](\sigma_0)\Big|
\nonumber\\=&\Big|\frac{s_0^2}{\sigma_0^2}[O(\widehat{q}_1(s_0))+O(\widehat{q}_3(s_0))+O((s_0^{-1}+
|\tau|+|\alpha|)\|g^2(s_0)\|_{L^{\infty}}]\Big|.\nonumber\end{align}
Hence, by (\ref{ru1}), (\ref{ru2}), (\ref{g2})  and (\ref{q33}), (\ref{psi6inequality}) holds.

\medskip

(vi) For $s_0$ large enough,
 \begin{align}\label{psi7inequality}|\partial_\alpha \overline{\psi}_0^2(\sigma_0)| \leq \frac{C\widehat{A}}{s_0^{2}}+
\frac{C\widehat{A}^2e^{-s_0}}{s_0^{\frac{1}{2}}}+\frac{C\widehat{A}^2}{s_0^{\frac{1}{2}}}(|\tau|+|\alpha|^2)
+C\frac{\widehat{A}^2}{s_0^{2}}\log s_0|\alpha|.\end{align}

 Indeed, by Lemma \ref{lemma g} (vi),
\begin{align}|\partial_\alpha \overline{\psi}_0^2(\sigma_0)|=|O(\widehat{q}_1(s_0))+O(\alpha \widehat{q}_2(s_0))+O((e^{-s_0}+
|\tau|+|\alpha|^2)\|g^2(s_0)\|_{L^{\infty}}|.\end{align}
Hence, by (\ref{ru1}), (\ref{ru2}) and (\ref{g2}), (\ref{psi7inequality}) holds.

\medskip

(vii) For $s_0$ large enough,
\begin{align}\label{psi8inequality}|\partial_\alpha \overline{\psi}_1^2(\sigma_0)|\leq \frac{C\widehat{A}}{s_0^{2}}|\alpha|+
\frac{C\widehat{A}^2e^{-s_0}}{s_0^{\frac{1}{2}}}+\frac{C\widehat{A}^2}{s_0^{\frac{1}{2}}}(|\tau|+|\alpha|^2)
+C\frac{\widehat{A}^2}{s_0^{2}}\log s_0+Ce^{-Cs_0}|\alpha|.\end{align}

 Indeed, by Lemma \ref{lemma g} (vii),
\begin{align}\label{8}|\partial_\alpha \overline{\psi}_1^2(\sigma_0)|=|2\widehat{q}_2(s_0)+O(\alpha \widehat{q}_3(s_0))+O((e^{-s_0}+
|\tau|+|\alpha|^2)\|g^2(s_0)\|_{L^{\infty}}|.\end{align}
Hence, by (\ref{ru1}), (\ref{ru2}), (\ref{g2})  and (\ref{q33}), (\ref{psi8inequality}) holds.

\medskip

(viii) For $s_0$ large enough,
\begin{align}\label{psi9inequality}&\left\|\frac{\overline{\psi}_{-}^2(z,\sigma_0)}{1+|z|^3}\right\|_{L^{\infty}(\mathbb{R})}\nonumber\\ \leq
& \frac{C\widehat{A}}{s_0^{2}}(|\alpha|+1)+
\frac{C\widehat{A}^2}{s_0^{\frac{5}{2}}}+\frac{C\widehat{A}^2}{s_0^{\frac{1}{2}}}(|\tau|+|\alpha|^3)
+C\frac{\widehat{A}^2}{s_0^{2}}\log s_0|\alpha|.\end{align}

 Indeed, by Lemma \ref{lemma g} (viii),
\begin{align}\label{9}\left\|\frac{\overline{\psi}_{-}^2(z,\sigma_0)}{1+|z|^3}\right\|_{L^{\infty}(\mathbb{R})}\leq & C\left\|\frac{{\widehat{q}}_{-}(\widetilde{z},s_0)}{1+|\widetilde{z}|^3}\right\|_{L^{\infty}(\mathbb{R})}
+C|\alpha|(|\widehat{q}_1(s_0)|+|\widehat{q}_2(s_0)|)\nonumber\\&+O((s_0^{-2}+
|\tau|+|\alpha|^3)\|g^2(s_0)\|_{L^{\infty}}).\end{align}
Hence, by (\ref{ru1}), (\ref{ru2}) and (\ref{g2}), (\ref{psi9inequality}) holds.

\medskip

(ix) For $s_0$ large enough, by  Lemma \ref{lemma g} (ix), (\ref{ru1}), (\ref{ru2}) and (\ref{g2}),
\begin{align}\label{10}&\|\overline{\psi}_{e}^2(\sigma_0)\|_{L^{\infty}(\mathbb{R})}\leq \|{\widehat{q}}_{e}(\widetilde{z},s_0)\|_{L^{\infty}(\mathbb{R})}
+C(|\tau|+|\alpha|)\|g^2(s_0)\|_{L^{\infty}}\nonumber\\ \leq
& \frac{C\widehat{A}^2}{s_0^{\frac{1}{2}}}(|\tau|+|\alpha|)
+\frac{C\widehat{A}^2}{s_0^{1/2}}.\end{align}

{\bf Step 3. Expansion of $\overline{\psi}^3$ and  $\overline{\psi}^4$}

\medskip

It is easy to check that
\begin{align}\label{ff}
\varphi(z,s)=&\Big(p-1+\frac{(p-1)^2}{4p}\frac{z^2}{s}\Big)^{-\frac{1}{p-1}}+\frac{(p-1)^{-\frac{1}{p-1}}}{2ps}\nonumber\\
=&\kappa-\frac{\kappa}{4ps}(|z|^2-2)+O\Big(\frac{|z|^4}{s^2}\Big).
\end{align}

Recalling that $k_2(z)=\frac{z^2-2}{8}$, it follows that
\begin{align}\label{f1}
\int \varphi(z,s)k_2(z)\rho(z)dz
=&\int \Big[\kappa-\frac{\kappa}{4ps}(|z|^2-2)+O\Big(\frac{|z|^4}{s^2}\Big)\Big]\frac{z^2-2}{8}\rho(z)dz\nonumber\\
=&-\frac{\kappa}{4ps}+O(s^{-2}),
\end{align}
and we have that for $s_0$ large enough,
\begin{align}\label{psi222}&\varphi_0(s)=\kappa+O(s^{-2}),\ \varphi_2(s)=-\frac{\kappa}{4ps}+O(s^{-2}),\\
\label{psi134}&\varphi_1(s)=\varphi_3(s)=0,\ \varphi_4(s)=O(s^{-2}),\ |\varphi(s)|\leq C.\end{align}

Define
\begin{align}{g}^2(\widetilde{z},s_0)=\left\{\begin{array}{ll}{\varphi}(\widetilde{z},s_0),\ \widetilde{z}\in (\Omega-\widehat{a})e^{s_0/2},\\0,
\ \widetilde{z}\not\in (\Omega-\widehat{a})e^{s_0/2}.
\end{array}\right.\end{align}
Then if $s_0$ is large enough, we have \begin{align}\label{g21234} g^2_i(s_0)=\varphi_i(s_0), i=0,\ 1, \ 2, \ 3,\ 4, \
\|g^2(s_0)\|\leq C.\end{align}

Let \begin{align}{\eta}({z},\sigma_0)=\left\{\begin{array}{ll}\label{ru} 0,\ {z}\in (\Omega-a)e^{\sigma_0/2},\\(1+\tau)^{\frac{1}{p-1}}{\varphi}(\widetilde{z},s_0),
\ {z}\not\in (\Omega-a)e^{\sigma_0/2}.
\end{array}\right.\end{align}
By (\ref{psibar3}) and (\ref{psibar7}), it holds that $\overline{\psi}^3(z,\sigma_0)+{\eta}(z,\sigma_0)=(1+\tau)^{\frac{1}{p-1}}\varphi(\widetilde{z},s_0),\ {z}\in \mathbb{R}$.

Since $|a-\widehat{a}|
=|\alpha|e^{-s_0/2}\leq e^{-s_0/2}$, $(\widehat{a}-4\widehat{\delta}, \widehat{a}+4\widehat{\delta})\subset \omega$ and $\widehat{\varepsilon}_0\in(0,\widehat{\delta}]$, we have that when $s_0$ is large enough, which implies $\sigma_0$ is large enough, if $z\not\in (\Omega-a)e^{\sigma_0/2}$, then $|z|\geq 2K_0\sqrt{\sigma_0}$.
Thus, $$\eta\chi_1(z,\sigma_0)=0.$$
Hence, we have that for $s_0$ large enough,
$$\overline{\psi}_0^3(\sigma_0)=\int_{\mathbb{R}} (1+\tau)^{\frac{1}{p-1}}\varphi({z}\sqrt{1+\tau}+\alpha,\sigma_0+\log(1+\tau))\chi_1(z,\sigma_0)\rho(z)dz.$$

On the other hand, by (\ref{psibar4}) and (\ref{psibar7}),
$$\overline{\psi}_0^4(\sigma_0)=-\int_{\mathbb{R}} \varphi(z,\sigma_0)\chi_1(z,\sigma_0)\rho(z)dz.$$

Thus, we can get that for $s_0$ large enough,
\begin{align}\label{11}&\Big|\Big(\frac{\sigma_0}{s_0}\Big)^2\overline{\psi}^3_0(\sigma_0)+\Big(\frac{\sigma_0}{s_0}\Big)^2\overline{\psi}^4_0(\sigma_0)
-\frac{\kappa \tau}{p-1}\Big|\nonumber\\ \leq & Ce^{-\sigma_0}+C\frac{|\tau|}{\sigma_0}+C\frac{\alpha^2}{\sigma_0}+C\tau^2.
\end{align}

Remember $b=\frac{(p-1)^2}{4p}$, $\frac{2b\kappa}{(p-1)^2}=\frac{\kappa}{2p}$.
By  Lemma \ref{lemma g} (ii), (\ref{psi222}), (\ref{psi134}) and (\ref{g21234}), for $s_0$ large enough,
\begin{align}\label{12}&\Big|\Big(\frac{\sigma_0}{s_0}\Big)^2\overline{\psi}^3_1(\sigma_0)+\frac{2b\kappa\alpha}{(p-1)^2s_0}\Big|\nonumber\\ = & \Big|\Big(\frac{\sigma_0}{s_0}\Big)^2\overline{\psi}^3_1(\sigma_0)+\frac{\kappa}{2p}\frac{\alpha}{s_0}\Big|\nonumber\\ \leq & O\Big(\frac{|\alpha|}{s_0^2}\Big)+O(e^{-s_0}+\tau^2+\frac{|\tau|}{s_0}+|\tau\alpha|+|\alpha|^3)\|\varphi(s_0)\|_{L^\infty}
\nonumber\\ \leq & C\Big(\frac{|\alpha|}{s_0^2}+e^{-s_0}+\tau^2+\frac{|\tau|}{s_0}+|\tau\alpha|+|\alpha|^3\Big).
\end{align}

By  Lemma \ref{lemma g} (iii), (\ref{psi222}), (\ref{psi134}) and (\ref{g21234}), for $s_0$ large enough,
\begin{align}\label{13}&\overline{\psi}^3_2(\sigma_0)\nonumber\\ = & -\frac{\kappa}{4ps_0}+O\Big(\frac{1}{s_0^2}\Big)+O\Big(\frac{\alpha^2}{s_0^2}\Big)+O(e^{-s_0}
+|\tau|+|\alpha|^3)\|\varphi(s_0)\|_{L^\infty}
\nonumber\\ = & -\frac{\kappa}{4ps_0}+O\Big(\frac{1}{s_0^2}\Big)+O(e^{-s_0}
+|\tau|+|\alpha|^3).
\end{align}

By  Lemma \ref{lemma g} (iv), (\ref{psi222}), (\ref{psi134}) and (\ref{g21234}), for $s_0$ large enough,
\begin{align}\label{14}\sigma_0^{-2}\partial_\tau[\sigma_0^2\overline{\psi}^3_0](\sigma_0)
=&\frac{\kappa}{p-1}+O\Big(\frac{1}{s_0^2}\Big)+O(\frac{1}{s_0})+O\Big(\frac{1}{s_0^2}\Big)+O\Big(\frac{1}{s_0}+
|\tau|+|\alpha|\Big)\|\varphi(s_0)\|_{L^{\infty}}\nonumber\\
=&\frac{\kappa}{p-1}+O\Big(\frac{1}{s_0}\Big)+O(
|\tau|+|\alpha|).
\end{align}

By  Lemma \ref{lemma g} (v), (\ref{psi134}) and (\ref{g21234}), for $s_0$ large enough,
\begin{align}\label{15}|s_0^{-2}\partial_\tau[\sigma_0^2\overline{\psi}^3_1](\sigma_0)|
=&O\Big(\frac{1}{s_0}+
|\tau|+|\alpha|\Big)\|\varphi(s_0)\|_{L^{\infty}}\nonumber\\
\leq&C\Big(\frac{1}{s_0}+
|\tau|+|\alpha|\Big).
\end{align}

By  Lemma \ref{lemma g} (vi), (\ref{psi222}), (\ref{psi134}) and (\ref{g21234}),  for $s_0$ large enough,
\begin{align}\label{16}|\partial_\alpha\overline{\psi}^3_0(\sigma_0)|
=&|O(\alpha\varphi_2(s_0))+O(e^{-s_0}+|\tau|+|\alpha|^2)\|\varphi(s_0)\|_{L^\infty}|\nonumber\\
\leq&C\Big(e^{-s_0}+
|\tau|+|\alpha|^2+\frac{|\alpha|}{s_0}\Big).
\end{align}

By  Lemma \ref{lemma g} (vii), (\ref{psi222}), (\ref{psi134}) and (\ref{g21234}), for $s_0$ large enough,
\begin{align}\label{17}\partial_\alpha\overline{\psi}^3_1(\sigma_0)
=&2\varphi_2(s_0)+O(e^{-s_0}+|\tau|+|\alpha|^2)\|\varphi(s_0)\|_{L^\infty}\nonumber\\
=&-\frac{\kappa}{2p}\frac{1}{s_0}+O\Big(\frac{1}{s_0^2}\Big)+O(e^{-s_0}+
|\tau|+|\alpha|^2)\nonumber\\
=&-\frac{\kappa}{2p}\frac{1}{s_0}+O\Big(\frac{1}{s_0^2}\Big)+O(
|\tau|+|\alpha|^2)\nonumber\\=&-\frac{2b\kappa}{(p-1)^2}\frac{1}{s_0}+O\Big(\frac{1}{s_0^2}\Big)+O(
|\tau|+|\alpha|^2).
\end{align}

Now for the remaining components, we need more refined estimates, for $s_0$ large enough and for all $z\in (\Omega-a)e^{\sigma_0/2}$,
\begin{align}\label{psi31}|\overline{\psi}^3(z,\sigma_0)-\varphi(z,s_0)|
\leq&C\Big(
|\tau|+\frac{|\alpha|}{s_0}\Big)(1+|z|^3).
\end{align}

Indeed, for $z\in (\Omega-a)e^{\sigma_0/2}$,
\begin{align}\varphi_z(z,s)=-\frac{1}{p-1}\Big(p-1+b\frac{{z}^2}{s}\Big)^{-\frac{1}{p-1}-1}\frac{2z}{4ps}(p-1)^2.
\end{align}
Then, we have that for $s_0$ large enough,
\begin{align}&|\overline{\psi}^3(z,\sigma_0)-\varphi(z,s_0)|\nonumber\\
=&|(1+\tau)^{\frac{1}{p-1}}\varphi(\widetilde{z},s_0)-\varphi(z,s_0)|\nonumber\\=&
|(1+\tau)^{\frac{1}{p-1}}\varphi(z\sqrt{(1+\tau)}+\alpha,s_0)-(1+\tau)^{\frac{1}{p-1}}\varphi(z,s_0)+(1+\tau)^{\frac{1}{p-1}}\varphi(z,s_0)-
\varphi(z,s_0)|
\nonumber\end{align}\begin{align}\leq&(1+\tau)^{\frac{1}{p-1}}|\varphi(z\sqrt{(1+\tau)}+\alpha,s_0)-\varphi(z,s_0)|+|[(1+\tau)^{\frac{1}{p-1}}-1]\varphi(z,s_0)|\nonumber\\\leq&
C\frac{1}{s_0}|z+\theta(z\sqrt{(1+\tau)}+\alpha-z)||z\sqrt{(1+\tau)}+\alpha-z|++C|\tau|\nonumber\\\leq&
C\frac{1}{s_0}(|z|+|z\tau|+|\alpha|)(|\tau z|+|\alpha|)+C|\tau|\nonumber\\\leq&
C\Big(|\tau|+\frac{|\alpha|}{s_0}\Big)(1+|z|^3).
\end{align}

For $z\not\in (\Omega-a)e^{\sigma_0/2}$, $\overline{\psi}^3(z,\sigma_0)-\varphi(z,s_0)=-\varphi(z,s_0)$.

On the other hand, for $s_0$ large enough and for all $z\in (\Omega-a)e^{\sigma_0/2}$, we could obtain
\begin{align}\label{psi32}|\overline{\psi}^3(z,\sigma_0)-\varphi(z,s_0)|
\leq&C\Big(|\tau|+\frac{|\alpha|}{\sqrt{s_0}}\Big).
\end{align}

For all $s$ large enough and for all $z\in \mathbb{R}$, it holds by (\ref{ff}) that
\begin{align}\Big|\varphi(z,s)-\kappa+\frac{\kappa}{4ps}(z^2-2)\Big|\leq C\frac{|z|^4}{s^2},
\end{align}
and we have
\begin{align}\label{psi4}|\partial_s\varphi(z,s)|
\leq C\min\Big(\frac{1}{s}, \frac{1+|z|^2}{s^2}\Big).
\end{align}


Since $\overline{\psi}^4(z,\sigma_0)=-\varphi(z,\sigma_0)$ with $\sigma_0=s_0-\log(1+\tau)$, it holds that for $s_0$ large enough,
$$\Big|\frac{1}{\sigma_0}-\frac{1}{s_0}\Big|\leq \frac{C}{s_0^2}|s_0-\sigma_0|=\Big|\frac{C}{s_0^2}\log(1+\tau)\Big|\leq\frac{C}{s_0^2}|\tau|.$$ Then, it follows that
by (\ref{ff}), for $s_0$ large enough,
\begin{align}\label{18}&|\overline{\psi}^4(z,\sigma_0)+\kappa-\frac{\kappa}{4ps_0}h_2(z)|\nonumber\\
\leq& C\frac{|z|^4}{s_0^2}+C\frac{|\tau|(1+|z|^2)}{s_0^2}.
\end{align}
Moreover, we have that for $s_0$ large enough,
\begin{align}\label{19}|\overline{\psi}^4(z,\sigma_0)+\varphi(z,s_0)|
\leq &C|\tau|\min\Big(\frac{1}{s_0}, \frac{1+|z|^2}{s_0^2}\Big).
\end{align}
Thus,  for $s_0$ large enough,
\begin{align}\label{20}&\Big(\frac{\sigma_0}{s_0}\Big)^2\overline{\psi}^4_0(\sigma_0)=-\kappa+O(s_0^2)+O(\tau s_0^{-1}),\end{align}

\begin{align}\label{212}\overline{\psi}^4_1(\sigma_0)=0,\ \overline{\psi}^4_2(\sigma_0)=\frac{\kappa}{4ps_0}+O(s_0^{-2})+O(\tau s_0^{-2}),\end{align}
\begin{align}\label{22}|s_0^{-2}\partial_\tau[\sigma_0^2\overline{\psi}^4_0](\sigma_0)|
\leq&\frac{C}{s_0},
\end{align}
and
\begin{align}\label{psi41}\partial_\tau[\sigma_0^2\overline{\psi}^4_1](\sigma_0)=0,\ \partial_\alpha\overline{\psi}^4_0(\sigma_0)=\partial_\alpha\overline{\psi}^4_1(\sigma_0)=0.
\end{align}


In order to estimate $\overline{\psi}_{-}^3(z,\sigma_0)+\overline{\psi}_{-}^4(z,\sigma_0)$, we will use and easily check the following estimate:
\begin{align}&\Big\|\frac{v_{-}(z,\sigma_0)}{1+|z|^3}\Big\|_{L^{\infty}(\mathbb{R})}\leq
C\Big\|\frac{v(z,\sigma_0)}{1+|z|^3}\Big\|_{L^{\infty}(\mathbb{R})}.\end{align}

For $z\not\in (\Omega-a)e^{\sigma_0/2}$,  it holds that for $s_0$ large enough,
\begin{align}\left|\frac{\overline{\psi}^3(z,\sigma_0)+\overline{\psi}^4(z,\sigma_0)}{1+|z|^3}\right|
=\frac{|-\varphi(z,\sigma_0)|}{1+|z|^3}\leq Ce^{-s_0}.\end{align}

For $z\in (\Omega-a)e^{\sigma_0/2}$, we introduce
\begin{align}\Psi(z,s_0,\tau,\alpha)=\left\{\begin{array}{ll}\ \ \
\ \overline{\psi}^3(z,\sigma_0) +\overline {\psi}^4(z,\sigma_0)
 \\ =(1+\tau)^{\frac 1{p-1}}\varphi(\tilde z, s_0) - \varphi(z,\sigma_0),\ z\in(\Omega-a)e^{\sigma_0/2},\\ \ \ \ \ 0,
\ {z}\not\in (\Omega-a)e^{\sigma_0/2},
\end{array}\right.\end{align}
and
\begin{align}\Psi_1(z,\sigma_0)=\left\{\begin{array}{ll} 0,\ z\in(\Omega-a)e^{\sigma_0/2},\\-\varphi(z,\sigma_0),
\ {z}\not\in (\Omega-a)e^{\sigma_0/2},
\end{array}\right.\end{align}
where
\[
  \varphi(z,s)
  = \left[p-1+\frac{(p-1)^2}{4p} \frac{z^2}s\right]^{-\frac 1{p-1}}
+ \frac \kappa{2ps},\\
\tilde z = z\sqrt{1+\tau}+\alpha,\;\;
\sigma_0 = s_0 - \log(1+\tau).
\]

Note that $\Psi(z,s_0,0,0)=0$.
Making a Taylor
expansion, we may write
\begin{align}
  &\Psi(z,s_0,\tau,\alpha) - \tau \Psi_\tau (z,s_0,0,0)
  -\alpha  \Psi_\alpha (z,s_0,0,0)\nonumber\\=&\frac{1}{2} \Psi_{\alpha\alpha} (z,s_0,\theta(z,s_0)\tau,\theta(z,s_0)\alpha)\alpha^2
  +\frac{1}{2} \Psi_{\tau\tau} (z,s_0,\theta(z,s_0)\tau,\theta(z,s_0)\alpha)\tau^2
  \nonumber\\&+ \Psi_{\tau\alpha} (z,s_0,\theta(z,s_0)\tau,\theta(z,s_0)\alpha)\tau\alpha,\label{taylor}
\end{align}
where $0<\theta(z,s_0)<1$.


\medskip


First, for $z\in(\Omega-a)e^{\sigma_0/2}$,
  \begin{align}
 &\Psi_{\tau}(z,s_0,\tau,\alpha) \nonumber\\
  = &\frac 1{p-1}(1+\tau)^{\frac
      1{p-1}-1}\varphi(z\sqrt{1+\tau}+\alpha, s_0)\nonumber\\
    &-\frac{p-1}{2p}(1+\tau)^{\frac
      1{p-1}}\Big[p-1+\frac{(p-1)^2}{4p}
      \frac{(z\sqrt{1+\tau}+\alpha)^2}{s_0}\Big]^{-\frac
      1{p-1}-1}\frac{(z\sqrt{1+\tau}+\alpha)}{s_0} \times\nonumber\\
    &\times \frac{z}{2\sqrt{1+\tau}}\nonumber\\
    &+\frac{p-1}{4p}\Big[p-1+\frac{(p-1)^2}{4p} \frac{z^2}{s_0 - \log(1+\tau)}\Big]^{-\frac{1}{p-1}-1}\frac{z^2}{(s_0 - \log(1+\tau))^2}\frac{1}{1+\tau}\nonumber\\
    &-\frac{\kappa}{2p(s_0 - \log(1+\tau))^2}\frac{1}{1+\tau},
  \end{align}
which implies for $z\in(\Omega-a)e^{\sigma_0/2}$,
  \begin{align}\label{psitao}
 &\Psi_{\tau}(z,s_0,0,0) \nonumber\\
  = &\frac 1{p-1}\varphi(z,s_0)-\frac{p-1}{2p}\Big[p-1+\frac{(p-1)^2}{4p}\frac{z^2}{s_0}\Big]^{-\frac 1{p-1}-1}{\frac{z^2}{2s_0}}\nonumber
  \\ &+\frac{p-1}{4p}\Big[p-1+\frac{(p-1)^2}{4p} \frac{z^2}{s_0}\Big]^{-\frac{1}{p-1}-1}\frac{z^2}{s_0^2}-\frac{\kappa}{2ps_0^2}.
  \end{align}
 For $z\not\in(\Omega-a)e^{\sigma_0/2}$,
 \begin{align}
 \Psi_{\tau}(z,s_0,0,0)=0.
  \end{align}

In (\ref{psitao}), we see four terms. Using the usual trick, we can bound
the projection on the negative part (after cut-off) of the 2nd, 3rd
and 4th term by
\[
\frac{C}{s_0}(1+|z|^3),
\]
which after multiplication by $\tau$, should give us a convenient
rate.

\medskip

Let us then try to bound the projection on the negative part (after
cut-off) of $\varphi(z,s_0)$, namely
\[
[\varphi(z,s_0)]_-(z,\sigma_0)=\chi_1(z,\sigma_0)\varphi(z,s_0) - \sum_{j=0}^2 [\varphi(z,s_0)]_j(\sigma_0)h_j(\sigma_0).
\]
Note first that
\[
  \chi_1(z,\sigma_0)  = 1+O\Big(\frac{|z|^2}{\sigma_0}\Big),
\]
and
\[
  \varphi(z,s_0)=\kappa+O\Big(\frac{|z|^2}{s_0}\Big)+\frac{\kappa}{2ps_0}.
\]

Recalling that $\varphi$ is even in $z$, this implies that for $s_0$ large enough,
\begin{align}
  [\varphi(z,s_0)]_0(\sigma_0)=&\int \chi_1(z,\sigma_0)k_0(z)\varphi(z,s_0)\rho(z)dz\nonumber\\
  =& \int \chi_1(z,\sigma_0)\Big[\kappa+O\Big(\frac{|z|^2}{s_0}\Big)+\frac{\kappa}{2ps_0}\Big]\rho(z)dz\nonumber\\
  =&\int [(\chi_1(z,\sigma_0)-1)\kappa+\kappa]\rho(z)dz+ \int \chi_1(z,\sigma_0)\Big[O\Big(\frac{|z|^2}{s_0}\Big)
  +\frac{\kappa}{2ps_0}\Big]\rho(z)dz\nonumber\\
  =&\kappa+O\Big(\frac{1}{s_0}\Big),
\end{align}
\[
  [\varphi(z,s_0)]_1(\sigma_0)=0,
\]
and by (\ref{f1}), for $s_0$ large enough,
\begin{align}
  [\varphi(z,s_0)]_2(\sigma_0)=&\int \chi_1(z,\sigma_0)k_2(z)\varphi(z,s_0)\rho(z)dz\nonumber\\
  =&\int [(\chi_1(z,\sigma_0)-1)k_2(z)\varphi(z,s_0)\rho(z)dz+ \int k_2(z)\varphi(z,s_0)\rho(z)dz\nonumber\\
   =&\int [(\chi_1(z,\sigma_0)-1)k_2(z)\varphi(z,s_0)\rho(z)dz-\frac{\kappa}{4ps_0}+O(\frac{1}{s_0})\nonumber\\
  =&O\Big(\frac{1}{s_0}\Big).
\end{align}
This way, for $s_0$ large enough,
\begin{align}
  &[\varphi(z,s_0)]_-(z,\sigma_0)\nonumber\\=&\Big[1+O\Big(\frac{|z|^2}{\sigma_0}\Big)\Big]\Big[\kappa+O\Big(\frac{|z|^2}{s_0}\Big)+\frac{\kappa}{2ps_0}\Big]
  -\Big(\kappa+O\Big(\frac{1}{s_0}\Big)\Big)h_0(z)
                     - O\Big(\frac{1}{s_0}\Big)h_2(z)\nonumber\\
  =& \kappa + O\Big(\frac{|z|^2}{s_0}\Big) - \kappa+O\Big(\frac{1}{s_0}\Big) +
     O\Big(\frac{1+|z|^2}{s_0}\Big) \nonumber\\
  =&   O\Big(\frac{1+|z|^2}{s_0}\Big)\label{bu1}.
\end{align}

Thus,  for $s_0$ large enough,
\[
|[\varphi(z,s_0)]_-(z,\sigma_0)|\le \frac{C}{s_0}(1+|z|^2) \le \frac{C}{s_0}(1+|z|^3),
\]
from which we have that for $s_0$ large enough,
\begin{align}
|\tau(\Psi_\tau)_- (z,s_0,0,0)|\le \frac{C|\tau|}{s_0}(1+|z|^3)\label{taylor1}.
 \end{align}

Second, for $z\in(\Omega-a)e^{\sigma_0/2}$,
  \begin{align}
 &\Psi_{\alpha}(z,s_0,\tau,\alpha) \nonumber\\
  = &-\frac{1}{p-1}(1+\tau)^{\frac
      1{p-1}}
   \Big[p-1+\frac{(p-1)^2}{4p}
      \frac{(z\sqrt{1+\tau}+\alpha)^2}{s_0}\Big]^{-\frac
      1{p-1}-1}\frac{(p-1)^2(z\sqrt{1+\tau}+\alpha)}{2ps_0}.
     \end{align}
Hence, for $z\in(\Omega-a)e^{\sigma_0/2}$,
 \begin{align}
 &\Psi_{\alpha}(z,s_0,0,0)
  =-\frac{1}{p-1}(1+\tau)^{\frac
      1{p-1}}
   \Big[p-1+\frac{(p-1)^2}{4p}
      \frac{z^2}{s_0}\Big]^{-\frac
      1{p-1}-1}\frac{(p-1)^2z}{2ps_0}.
     \end{align}

Write $\overline{\varphi}=\Big[p-1+\frac{(p-1)^2}{4p}
      \frac{z^2}{s_0}\Big]^{-\frac
      1{p-1}-1}z$. Then $\overline{\varphi}=(p-1)^{-\frac
      1{p-1}-1}z+O(\frac{|z|^2}{s_0})z$.

It is easy to check that $\overline{\varphi}_0(\sigma_0)=\overline{\varphi}_2(\sigma_0)=0$, and for $s_0$ large enough,
 \begin{align}
 \overline{\varphi}_1=&\int (\chi_1-1+1)(p-1)^{-\frac
      1{p-1}-1}\frac{z^2}{2}\rho dz+\int \chi_1(z,\sigma_0)O\Big(\frac{|z|^2}{s_0}\Big)\frac{z^2}{2}\rho dz\nonumber\\
      =&\int (p-1)^{-\frac
      1{p-1}-1}\Big(\frac{z^2-2}{2}\Big)\rho dz+\int (p-1)^{-\frac
      1{p-1}-1}\rho dz+O\Big(\frac{1}{s_0}\Big)\nonumber\\=&(p-1)^{-\frac
      1{p-1}-1}+O\Big(\frac{1}{s_0}\Big).
     \end{align}
     Thus, for $s_0$ large enough,
\begin{align}
&(\Psi_{\alpha})(z,s_0,0,0)_{-}(z,\sigma_0)\nonumber\\=&\Big[1+O\Big(\frac{|z|^2}{s_0}\Big)\Big]\Big[-(1+\tau)^{\frac
      1{p-1}}\frac{(p-1)}{2ps_0}(p-1)^{-\frac
      1{p-1}-1}z-(1+\tau)^{\frac
      1{p-1}}\frac{(p-1)}{2ps_0}O\Big(\frac{|z|^2}{s_0}\Big)z\Big]\nonumber\\&
  +(1+\tau)^{\frac
      1{p-1}}\frac{(p-1)}{2ps_0}\Big[(p-1)^{-\frac
      1{p-1}-1}+O\Big(\frac{1}{s_0}\Big)\Big]z
                   \nonumber\\
  =&   O\Big(\frac{1+|z|^3}{s_0^2}\Big),
\end{align}
which implies that for $s_0$ large enough,
\begin{align}
|\alpha(\Psi_{\alpha})_- (z,s_0,0,0)|\le \frac{C|\alpha|}{s_0^2}(1+|z|^3)\label{taylor2}.
\end{align}

Further, we can check that for $s_0$ large enough,
\begin{align}&|\Psi_{\tau\tau} (z,s_0,\theta(z,s_0)\tau,\theta(z,s_0)\alpha)\tau^2|\leq C(1+|z|^3)\tau^2,\nonumber\\
&|\Psi_{\tau\alpha} (z,s_0,\theta(z,s_0)\tau,\theta(z,s_0)\alpha)\tau\alpha|\leq  C(1+|z|^3)\Big(\tau^2+\frac{\alpha^2}{s_0}\Big).\nonumber\end{align}
Hence, for $s_0$ large enough,
\begin{align}&|\frac{1}{2}(\Psi_{\tau\tau})_- (z,s_0,\theta(z,s_0)\tau,\theta(z,s_0)\alpha)\tau^2|\leq C(1+|z|^3)\tau^2,\label{taylor3}\\
&|(\Psi_{\tau\alpha})_- (z,s_0,\theta(z,s_0)\tau,\theta(z,s_0)\alpha)\tau\alpha|\leq C(1+|z|^3)\Big(\tau^2+\frac{\alpha^2}{s_0}\Big).\label{taylor4}\end{align}

Finally, we have
 \begin{align}
 &\Psi_{\alpha\alpha}(z,s_0,\tau,\alpha) \nonumber\\
  = &\frac{p}{(p-1)^2}(1+\tau)^{\frac
      1{p-1}}
   \Big[p-1+\frac{(p-1)^2}{4p}
      \frac{(z\sqrt{1+\tau}+\alpha)^2}{s_0}\Big]^{-\frac
      1{p-1}-2}\frac{(p-1)^4(z\sqrt{1+\tau}+\alpha)^2}{4(ps_0)^2}\nonumber\\
      &-\frac{p-1}{2ps_0}(1+\tau)^{\frac
      1{p-1}}
   \Big[p-1+\frac{(p-1)^2}{4p}
      \frac{(z\sqrt{1+\tau}+\alpha)^2}{s_0}\Big]^{-\frac
      1{p-1}-1}=J_1+J_2.
     \end{align}

     For $J_1$, it holds that for $s_0$ large enough,
     $$|J_1(z,s_0,\theta(z,s_0)\tau,\theta(z,s_0))\alpha^2|\leq  C(1+|z|^3)\frac{\alpha^2}{s_0^2},$$
which implies  that for $s_0$ large enough,
 \begin{align}|(J_1)_-(z,s_0,\theta(z,s_0)\tau,\theta(z,s_0))\alpha^2|\leq  C(1+|z|^3)\frac{\alpha^2}{s_0^2}.\label{taylor5}\end{align}

For $J_2$, we have
 \begin{align}&J_2(z,s_0,\theta(z,s_0)\tau,\theta(z,s_0))\nonumber\\=&-\frac{p-1}{2ps_0}(1+\theta\tau)^{\frac
      1{p-1}}
   \Big[p-1+\frac{(p-1)^2}{4p}
      \frac{(z\sqrt{1+\theta\tau}+\theta\alpha)^2}{s_0}\Big]^{-\frac
      1{p-1}-1}\nonumber\\
      =&-\frac{p-1}{2ps_0}(1+\theta\tau)^{\frac
      1{p-1}}
   \Big\{\Big[p-1+\frac{(p-1)^2}{4p}
      \frac{(z\sqrt{1+\theta\tau}+\theta\alpha)^2}{s_0}\Big]^{-\frac
      1{p-1}-1}\nonumber\\
      &-\Big[p-1+\frac{(p-1)^2}{4p}
      \frac{z^2}{s_0}\Big]^{-\frac
      1{p-1}-1}\Big\}-\frac{p-1}{2ps_0}(1+\theta\tau)^{\frac
      1{p-1}}\Big\{\Big[p-1+\frac{(p-1)^2}{4p}
      \frac{z^2}{s_0}\Big]^{-\frac
      1{p-1}-1}\nonumber\\
      &+\frac{(p-1)^{-\frac
      1{p-1}-1}}{2s_0}\Big\}+\frac{(p-1)^{-\frac
      1{p-1}}}{4ps_0^2}(1+\theta\tau)^{\frac
      1{p-1}}=J_2^1+J_2^2+J_2^3.\nonumber\end{align}

It is easy to check that
\begin{align}|(J_2^1)_-(z,s_0,\theta(z,s_0)\tau,\theta(z,s_0))\alpha^2|\leq  C(1+|z|^3)\frac{\alpha^2}{s_0^2},\label{taylor6}\end{align}
and
\begin{align}|(J_2^3)_-(z,s_0,\theta(z,s_0)\tau,\theta(z,s_0))\alpha^2|\leq  C(1+|z|^3)\frac{\alpha^2}{s_0^2},\label{taylor7}
\end{align}
and by the similar argument used to prove (\ref{bu1}), we have
\begin{align}
|(J_2^2)_-(z,s_0,\theta(z,s_0)\tau,\theta(z,s_0))\alpha^2|\le C(1+|z|^3)\frac{\alpha^2}{s_0^2}\label{taylor8}.
\end{align}

Hence, by (\ref{taylor}), (\ref{taylor1}), (\ref{taylor2})-(\ref{taylor4}), (\ref{taylor5})-(\ref{taylor8}), it holds that
for $s_0$ large enough,
\begin{align}\label{2511}\left\|\frac{\overline{\psi}_{-}^3(z,\sigma_0)+\overline{\psi}_{-}^4(z,\sigma_0)}{1+|z|^3}\right\|_{L^\infty(\mathbb{R}^n)}
\leq \frac{C}{s_0}|\tau|+C\frac{|\alpha|}{s_0^2}+C\frac{|\alpha|^2}{s_0^2}+C|\tau|^2+Ce^{-s_0}.
\end{align}

When $z\in (\Omega-a)e^{\sigma_0/2}$, it holds that by (\ref{psi31}) that for $s_0$ large enough,
\begin{align}\label{25}|\overline{\psi}_{e}^3(z,\sigma_0)+\overline{\psi}_{e}^4(z,\sigma_0)|
=|(1-\chi_1)(\overline{\psi}^3(z,\sigma_0)+\overline{\psi}^4(z,\sigma_0))|\leq C\Big(|\tau|+\frac{|\alpha|}{\sqrt{s_0}}\Big).
\end{align}
When $z\not\in (\Omega-a)e^{\sigma_0/2}$, for $s_0$ large enough,
\begin{align}\label{251}|\overline{\psi}_{e}^3(z,\sigma_0)+\overline{\psi}_{e}^4(z,\sigma_0)|
=|(1-\chi_1)\varphi(z,\sigma_0)|\leq\frac{C}{\sigma_0}.\end{align}

{\bf Step 4. Expansion of $\overline{\psi}^5$}

\medskip

(i) For $\overline{\psi}_0^5(\sigma_0)$, we have that  when $s_0$ is large enough such that  $|a-\widehat{a}|=|\alpha e^{-\frac{s_0}{2}}|
\leq e^{-\frac{s_0}{2}}<\frac{\varepsilon_0}{4}$,
\begin{align} \label{star1}|\overline{\psi}_0^5(\sigma_0)|\leq C|\alpha|e^{-s_0/2}.\end{align}

Indeed,
\begin{align}\overline{\psi}_0^5(\sigma_0)=&\int_{\mathbb{R}} \overline{\psi}^5(\sigma_0)k_0(z)\chi_1(z,\sigma_0)\rho(z)dz\nonumber\\
=&\int_{(\Omega-a)e^{\sigma_0/2}} e^{-\frac{\sigma_0}{p-1}}\widehat{y}(e^{-\frac{\sigma_0}{2}}{z}+a,T-e^{-\sigma_0})[\chi(z,\sigma_0)-\chi(\widetilde{z},s_0)]\chi_1(z,\sigma_0)\rho(z)dz
\nonumber,\end{align}
where
$$\chi(z,\sigma_0)=\chi_0\Big(\frac{ze^{-\sigma_0/2}}{\varepsilon_0}\Big),\ \chi(\widetilde{z},s_0)=\chi_0\Big(\frac{\widetilde{z}e^{-s_0/2}}{\varepsilon_0}\Big)=\chi_0\Big(\frac{ze^{-\sigma_0/2}+a-\widehat{a}}{\varepsilon_0}\Big).$$

On one hand, when $x\in \Omega$ with $|x-\widehat{a}|\leq \frac{\varepsilon_0}{2}$ and $s_0$ is large enough such that  $|a-\widehat{a}|<\frac{\varepsilon_0}{4}$,
we have $\chi(z,\sigma_0)-\chi(\widetilde{z},s_0)=0$.


\medskip

On the other hand, for $x\in \Omega$ with  $|x-\widehat{a}|\geq \varepsilon_0/2$, it holds that
$$\Big|\chi_0\Big(\frac{ze^{-\sigma_0/2}}{\varepsilon_0}\Big)-\chi_0\Big(\frac{ze^{-\sigma_0/2}+a-\widehat{a}}{\varepsilon_0}\Big)\Big|
\leq \frac{C}{\varepsilon_0}|a-\widehat{a}|,$$
and
$$|\widehat{y}(e^{-\frac{\sigma_0}{2}}{z}+a,T-e^{-\sigma_0})|
=|\widehat{y}(e^{-\frac{s_0}{2}}{\widetilde{z}}+\widehat{a},\widehat{T}-e^{-s_0})|\leq \widehat{\eta}_0\leq 1.$$

Hence,  when $s_0$ is large enough such that  $|a-\widehat{a}|=|\alpha e^{-\frac{s_0}{2}}|
\leq e^{-\frac{s_0}{2}}<\frac{\varepsilon_0}{4}$,
\begin{align}\label{}|\overline{\psi}_0^5(\sigma_0)|\leq \frac{C}{\varepsilon_0}e^{-\frac{\sigma_0}{p-1}}|a-\widehat{a}|=\frac{C}{\varepsilon_0}e^{-\frac{\sigma_0}{p-1}}
|a-\widehat{a}|e^{s_0/2}e^{-s_0/2}\leq C|\alpha|e^{-s_0/2}.\end{align}

(ii) Similarly to (\ref{star1}), we have that when $s_0$ is large enough,
\begin{align}\label{star2} \Big|\frac{\sigma_0^2}{s_0^2}\overline{\psi}_1^5(\sigma_0)\Big|\leq C|\alpha|e^{-s_0/2},\end{align}
and
\begin{align}\label{star3} \Big|\overline{\psi}_2^5(\sigma_0)\Big|\leq C|\alpha|e^{-s_0/2}.\end{align}

(iii) For $\sigma_0^{-2}\partial_\tau[\sigma_0^2\overline{\psi}^5_0]$, we have that when $s_0$ is large enough,
\begin{align}\label{star5}&\Big|\sigma_0^{-2}\partial_\tau[\sigma_0^2\overline{\psi}^5_0]\Big|\leq
Ce^{-Cs_0}(1+|\alpha|).\end{align}



Similarly to $\overline{\psi}_0^1(\sigma_0)$, it hods that for $s_0$ large enough,
\begin{align}&\Big(\frac{\sigma_0}{s_0}\Big)^2\overline{\psi}_0^5(\sigma_0)\nonumber\\=&\Big(\frac{\sigma_0}{s_0}\Big)^2\int_{\mathbb{R}} e^{-\frac{\sigma_0}{p-1}}\widehat{y}(e^{-\frac{s_0}{2}}\widetilde{z}+\widehat{a},\widehat{T}-e^{-s_0})\chi_1(z,\sigma_0)\rho(z)
[\chi(z,\sigma_0)-\chi(\widetilde{z},s_0)]dz \nonumber\\=&\Big(\frac{s_0-\log(1+\tau)}{s_0}\Big)^2(1+\tau)^{-\frac{1}{2}}\int_{\mathbb{R}}
\chi_1(z,s_0-\log(1+\tau))e^{-\frac{\sigma_0}{p-1}}\widehat{y}(e^{-\frac{s_0}{2}}\widetilde{z}+\widehat{a},\widehat{T}-e^{-s_0})
\nonumber\\&\cdot[\chi(z,\sigma_0)-\chi(\widetilde{z},s_0)]\rho(z)d\widetilde{z}.\end{align}
Hence,
\begin{align}&\frac{\partial}{\partial\tau}\Big[\Big(\frac{\sigma_0}{s_0}\Big)^2\overline{\psi}_0^5\Big](\sigma_0)=
-2\frac{1}{s_0(1+\tau)}\frac{s_0-\log(1+\tau)}{s_0}
\int_{(\Omega-a)e^{\sigma_0/2}}
\chi_1(z,s_0-\log(1+\tau))\nonumber\\&e^{-\frac{\sigma_0}{p-1}}\widehat{y}(e^{-\frac{s_0}{2}}\widetilde{z}+\widehat{a},\widehat{T}-e^{-s_0})
[\chi(z,\sigma_0)-\chi(\widetilde{z},s_0)]\rho(z)dz\nonumber\\&-\frac{1}{2}(1+\tau)^{-1}\Big(\frac{s_0-\log(1+\tau)}{s_0}\Big)^2
\int_{(\Omega-a)e^{\sigma_0/2}}
\chi_1(z,s_0-\log(1+\tau))e^{-\frac{\sigma_0}{p-1}}\widehat{y}(e^{-\frac{s_0}{2}}\widetilde{z}\nonumber\\&+\widehat{a},\widehat{T}-e^{-s_0})
[\chi(z,\sigma_0)-\chi(\widetilde{z},s_0)]\rho(z)dz
\nonumber\\&+\frac{1}{(p-1)(1+\tau)}\Big(\frac{s_0-\log(1+\tau)}{s_0}\Big)^2
\int_{(\Omega-a)e^{\sigma_0/2}}
\chi_1(z,s_0-\log(1+\tau))e^{-\frac{\sigma_0}{p-1}}\widehat{y}(e^{-\frac{s_0}{2}}\widetilde{z}\nonumber\\&+\widehat{a},\widehat{T}-e^{-s_0})
[\chi(z,\sigma_0)-\chi(\widetilde{z},s_0)]\rho(z)dz
\nonumber\\&-\Big(\frac{s_0-\log(1+\tau)}{s_0}\Big)^2
\int_{(\Omega-a)e^{\sigma_0/2}}(1+\tau)^{-1}\chi_{1\sigma_0}(z,s_0-\log(1+\tau))
e^{-\frac{\sigma_0}{p-1}}\widehat{y}(e^{-\frac{s_0}{2}}\widetilde{z}\nonumber\\&+\widehat{a},\widehat{T}-e^{-s_0})
[\chi(z,\sigma_0)-\chi(\widetilde{z},s_0)]\rho(z)dz
\nonumber\\&-\frac{1}{2}\Big(\frac{s_0-\log(1+\tau)}{s_0}\Big)^2
\int_{(\Omega-a)e^{\sigma_0/2}}(1+\tau)^{-1}z\chi_{1z}(z,s_0-\log(1+\tau))
e^{-\frac{\sigma_0}{p-1}}\widehat{y}(e^{-\frac{s_0}{2}}\widetilde{z}\nonumber\\&+\widehat{a},\widehat{T}-e^{-s_0})
[\chi(z,\sigma_0)-\chi(\widetilde{z},s_0)]\rho(z)dz
\nonumber\\&-\frac{1}{2}\Big(\frac{s_0-\log(1+\tau)}{s_0}\Big)^2
\int_{(\Omega-a)e^{\sigma_0/2}}(1+\tau)^{-1}z\chi_1(z,s_0-\log(1+\tau))
e^{-\frac{\sigma_0}{p-1}}\widehat{y}(e^{-\frac{s_0}{2}}\widetilde{z}\nonumber\\&+\widehat{a},\widehat{T}-e^{-s_0})
\chi_z(z,\sigma_0)\rho(z)dz+\Big(\frac{s_0-\log(1+\tau)}{s_0}\Big)^2
\int_{(\Omega-a)e^{\sigma_0/2}}\frac{z^2}{4}e^{-\frac{\sigma_0}{p-1}}\widehat{y}(e^{-\frac{s_0}{2}}\widetilde{z}\nonumber\\&
+\widehat{a},\widehat{T}-e^{-s_0})
\chi_1(z,\sigma_0)[\chi(z,\sigma_0)-\chi(\widetilde{z},s_0)](1+\tau)^{-1}\rho(z)dz\nonumber\\&-\Big(\frac{s_0-\log(1+\tau)}{s_0}\Big)^2
\int_{(\Omega-a)e^{\sigma_0/2}}(1+\tau)^{-1}\chi_1(z,s_0-\log(1+\tau))
e^{-\frac{\sigma_0}{p-1}}\widehat{y}(e^{-\frac{s_0}{2}}\widetilde{z}\nonumber\\&+\widehat{a},\widehat{T}-e^{-s_0})
\chi_{\sigma_0}(z,\sigma_0)\rho(z)dz.\end{align}

 For $|x-\widehat{a}|\geq \varepsilon_0/2$, we have
$$|\widehat{y}(e^{-\frac{\sigma_0}{2}}{z}+a,T-e^{-\sigma_0})|
=|\widehat{y}(e^{-\frac{s_0}{2}}{\widetilde{z}}+\widehat{a},\widehat{T}-e^{-s_0})|\leq \widehat{\eta}_0\leq 1.$$

If $|x-\widehat{a}|
< \frac{\varepsilon_0}{2}$ and $|a-\widehat{a}|<\frac{\varepsilon_0}{4}$, then $\chi(z,\sigma_0)-\chi(\widetilde{z},s_0)=0$,
which implies $\partial_\tau[\chi(z,\sigma_0)-\chi(\widetilde{z},s_0)]=0$. Thus, for $s_0$ large enough,  we can get (\ref{star5}) by
the similar argument used in the proof of (\ref{partial tao 0}).

\medskip

(iv) For $\partial_\tau\Big[\frac{\sigma_0^2}{s_0^2}\overline{\psi}_1^5\Big]$, we have that for $s_0$ large enough,
\begin{align}\label{star7} &\Big|\partial_\tau\Big[\frac{\sigma_0^2}{s_0^2}\overline{\psi}_1^5\Big]\Big| \leq
Ce^{-Cs_0}(1+|\alpha|).\end{align}

Since
\begin{align}\overline{\psi}_1^5(\sigma_0)=&\int_{\mathbb{R}} k_1(z)\overline{\psi}^5(\sigma_0)\chi_1(z,\sigma_0)\rho(z)dz\nonumber\\
=&\int_{(\Omega-a)e^{\sigma_0/2}} \frac{z}{2}e^{-\frac{\sigma_0}{p-1}}\widehat{y}(e^{-\frac{\sigma_0}{2}}{z}+a,T-e^{-\sigma_0})\chi_1(z,\sigma_0)
[\chi(z,\sigma_0)-\chi(\widetilde{z},s_0)]\rho(z)dz\nonumber,\end{align}
similarly to (iii), we can get (\ref{star7}).

\medskip

(v) For $\partial_\alpha\overline{\psi}_0^5$ and $\partial_\alpha\overline{\psi}_1^5$, we can use the similar techniques used in the proof of (iii)
to get that for $s_0$ large enough,
\begin{align}\label{star8}\Big|\partial_\alpha\overline{\psi}_0^5\Big| \leq
Ce^{-Cs_0}(1+|\alpha|),\end{align}
and
\begin{align}\label{star9}\Big|\partial_\alpha\overline{\psi}_1^5\Big| \leq
Ce^{-Cs_0}(1+|\alpha|).\end{align}

(vi) For $\overline{\psi}_{-}^5(z,\sigma_0)$, similarly to (i), we have that  for $s_0$ large enough,
\begin{align}\label{star10}&\left\|\frac{\overline{\psi}_{-}^5(z,\sigma_0)}{1+|z|^3}\right\|_{L^{\infty}(\mathbb{R})}\leq
\left\|\frac{\overline{\psi}^5(z,\sigma_0)}{1+|z|^3}\right\|_{L^{\infty}(\mathbb{R})}\nonumber\\
=
& \left\|\frac{ e^{-\frac{\sigma_0}{p-1}}\widehat{y}(e^{-\frac{\sigma_0}{2}}{z}+a,T-e^{-\sigma_0})
[\chi(z,\sigma_0)-\chi(\widetilde{z},s_0)]}{1+|z|^3}\right\|_{L^{\infty}((\Omega-a)e^{\sigma_0/2})}\nonumber\\
\leq& Ce^{-s_0/2}|\alpha|.\end{align}

(vii) For $\overline{\psi}_{e}^5(z,\sigma_0)$,  similarly to (i), we have that  for $s_0$ large enough,
\begin{align}\label{star11}&\Big\|\overline{\psi}_{e}^5(z,\sigma_0)\Big\|_{L^{\infty}(\mathbb{R})} \nonumber\\
\leq &\|e^{-\frac{\sigma_0}{p-1}}\widehat{y}(e^{-\frac{\sigma_0}{2}}{z}+a,T-e^{-\sigma_0})(1-\chi_1(z,\sigma_0))
[\chi(z,\sigma_0)-\chi(\widetilde{z},s_0)]\|_{L^{\infty}((\Omega-a)e^{\sigma_0/2)}}\nonumber\\
\leq& Ce^{-s_0/2}|\alpha|.\end{align}



{\bf Conclusion of Lemma \ref{lem 2.2}.}

\medskip

For $s_0$ large enough, Lemma \ref{lem 2.2} (i) holds from (\ref{1}), (\ref{psi2inequality}),  (\ref{11}) and (\ref{star1});

\medskip

For $s_0$ large enough, Lemma \ref{lem 2.2} (ii) holds from (\ref{1}), (\ref{psi3inequality}), (\ref{12}), (\ref{212}) and (\ref{star2});

\medskip

For $s_0$ large enough, Lemma \ref{lem 2.2} (iii) holds from  (\ref{1}), (\ref{psi4inequality}), (\ref{13}), (\ref{212}) and (\ref{star3});

\medskip

For $s_0$ large enough, Lemma \ref{lem 2.2} (iv) holds from (\ref{1}), (\ref{psi5inequality}), (\ref{14}), (\ref{22}) and (\ref{star5});

\medskip

For $s_0$ large enough, Lemma \ref{lem 2.2} (v) holds from (\ref{1}), (\ref{psi6inequality}), (\ref{15}), (\ref{psi41}) and (\ref{star7});

\medskip

For $s_0$ large enough, Lemma \ref{lem 2.2} (vi) holds from  (\ref{1}), (\ref{psi7inequality}), (\ref{16}), (\ref{psi41}) and (\ref{star8});

\medskip

For $s_0$ large enough, Lemma \ref{lem 2.2} (vii) holds from  (\ref{1}), (\ref{psi8inequality}), (\ref{17}), (\ref{psi41})  and (\ref{star9});

\medskip

For $s_0$ large enough, Lemma \ref{lem 2.2} (viii) holds from  (\ref{1}), (\ref{psi9inequality}), (\ref{2511}) and  (\ref{star10});

\medskip

For $s_0$ large enough, Lemma \ref{lem 2.2} (ix) holds from (\ref{1}), (\ref{10}), (\ref{25}),  (\ref{251}) and  (\ref{star11}).

\bigskip

{\bf Proof of Proposition \ref{prop3.1}.}

\medskip

Assuming Lemma \ref{lem 2.2} holds.

(i) Introducing the following change of functions and variables:
\begin{align}\label{transformation}\widetilde{\psi}_m(s_0,\widetilde{\varepsilon}_0,\widetilde{\tau},\widetilde{\alpha})=\sigma_0^2\overline{\psi}_m(s_0,y_0,T,a)\end{align}
for $m=0,\ 1$, $\widetilde{\varepsilon}_0=y_0-\widehat{y}_0$, $\widetilde{\tau}=s_0^2\tau$, $\widetilde{\alpha}=s_0\alpha$. By Lemma \ref{lem 2.2}, it holds that
whenever $\|\widetilde{\varepsilon}_0\|_{W_0^{1,q}(\Omega)\cap W^{2,q}(\Omega)}\leq \frac{1}{C_0(s_0)}$, $|\widetilde{\tau}|\leq \frac{s_0^2}{16}$, $|\widetilde{\alpha}|\leq \frac{s_0}{4}$
and $s_0$ is large enough, it follows that
\begin{align}\label{y1}&\Big|\widetilde{\psi}_0(s_0,\widetilde{\varepsilon}_0,\widetilde{\tau},\widetilde{\alpha})-s_0^2\widehat{q}_0(s_0)
-\frac{{\kappa} \widetilde{\tau}}{p-1}\Big|\nonumber\\ =&
\Big|\sigma_0^2\overline{\psi}_0(s_0,y_0,T,a)-s_0^2\widehat{q}_0(s_0)
-\frac{\kappa s_0^2\tau}{p-1}\Big|\nonumber\\ =&
s_0^2\Big|\frac{\sigma_0^2}{s_0^2}\overline{\psi}_0(s_0,y_0,T,a)-\widehat{q}_0(s_0)
-\frac{\kappa \tau}{p-1}\Big|\nonumber\\ \leq& Cs_0^2\|\widetilde{\varepsilon}_0\|_{
{W_0^{1,q}(\Omega)\cap W^{2,q}(\Omega)}}+C\frac{\widehat{A}}{s_0^2}|\widetilde{\tau}|+C\frac{\widehat{A}}{s_0}|\widetilde{\alpha}|+C\widehat{A}^2s_0^{\frac{3}{2}}e^{-s_0}
+C\frac{\widehat{A}^2}{s_0^{\frac{1}{2}}}|\widetilde{\tau}|+C\frac{\widehat{A}^2}{s_0^{\frac{3}{2}}}|\widetilde{\alpha}|^3
\nonumber\\&+C\widehat{A}^2\log s_0(\frac{|\widetilde{\tau}|}{s_0^2}
+\frac{|\widetilde{\alpha}|^2}{s_0^2})+Ce^{-s_0}s_0^2+C\frac{\widetilde{\tau}^2}{s_0^2}+C\frac{|\widetilde{\tau}|}{s_0}+C\frac{\widetilde{\alpha}^2}{s_0}
+C|\widetilde{\alpha}|s_0e^{-s_0/2}
,\end{align}
\begin{align}\label{y2}&\Big|\widetilde{\psi}_1(s_0,\widetilde{\varepsilon}_0,\widetilde{\tau},\widetilde{\alpha})-s_0^2\widehat{q}_1(s_0)
+\frac{2b\kappa}{(p-1)^2}\widetilde{\alpha}\Big|\nonumber\\ =&
\Big|\sigma_0^2\overline{\psi}_1(s_0,y_0,T,a)-s_0^2\widehat{q}_1(s_0)
+\frac{2b\kappa}{(p-1)^2}s_0\alpha \Big|\nonumber\\ =&
s_0^2\Big|\frac{\sigma_0^2}{s_0^2}\overline{\psi}_1(s_0,y_0,T,a)-\widehat{q}_1(s_0)
+\frac{2b\kappa}{(p-1)^2}\frac{\alpha }{s_0}\Big|\nonumber\\ \leq& Cs_0^2\|\widetilde{\varepsilon}_0\|_{
{W_0^{1,q}(\Omega)\cap W^{2,q}(\Omega)}}+C\widehat{A}^2s_0^{\frac{3}{2}}e^{-s_0}
+C\frac{\widehat{A}}{s_0^2}|\widetilde{\tau}|+C\frac{\widehat{A}^2}{s_0^{\frac{5}{2}}}|\widetilde{\tau}^2|\nonumber\\&+C\frac{\widehat{A}^2}{s_0^{\frac{3}{2}}}|\widetilde{\alpha}|^3
+C\frac{\widehat{A}^2}{s_0^{\frac{3}{2}}}|\widetilde{\tau}|
+C\frac{\widehat{A}}{s_0^2}|\widetilde{\tau}|+C\frac{\widehat{A}}{s_0^2}|\widetilde{\alpha}|^2+C\widehat{A}^2\log s_0\Big(\frac{|\widetilde{\alpha}|}{s_0}\Big)\nonumber\\&+C\frac{\widehat{A}^2}{s_0^{\frac{3}{2}}}|\widetilde{\tau}\widetilde{\alpha}|+Ce^{-s_0}s_0^2
+C\frac{\widetilde{\tau}^2}{s_0^2}+C\frac{|\widetilde{\tau}|}{s_0}+C\frac{|\widetilde{\alpha}|}{s_0}\nonumber\\&+C\frac{|\widetilde{\tau}\widetilde{\alpha}|}{s_0}
+C\frac{|\widetilde{\alpha}|^3}{s_0}
+C|\widetilde{\alpha}|s_0e^{-s_0/2}+C|\widetilde{\tau}|e^{-Cs_0}+C|\widetilde{\alpha}|^2e^{-Cs_0}
.\end{align}

Let
\begin{align}
\mbox{Jac}_{\widetilde{\tau},\widetilde{\alpha}}(s_0,\widetilde{\varepsilon}_0,\widetilde{\psi}_0,\widetilde{\psi}_1)(\widetilde{\tau},\widetilde{\alpha})=\left|\begin{matrix}
\frac{\partial \widetilde{\psi}_0}{\partial\widetilde{\tau}} & \frac{\partial \widetilde{\psi}_0}{\partial\widetilde{\alpha}}\\\\
\frac{\partial \widetilde{\psi}_1}{\partial\widetilde{\tau}} & \frac{\partial \widetilde{\psi}_1}{\partial\widetilde{\alpha}}
\end{matrix}\right|=\frac{\partial \widetilde{\psi}_0}{\partial\widetilde{\tau}}\frac{\partial \widetilde{\psi}_1}{\partial\widetilde{\alpha}}
-\frac{\partial \widetilde{\psi}_0}{\partial\widetilde{\alpha}}\frac{\partial \widetilde{\psi}_1}{\partial\widetilde{\tau}}.
\end{align}

Since
\begin{align}
&\frac{\partial \widetilde{\psi}_0}{\partial\widetilde{\tau}}-\frac{\kappa}{p-1}
=\frac{\partial( \sigma_0^2\overline{\psi}_0)}{\partial{\tau}}\frac{\partial{\tau}}{\partial\widetilde{\tau}}-\frac{\kappa}{p-1}\nonumber\\=&
\frac{1}{s_0^2}\frac{\partial (\sigma_0^2\overline{\psi}_0)}{\partial{\tau}}-\frac{\kappa}{p-1}=\frac{\sigma_0^2}{s_0^2}\frac{1}{\sigma_0^2}\frac{\partial (\sigma_0^2\overline{\psi}_0)}{\partial{\tau}}-\frac{\sigma_0^2}{s_0^2}\frac{\kappa}{p-1}\nonumber\\&+\Big(\frac{\sigma_0^2}{s_0^2}-1\Big)\frac{\kappa}{p-1},
\end{align}
and
\begin{align}&\Big|\frac{\sigma_0^2}{s_0^2}-1\Big|
\leq\frac{C}{s_0}|\log(1+\tau)|\leq \frac{C|\tau|}{s_0}= \frac{C|\widetilde{\tau}|}{s_0^3},\end{align}
we have by Lemma \ref{lem 2.2} that for $s_0$ large enough, the following four inequalities holds:
\begin{align}&\Big|\frac{\partial \widetilde{\psi}_0}{\partial\widetilde{\tau}}-\frac{\kappa}{(p-1)}\Big|
\nonumber\\ \leq& C\|\widetilde{\varepsilon}_0\|_{
{W_0^{1,q}(\Omega)\cap W^{2,q}(\Omega)}}+C\frac{\widehat{A}}{s_0^2}+Ce^{-Cs_0}
+C\frac{\widehat{A}^2}{s_0^{\frac{1}{2}}}\Big(\frac{1}{s_0}+\frac{|\widetilde{\tau}|}{s_0^{2}}+\frac{|\widetilde{\alpha}|}{s_0}\Big)\nonumber\\&+
C\frac{\widehat{A}^2}{s_0^2}\log s_0+C\Big(\frac{1}{s_0}+
\frac{|\widetilde{\tau}|}{s_0^2}+\frac{|\widetilde{\alpha}|}{s_0}\Big)
+C\frac{|\widetilde{\alpha}|}{s_0}e^{-Cs_0}+C\frac{|\widetilde{\tau}|}{s_0^3}
,\end{align}
\begin{align}&\Big|\frac{\partial \widetilde{\psi}_0}{\partial\widetilde{\alpha}}\Big|=\Big|\frac{\partial( \sigma_0^2\overline{\psi}_0)}{\partial{\alpha}}\frac{\partial{\alpha}}{\partial\widetilde{\alpha}}\Big|=\Big|\frac{\partial( \overline{\psi}_0)}{\partial{\alpha}}\frac{\sigma_0^2}{s_0}\Big|
\nonumber\\ \leq& Cs_0\|\widetilde{\varepsilon}_0\|_{
{W_0^{1,q}(\Omega)\cap W^{2,q}(\Omega)}}+C\frac{\widehat{A}}{s_0}+\frac{C\widehat{A}^2\log s_0|\widetilde{\alpha}|}{s_0^2}+Cs_0^{\frac{1}{2}}\widehat{A}^2e^{-s_0}+Cs_0e^{-s_0}
\nonumber\\&+C\frac{{|\widetilde\alpha|}}{s_0}+C\frac{|\widetilde{\tau}|}{s_0}
+C\frac{\widehat{A}^2}{s_0^{\frac{3}{2}}}|\widetilde{\tau}|+C\frac{\widehat{A}^2}{s_0^{\frac{3}{2}}}|\widetilde{\alpha}|^2+C\frac{{|\widetilde{\alpha}|}^2}{s_0}
+C|\widetilde{\alpha}|e^{-Cs_0}+Cs_0e^{-Cs_0}
,\end{align}
\begin{align}&\Big|\frac{\partial \widetilde{\psi}_1}{\partial\widetilde{\tau}}\Big|=\Big|\frac{\partial( \sigma_0^2\overline{\psi}_1)}{\partial{\tau}}\frac{\partial{\tau}}{\partial\widetilde{\tau}}\Big|=\Big|\frac{\partial( \sigma_0^2 \overline{\psi}_1)}{\partial{\tau}}\frac{1}{s_0^2}\Big|
\nonumber\\ \leq& C\|\widetilde{\varepsilon}_0\|_{
{W_0^{1,q}(\Omega)\cap W^{2,q}(\Omega)}}+C\frac{\widehat{A}}{s_0^2}+C\frac{\widehat{A}^2}{s_0^{\frac{1}{2}}}\Big(\frac{1}{s_0}+\frac{|\widetilde{\tau}|}{s_0^{2}}+\frac{|\widetilde{\alpha}|}{s_0}\Big)+C\Big(\frac{1}{s_0}+
\frac{|\widetilde{\tau}|}{s_0^2}
\nonumber\\&
+\frac{|\widetilde{\alpha}|}{s_0}\Big)+Ce^{-Cs_0}
+C\frac{|\widetilde{\alpha}|}{s_0}e^{-Cs_0}
,\end{align}
and
\begin{align}
&\Big|\frac{\partial \widetilde{\psi}_1}{\partial\widetilde{\alpha}}+\frac{2b\kappa}{(p-1)^2}\Big|=\Big|\frac{\partial( \overline{\psi}_1)}{\partial{\alpha}}\frac{\sigma_0^2}{s_0}+\frac{2b\kappa}{(p-1)^2}\Big|\nonumber\\=&\Big|\frac{\partial( \overline{\psi}_1)}{\partial{\alpha}}\frac{\sigma_0^2}{s_0}+\frac{1}{s_0}\frac{\sigma_0^2}{s_0}\frac{2b\kappa}{(p-1)^2}
-\frac{\sigma_0^2}{s_0^2}\frac{2b\kappa}{(p-1)^2}+\frac{2b\kappa}{(p-1)^2}\Big|
 \nonumber\\\leq& Cs_0\|\widetilde{\varepsilon}_0\|_{
{W_0^{1,q}(\Omega)\cap W^{2,q}(\Omega)}}+C\frac{\widehat{A}^2}{s_0}\log s_0+C\frac{\widehat{A}}{s_0^2}|\widetilde{\alpha}|+Cs_0^{\frac{1}{2}}\widehat{A}^2e^{-s_0}
+C\frac{1}{s_0}+C\frac{|\widetilde{\tau}|}{s_0}\nonumber\\&
+C\frac{\widehat{A}^2}{s_0^{\frac{3}{2}}}|\widetilde{\tau}|+C\frac{\widehat{A}^2}{s_0^{\frac{3}{2}}}|\widetilde{\alpha}|^2+C\frac{|\widetilde{\alpha}|^2}{s_0}
+C|\widetilde{\alpha}|e^{-Cs_0}+Cs_0e^{-Cs_0}+C\frac{|\widetilde{\tau}|}{s_0^3}
.\end{align}

Consider now  $s_0$ large, and $\widetilde{\varepsilon}_0$ such that
\begin{align}\label{estimate}\|\widetilde{\varepsilon}_0\|_{
{W_0^{1,q}(\Omega)\cap W^{2,q}(\Omega)}}\leq \widehat{\varepsilon}_0(s_0):=\frac{1}{s_0^3\widetilde{C}_0(s_0)}.\end{align}
From the above-mentioned expansions, we see that for $s_0$ large enough, the function
$$(\widetilde{\tau},\widetilde{\alpha})\rightarrow (\widetilde{\psi}_0,\widetilde{\psi}_1)(s_0,\widetilde{\varepsilon}_0,\widetilde{\tau},\widetilde{\alpha})$$
is $C^1$-diffemorphism from the rectangle
$$\mathcal{R}_{A}=\Big[-\frac{2(p-1)A}{\kappa},\frac{2(p-1)A}{\kappa}\Big]\times \Big[-\frac{(p-1)^2A}{b\kappa},\frac{(p-1)^2A}{b\kappa}\Big]$$
($\forall (\widetilde{\tau},\widetilde{\alpha})\in \mathcal{R}_{A}$, and $\mbox{Jac}_{\widetilde{\tau},\widetilde{\alpha}}(s_0,\varepsilon_0,\widetilde{\psi}_0,\widetilde{\psi}_1)\neq 0$,)
onto a set $\mathcal{E}_{A,s_0,y_0}$ which approaches (in some appropriate sense) from the rectangle
$$[s_0^2\widehat{q}_0(s_0)-2A,s_0^2\widehat{q}_0(s_0)+2A]\times[s_0^2\widehat{q}_1(s_0)-2A,s_0^2\widehat{q}_1(s_0)+2A]$$
as $s_0\rightarrow\infty$, where $A>0$ large enough will be fixed later.

\medskip

Since $\widehat{q}(s_0)\in V_{{K}_0,\widehat{A}}(s_0)$, we have
$$|s_0^2\widehat{q}_0(s_0)|\leq \widehat{A},\ |s_0^2\widehat{q}_1(s_0)|\leq \widehat{A},$$
we clearly see that
$$[-A,A]^2\subset \mathcal{E}_{A,s_0,y_0}$$
for $A\geq 2\widehat{A}$ and $s_0$ large enough, hence, there exists a set $\overline{D}_{A,s_0,y_0}\subset \mathcal{R}_{A}$ such that
$$(\widetilde{\psi}_0,\widetilde{\psi}_1)(s_0,\varepsilon_0,\widetilde{D}_{A,s_0,y_0})=[-A,A]^2.$$

Moreover, from (\ref{y1}) and (\ref{y2}), the function $(\widetilde{\psi}_0,\widetilde{\psi}_1)$ has degree -1 on the boundary of
$\overline{D}_{A,s_0,y_0}$. By transformation $(\ref{transformation})$, the conclusion of item (i) holds.

\medskip

(ii) Take $(T,a)\in\overline{D}_{A,s_0,y_0}$ and let us check that $\overline{\psi}(s_0,y_0,T,a)\in V_{{K}_0,A}(\sigma_0)$.
Note that $\widetilde{\tau}=s_0^2\tau$, $\widetilde{\alpha}=s_0\alpha$, $\tau=(T-\widehat{T})e^{s_0}$, $\alpha=(a-\widehat{a})e^{s_0/2}$ and
\begin{align}\label{y3}\overline{D}_{A,s_0,y_0}\subset \Big\{(T,a)\Big|\ |T-\widehat{T}|\leq \frac{2e^{-s_0}(p-1)A}{\kappa s_0^2},\ |a-\widehat{a}|\leq
\frac{e^{-s_0/2}(p-1)^2A}{b\kappa s_0}\Big\}.\end{align}
From item (i), we know by construction that
$$|\overline{\psi}_m|\leq A\sigma_0^{-2},\ m=0,1.$$
For the other estimates to be checked, note first from (\ref{y3}), and the definition (\ref{dddd}) of $(\tau,\alpha)$ that
$$ |\tau|\leq \frac{2(p-1)A}{\kappa s_0^2},\ |\alpha|\leq \frac{(p-1)^2A}{b\kappa s_0}.$$
Therefore, using (\ref{estimate}) and then items (iii), (viii) and (ix) of Lemma  \ref{lem 2.2}, we see that for $s_0$ large, we have
$$|\overline{\psi}_2|< C\widehat{A}^2\sigma_0^{-2}\log \sigma_0,$$
\begin{align}&\left\|\frac{\overline{\psi}_{-}(z,\sigma_0)}{1+|z|^3}\right\|_{L^{\infty}(\mathbb{R})}
<\frac{C\widehat{A}}{\sigma_0^2},\end{align}
and
\begin{align}&\Big\|{\overline{\psi}_{e}(z,\sigma_0)}\Big\|_{L^{\infty}(\mathbb{R})}<\frac{C\widehat{A}^2}{\sigma_0^{\frac{1}{2}}}.\end{align}
If $A\geq \overline{C}\widehat{A}$ for some large $\overline{C}>0$ and $s_0\geq s_0(A)$, then we see that the conclusion follows.

\medskip

(iii) Take $(T,a)\in\overline{D}_{A,s_0,y_0}$ with $A\geq \overline{C}\widehat{A}$. Recall that  for all $x\in\Omega$ with $|x-\widehat{a}|\geq \varepsilon_0/2$, $|\widehat{y}(x,t)|\leq\widehat{\eta}_0$ for any $t\in[0,\widehat{T})$. Then, for any $s_0\geq \overline{s}_0(A)$ but fixed,
there exists a constant $C_3>0$ such that $\|\widehat{y}(\widehat{T}-e^{-s_0})\|_{L^\infty(\Omega)}\leq C_3$. By system (\ref{profile1}), we have that if $\|\widetilde{\varepsilon}_0\|_{W_0^{1,q}(\Omega)\cap W^{2,q}(\Omega)}$ is small enough, then
$\|y_{y_0}(t)\|_{L^\infty(\Omega)}\leq C_4$ for some $C_4>0$, where $t=\widehat{T}-e^{-s_0}=T-e^{-\sigma_0}$, and
$\|y_{y_0}(t)-\widehat{y}(t)\|_{L^\infty(\Omega)}\leq C_5(s_0)\|\widetilde{\varepsilon}_0\|_{W_0^{1,q}(\Omega)\cap W^{2,q}(\Omega)}.$
Thus, if $C_5(s_0)\|\widetilde{\varepsilon}_0\|_{W_0^{1,q}(\Omega)\cap W^{2,q}(\Omega)}\leq\widehat{\eta}_0$, then for all $x\in\Omega$
 with $|x-a|\geq 3{\varepsilon}_0/4$ and when $s_0$ is large enough such that  $|a-\widehat{a}|=|\alpha e^{-\frac{s_0}{2}}|
\leq e^{-\frac{s_0}{2}}<\frac{\varepsilon_0}{4}$, which implies
 $|x-\widehat{a}|\geq |x-a|-|a-\widehat{a}|>{\varepsilon}_0/2$, we have $|{y}_{y_0}(x,T-e^{-\sigma_0})|\leq2\widehat{\eta}_0$.\end{proof}

{\bf Proof of Theorem \ref{Pros1.2}.} By Proposition {\ref{prop3.1}} and lemmas in Subsection 2.1, we can use the very similar techniques
we used in the proof of Proposition \ref{lemma1.3} in Subsection 2.2 to complete the proof of Theorem \ref{Pros1.2}.

\section{Proof of Theorem \ref{Main Tho}}
In this section, we will use Theorem \ref{Pros1.2} and  the feedback null controllability
 results for linear heat equations obtained in \cite{Sirbu} to prove Theorem \ref{Main Tho}.

\begin{proof}
Consider $p>1$. By Proposition \ref{lemma1.3}, we  conclude  that for  any $a\in \omega$, there exists
  $T_0>0$ such that for any $T\in (0,T_0)$, there exists some initial
  data $\widetilde{y}_0\in C_0^\infty(\omega)$ such that the corresponding solution
$y$  to  (\ref{profile1}) exists on $[0,T)$, $T$ is the blowup time of $y$ and $y$
has a unique blowup point $a$.

Then,  given  $(a,T)\in \omega\times(0,+\infty)$,  we take a time
$T_1\in(0, \min\{T_0,T/2\})$. Applying Proposition \ref{lemma1.3}, we know that there exists initial data $\widetilde{y}_0\in C_0^\infty(\omega)$ such that the solution $y$ to  (\ref{profile1}) blows up at time $T_1$ with $a$ being its unique
blowup point. Introducing $$\mathcal{Y}^*(x,t):=y(x, t-T+T_1), \ (x,t)\in \Omega
\times [T-T_1, T),$$
we see that $\mathcal{Y}^*$ satisfies the following equation:
\begin{eqnarray}\nonumber
\left\{\begin{array}{ll} \mathcal{Y}^*_t-
\Delta \mathcal{Y}^*=\mathds{1}_\omega |\mathcal{Y}^*|^{p-1}\mathcal{Y}^*,&(x,t)\in \Omega\times (T-T_1, T),\\
\mathcal{Y}^*=0,& (x,t)\in \partial\Omega\times(T-T_1,T),\\
\displaystyle  \mathcal{Y}^*(x,T-T_1)=\widetilde{y}_0(x),&
x\in \Omega.
\end{array}\right.
\end{eqnarray}
Moreover, $\mathcal{Y}^*$ blows up at time $T$, only at the blowup point $a$.

\medskip

On the other hand, given $y_0\in H_0^1(\Omega)$,  take
the following system in consideration:
\begin{eqnarray}\label{131}
\left\{\begin{array}{ll} \mathcal{Z}_t-
\Delta \mathcal{Z}=\mathds{1}_\omega v,&(x,t)\in \Omega\times  (0,T-T_1),\\
\mathcal{Z}=0,&(x,t)\in \partial\Omega\times  (0,T-T_1),\\
\displaystyle  \mathcal{Z}(x,0)=y_0(x)-\widetilde{y}_0(x),&x\in \Omega,
\end{array}\right.
\end{eqnarray}
where $v\in L^2(0,T-T_1;H)$. Since $y_0-\widetilde{y}_0\in H_0^1(\Omega)$, it is well known that for each $v\in L^2(0,T-T_1;H)$, there exists a unique  solution $\mathcal{Z}$ in $C([0,T-T_1];H_0^1(\Omega))$ to (\ref{131}).

\medskip

Consider now the following optimal control problem:
$$(\mathcal{P})\ \min\Big\{\int_0^{T-T_1}\|v(t)\|_{H}^2dt;\ \mathcal{Z}\ \mbox{satisfies}\ (\ref{131}),\ \mathcal{Z}(T-T_1)=0\Big\},$$
together with the Riccati system (\ref{xe1.2}). By Theorem 2.1 in \cite{Sirbu} and its proof, there exists a unique mild solution $P\in C_S([0,T-T_1); \Sigma^+(H))$ to problem
(\ref{xe1.2}).
Moreover, $v(t)=-\mathds{1}_\omega^*P(t)\mathcal{Z}(t)$ defined for all $t\in [0,T-T_1)$ is the optimal feedback control for problem ($\mathcal{P}$).

\medskip

Set now for all  $t\in [0,T-T_1)$,
$\check{y}(t):=\mathcal{Z}(t)+\widetilde{y}_0(x)$ and $$u_1(t):=v(t)-\Delta \widetilde{y}_0(x)=-\mathds{1}_\omega^*P(t)(\check{y}(t)-\widetilde{y}_0(x))-
\Delta \widetilde{y}_0(x).$$  It
holds by the construction of $\widetilde{y}_0(x)$ that
\begin{eqnarray}\nonumber\Delta \widetilde{y}_0(x)=\mathds{1}_\omega \Delta \widetilde{y}_0(x).\end{eqnarray}
Then,  it follows that $\check{y}$ is the solution to the following equation:
\begin{eqnarray}\label{review1}
\left\{\begin{array}{ll} \check{y}_t-
\Delta \check{y}=\mathds{1}_\omega u_1,&x\in \Omega\times  (0,T-T_1),\\
\check{y}=0,& x\in \partial\Omega\times  (0,T-T_1),\\
\displaystyle  \check{y}(x,0)=y_0(x),&
x\in \Omega,
\end{array}\right.
\end{eqnarray}
with \begin{eqnarray}\nonumber\check{y}(T-T_1)=\widetilde{y}_0(x).\end{eqnarray}

Set
$$y(x,t):=\left\{\begin{array}{ll}\check{y}(x,t),\ &(x,t)\in \Omega\times  [0,T-T_1), \\
\mathcal{Y}^*(x,t), \ &(x,t)\in \Omega\times [T-T_1,T),
\end{array}\right.$$
and
$$u(x,t):=\left\{\begin{array}{ll}-\mathds{1}_\omega^*P(t)({y}(t)-\widetilde{y}_0)(x)-\Delta \widetilde{y}_0(x),\ &(x,t)\in \Omega\times  (0,T-T_1), \\
|y|^{p-1}y(x,t), \ &(x,t)\in \Omega\times [T-T_1,T).
\end{array}\right.$$
Then, $y$ is the solution to (\ref{xe1.1})
with  the   feedback control $u$, and  it follows that
 $y$ blows up in $T$ and has $a$ as a unique blowup point.

 \medskip

 Consider the following two equations,
 \begin{eqnarray}\label{review3}
\left\{\begin{array}{ll} \widetilde{\Phi}_t-
\Delta \widetilde{\Phi}=0,&x\in \Omega,\  t>0,\\
\widetilde{\Phi}=0,& x\in \partial\Omega,\   t>0,\\
\displaystyle  \widetilde{\Phi}(x,0)=\widetilde{\Phi}_0(x),&
x\in \Omega,
\end{array}\right.
\end{eqnarray}
and
 \begin{eqnarray}\label{review4}
\left\{\begin{array}{ll} \phi_t-
\Delta \phi=0,&x\in \Omega,\  t>0,\\
\phi=0,& x\in \partial\Omega,\   t>0,\\
\displaystyle  \phi(x,0)=\widetilde{y}_0(x),&
x\in \Omega.
\end{array}\right.
\end{eqnarray}

It is well known that if $\widetilde{\Phi}_0\in W_0^{1,q}(\Omega)\cap W^{2,q}(\Omega)$ with $q>n+2$, then the solution $\widetilde{\Phi}\in C([0,T];W^{2,q}(\Omega))\cap W_0^{1,q}(\Omega))\cap C^1([0,T];L^q(\Omega))$ for any $T>0$. By Sobolev embedding theorem, we have
$W_0^{1,q}(\Omega)\cap W^{2,q}(\Omega)\hookrightarrow C^\alpha (\overline{\Omega})$ and $C([0,T];W^{2,q}(\Omega))\cap W_0^{1,q}(\Omega))\cap C^1([0,T];L^q(\Omega))\hookrightarrow C^{\beta,\beta/2} (\overline{\Omega\times(0,T)})$ for some $0<\alpha<1$ and $0<\beta<1$, where $q>n+2$.

\medskip

If $\widetilde{\Phi}_0\in H_0^1(\Omega)$, then for any $t>0$,
\begin{eqnarray}\label{review5}
\Big\|\widetilde{\Phi}(t)-\phi(t)\Big\|_{L^\infty(\Omega)}\leq Ct^{-\frac{n}{4}}\|\widetilde{\Phi}_0-\widetilde{y}_0\|_{L^2(\Omega)}\leq Ct^{-\frac{n}{4}}\|\widetilde{\Phi}_0-\widetilde{y}_0\|_{H_0^1(\Omega)}.
\end{eqnarray}
Thus, for $1<q<\infty$,
\begin{eqnarray}\label{review5}
\|\widetilde{\Phi}(t)-\phi(t)\|_{W^{2,q}(\Omega)}\leq C\Big(1+\frac{2}{t}\Big)\Big\|\widetilde{\Phi}\Big(\frac{t}{2}\Big)-\phi\Big(\frac{t}{2}\Big)\Big\|_{L^q(\Omega)}\\ \leq C\Big(1+\frac{2}{t}\Big)\Big(\frac{t}{2}\Big)^{-\frac{n}{4}}\|\widetilde{\Phi}_0-\widetilde{y}_0\|_{H_0^1(\Omega)}.
\end{eqnarray}
Hence
\begin{eqnarray}\label{review8}
\|\widetilde{\Phi}(t)-\phi(t)\|_{L^\infty(\Omega)\cap W^{2,q}(\Omega)}\leq \widetilde{C}\Big(t^{-\frac{n}{4}}+\Big(1+\frac{2}{t}\Big)\Big(\frac{t}{2}\Big)^{-\frac{n}{4}}\Big)\|\widetilde{\Phi}_0-\widetilde{y}_0\|_{H_0^1(\Omega)},
\end{eqnarray}
for some $\widetilde{C}>0$.

 Let $\varepsilon>0$ and let  $\varepsilon_1>0$ be given in Theorem \ref{Pros1.2} for $\varepsilon/2$.  Since $\widetilde{y}_0\in C_0^\infty(\omega)$, we have that there exists a time $\widehat{\varepsilon}_1\in (0, \min(\frac{\varepsilon}{4}, \frac{T_1}{4}))$ such that
\begin{eqnarray}\label{review6}
\|\phi(\widehat{\varepsilon}_1)-\widetilde{y}_0\|_{L^\infty(\Omega)\cap W^{2,q}(\Omega)}< \frac{\varepsilon_1}{2}.
\end{eqnarray}

Since $\check{y}\in C([0,T-T_1];H_0^1(\Omega))$,  there exist ${\delta}\in(0,\min\{\varepsilon/4,T-T_1\})$, $\delta^*>0$ with
$({\delta}-\delta^*,{\delta}+\delta^*)\subset(0,\min\{\varepsilon/4,T-T_1\})$
  such that for any $\widehat{\varepsilon}\in ({\delta}-\delta^*,{\delta}+\delta^*)
  $,
 \begin{eqnarray}\label{checky}\|\check{y}(T-T_1-\widehat{\varepsilon})-\widetilde{y}_0\|_{H_0^1(\Omega)}
<\frac{1}{\widetilde{C}\Big(\widehat{\varepsilon}_1^{-\frac{n}{4}}+\Big(1+\frac{2}{\widehat{\varepsilon}_1}\Big)
\Big(\frac{\widehat{\varepsilon}_1}{2}\Big)^{-\frac{n}{4}}\Big)}\frac{\varepsilon_1}{4}.\end{eqnarray}

Since $P\in C_S([0,T-T_1-{\delta}+\delta^*]; \Sigma^+(H))$, there exists $\varepsilon^*>0$ such that
for any initial data $y_0^*$ satisfying $\|y_0^*-y_0\|_{H_0^1(\Omega)}<\varepsilon^*$ and any
$\widehat{\varepsilon}\in ({\delta}-\delta^*,{\delta}+\delta^*)$, the  solution
$y^*$ to  (\ref{xe1.1}) with initial data $y_0^*$ and with the following feedback control
$$u(x,t):=-\mathds{1}_\omega^*P(t)({y^*}(t)-\widetilde{y}_0)(x)-\Delta \widetilde{y}_0(x),\ (x,t)\in \Omega\times  (0,T-T_1-{\delta}+\delta^*)$$
satisfies
\begin{eqnarray}\label{review10}\|{y^*}(T-T_1-\widehat{\varepsilon})-\check{y}(T-T_1-\widehat{\varepsilon})\|_{H_0^1(\Omega)}<
\frac{1}{\widetilde{C}\Big(\widehat{\varepsilon}_1^{-\frac{n}{4}}+\Big(1+\frac{2}{\widehat{\varepsilon}_1}\Big)\Big(\frac{\widehat{\varepsilon}_1}{2}\Big)^{-\frac{n}{4}}\Big)}\frac{\varepsilon_1}{4},\end{eqnarray}
from which and (\ref{checky}), it holds  that
 \begin{eqnarray}\label{checky1}\|{y^*}(T-T_1-\widehat{\varepsilon})-\widetilde{y}_0\|_{H_0^1(\Omega)}<\frac{1}{\widetilde{C}\Big(\widehat{\varepsilon}_1^{-\frac{n}{4}}+\Big(1+\frac{2}{\widehat{\varepsilon}_1}\Big)
 \Big(\frac{\widehat{\varepsilon}_1}{2}\Big)^{-\frac{n}{4}}\Big)}\frac{\varepsilon_1}{2}.\end{eqnarray}

 Further, we consider the following equation,
 \begin{eqnarray}\nonumber
\left\{\begin{array}{ll} Y_t-
\Delta Y=0,& (x,t)\in \Omega\times(T-T_1-\widehat{\varepsilon}, T-T_1-\widehat{\varepsilon}+\widehat{\varepsilon}_1),\\
Y=0,& (x,t)\in \partial\Omega\times(T-T_1-\widehat{\varepsilon}, T-T_1-\widehat{\varepsilon}+\widehat{\varepsilon}_1),\\
\displaystyle  Y(x,T-T_1-\widehat{\varepsilon})=y^*(T-T_1-\widehat{\varepsilon})(x),&
x\in \Omega,
\end{array}\right.
\end{eqnarray}
and
 \begin{eqnarray}\nonumber
\left\{\begin{array}{ll} \widetilde{\phi}_t-
\Delta \widetilde{\phi}=0,&  (x,t)\in \Omega\times(T-T_1-\widehat{\varepsilon}, T-T_1-\widehat{\varepsilon}+\widehat{\varepsilon}_1),\\
\widetilde{\phi}=0,& (x,t)\in \partial\Omega\times(T-T_1-\widehat{\varepsilon}, T-T_1-\widehat{\varepsilon}+\widehat{\varepsilon}_1),\\
\displaystyle  \widetilde{\phi}(x,T-T_1-\widehat{\varepsilon})=\widetilde{y}_0(x),&
x\in \Omega.
\end{array}\right.
\end{eqnarray}
By (\ref{review8}) and (\ref{review6}), it holds that
\begin{align}\label{review7}
&\|Y(T-T_1-\widehat{\varepsilon}+\widehat{\varepsilon}_1)-\widetilde{\phi}(T-T_1-\widehat{\varepsilon}+\widehat{\varepsilon}_1)\|_{L^\infty(\Omega)\cap W^{2,q}(\Omega)}\nonumber\\ \leq &\widetilde{C}\Big(\widehat{\varepsilon}_1^{-\frac{n}{4}}+\Big(1+\frac{2}{\widehat{\varepsilon}_1}\Big)\Big(\frac{\widehat{\varepsilon}_1}{2}\Big)^{-\frac{n}{4}}\Big)
\|y^*(T-T_1-\widehat{\varepsilon})-\widetilde{y}_0\|_{H_0^1(\Omega)}
\end{align}
and
\begin{eqnarray}\label{review9}
\|\widetilde{\phi}(T-T_1-\widehat{\varepsilon}+\widehat{\varepsilon}_1)-\widetilde{y}_0\|_{L^\infty(\Omega)\cap W^{2,q}(\Omega)}< \frac{\varepsilon_1}{2},
\end{eqnarray}
from which, (\ref{review7}) and (\ref{checky1}) we have
\begin{align}\label{review10}
&\|Y(T-T_1-\widehat{\varepsilon}+\widehat{\varepsilon}_1)-\widetilde{y}_0\|_{L^\infty(\Omega)\cap W^{2,q}(\Omega)\cap W_0^{1,q}(\Omega)} < \varepsilon_1.
\end{align}

Then by Theorem \ref{Pros1.2}, the solution $\overline{Y}$
to the following equation
\begin{eqnarray}\nonumber
\left\{\begin{array}{ll} \overline{Y}_t-
\Delta \overline{Y}=\mathds{1}_\omega |\overline{Y}|^{p-1}\overline{Y},&x\in \Omega,\  t>T-T_1-\widehat{\varepsilon}+\widehat{\varepsilon}_1,\\
\overline{Y}=0,& x\in \partial\Omega,\ t>T-T_1-\widehat{\varepsilon}+\widehat{\varepsilon}_1,\\
\displaystyle  \overline{Y}(x,T-T_1-\widehat{\varepsilon}+\widehat{\varepsilon}_1)=\overline{Y}(T-T_1-\widehat{\varepsilon}+\widehat{\varepsilon}_1)(x),&
x\in \Omega,
\end{array}\right.
\end{eqnarray}
blows up at time $T^*$, only at the blowup point $a^*$,
$|T^*-(T-T_1-\widehat{\varepsilon}+\widehat{\varepsilon}_1)-T_1|<\varepsilon/2$ and $|a^*-a|<\varepsilon/2$.
Since $\widehat{\varepsilon}<\varepsilon/4$ and $\widehat{\varepsilon}_1<\varepsilon/4$, we have $|T^*-T|<\varepsilon$.

Moreover, for all $R>0$,
 \begin{eqnarray}\nonumber
\sup\limits_{\big\{|x-a^*|\leq R
\sqrt{(T^*-t)|\log(T^*-t)|}\big\}}\Big|(T^*-t)^{\frac{1}{p-1}}\overline{Y}(x,t)-f\Big(\frac{x-a^*}{\sqrt{(T^*-t)
|\log(T^*-t)|}}\Big)\Big|\rightarrow 0
\end{eqnarray}
as $t\rightarrow T^*$, where
$f$ is defined in (\ref{fnew}).

Set
$$y(x,t):=\left\{\begin{array}{ll}y^*(x,t),\ &(x,t)\in \Omega\times  [0,T-T_1-\widehat{\varepsilon}), \\
Y(x,t),\ &(x,t)\in \Omega\times  [T-T_1-\widehat{\varepsilon},T-T_1-\widehat{\varepsilon}+\widehat{\varepsilon}_1),\\
\overline{Y}(x,t), \ &(x,t)\in \Omega\times [T-T_1-\widehat{\varepsilon}+\widehat{\varepsilon}_1,T^*),
\end{array}\right.$$
and
$$u(x,t):=\left\{\begin{array}{ll}-\mathds{1}_\omega^*P(t)({y}(t)-\widetilde{y}_0)(x)-\Delta \widetilde{y}_0(x),\ (x,t)\in \Omega\times  [0,T-T_1-\widehat{\varepsilon}),
\\
0,\ \ \ \ \ \ \ \ \ \ \ \ \ \ \ \ \ \ \ \ \ \ \ \ \ \ \ \ \ (x,t)\in \Omega\times  [T-T_1-\widehat{\varepsilon},T-T_1-\widehat{\varepsilon}+\widehat{\varepsilon}_1),\\
|y|^{p-1}y(x,t), \ \ \ \ \ \ \ \ \ \ (x,t)\in \Omega\times [T-T_1-\widehat{\varepsilon}+\widehat{\varepsilon}_1
,T^*).
\end{array}\right.$$

Then from the above argument, Theorem \ref{Main Tho} holds.
\end{proof}

\end{document}